\documentclass[twoside,reqno,12pt]{amsart}
\usepackage{a4wide} 
\usepackage{amsmath,amsfonts,amssymb}
\usepackage{times}
\usepackage{hyperref}
\usepackage[all,rotate]{xy}
\usepackage{harvard}
\usepackage{verbatim}
\usepackage{tikz}
\usetikzlibrary{arrows,backgrounds}
\usetikzlibrary{decorations.pathreplacing}
\usetikzlibrary{decorations.markings}
\newenvironment{pic}[1][]%
{\begin{aligned}\begin{tikzpicture}[#1]}%
{\end{tikzpicture}\end{aligned}}
\tikzstyle{string}=[line width=1.25pt]
\tikzstyle{thickstring}=[line width=2.5pt]
\tikzstyle{dot}=[circle, draw=black, fill=gray, inner sep=.4ex, line width=1.25pt]
\tikzstyle{whitedot}=[circle, draw=black, fill=white, inner sep=.4ex, line width=1.25pt]
\tikzstyle{blackdot}=[circle, draw=black, fill=black, inner sep=.4ex, line width=1.25pt]
\tikzstyle{cross}=[preaction={draw=white, -, line width=6pt}]
\tikzstyle{coupon}=[rectangle,draw=black,minimum size=10pt,line width=1.25pt]
\tikzstyle{unit}=[triangle,line width=1.25pt,fill=white,draw,inner sep=1pt,minimum width=1.1cm]
\tikzstyle{counit}=[triangle,hflip,line width=1.25pt,fill=white,draw,inner sep=1pt,minimum width=1.1cm]
\tikzstyle{iso2cell}=[iso,line width=1.25pt,fill=white,draw,inner sep=1pt,minimum width=1.1cm]
\tikzstyle{pullback2cell}=[iso,pullback,line width=1.25pt,fill=white,draw,inner sep=1pt,minimum width=1.1cm]
\tikzstyle{2cell}=[rectangle,line width=1.25pt,fill=white,draw,inner sep=1pt,minimum size=10pt,minimum width=1.1cm]

\tikzset{arrow/.style={decoration={
    markings,
    mark=at position #1 with \arrow{>}},
    postaction=decorate}
}
\tikzset{reverse arrow/.style={decoration={
    markings,
    mark=at position #1 with \arrow{<}},
    postaction=decorate}
}
\makeatletter
\newif\ifpullback
\pgfkeys{/tikz/pullback/.is if=pullback}
\pgfkeys{/tikz/pullback=false}
\pgfdeclareshape{iso}{%
    \inheritsavedanchors[from=rectangle]
    \inheritanchorborder[from=rectangle]
    \inheritanchor[from=rectangle]{north}
    \inheritanchor[from=rectangle]{north west}
    \inheritanchor[from=rectangle]{north east}
    \inheritanchor[from=rectangle]{center}
    \inheritanchor[from=rectangle]{west}
    \inheritanchor[from=rectangle]{east}
    \inheritanchor[from=rectangle]{mid}
    \inheritanchor[from=rectangle]{mid west}
    \inheritanchor[from=rectangle]{mid east}
    \inheritanchor[from=rectangle]{base}
    \inheritanchor[from=rectangle]{base west}
    \inheritanchor[from=rectangle]{base east}
    \inheritanchor[from=rectangle]{south}
    \inheritanchor[from=rectangle]{south west}
    \inheritanchor[from=rectangle]{south east}
    \savedanchor\centerpoint
    {
        \pgf@x=0pt
        \pgf@y=0pt
    }
    \anchor{text}{%
         \pgf@x=0pt
         \pgf@y=0pt
         \advance\pgf@y by -0.2\ht\pgfnodeparttextbox
    }
    \backgroundpath{%
      \northeast \pgf@xa=\pgf@x \pgf@ya=\pgf@y
      \southwest \pgf@xb=\pgf@x \pgf@yb=\pgf@y
      \begingroup
         \tikz@mode
         \iftikz@mode@draw
            \pgfpathmoveto\northeast
            \pgfpathlineto{\pgfqpoint{\pgf@xb}{\pgf@ya}}
            \pgfpathmoveto\southwest
            \pgfpathlineto{\pgfqpoint{\pgf@xa}{\pgf@yb}}
            \pgfusepath{stroke}
         \fi
         \iftikz@mode@fill
            \pgfsetlinewidth{\pgflinewidth}
            \advance\pgf@ya by -0.5\pgflinewidth%
            \advance\pgf@yb by  0.5\pgflinewidth%
            \pgfpathrectanglecorners{\pgfqpoint{\pgf@xb}{\pgf@yb}}{\pgfqpoint{\pgf@xa}{\pgf@ya}}
            \pgfusepath{fill}
         \fi
         \ifpullback
            \northeast \pgf@xa=\pgf@x \pgf@ya=\pgf@y
            \southwest \pgf@xb=\pgf@x \pgf@yb=\pgf@y
            \pgf@xc=\pgf@xb
            \advance\pgf@xc by  3\pgflinewidth
            \pgf@yc=\pgf@ya
            \advance\pgf@yc by -2\pgflinewidth
            \pgfpathmoveto{\pgfqpoint{\pgf@xc}{\pgf@ya}}
            \pgfpathlineto{\pgfqpoint{\pgf@xc}{\pgf@yc}}
            \pgfpathlineto{\pgfqpoint{\pgf@xb}{\pgf@yc}}
            \pgfusepath{stroke}
         \fi
    \endgroup
    }
}
\newif\ifhflip
\pgfkeys{/tikz/hflip/.is if=hflip}
\pgfkeys{/tikz/hflip=false}
\pgfdeclareshape{triangle}{%
     \saveddimen{\halfbaselength}{%
         \pgf@x=0.5\wd\pgfnodeparttextbox
         \pgfmathsetlength\pgf@xc{\pgfkeysvalueof{/pgf/inner xsep}}%
         \advance\pgf@x by \pgf@xc%
         \advance\pgf@x by \ht\pgfnodeparttextbox
         \pgfmathsetlength\pgf@xb{\pgfkeysvalueof{/pgf/minimum width}}%
         \divide\pgf@xb by 2
         \ifdim\pgf@x<\pgf@xb%
             \pgf@x=\pgf@xb%
         \fi%
     }
     \saveddimen\triangleheight{%
         \pgfmathsetlength\pgf@xc{\pgfkeysvalueof{/pgf/inner ysep}}%
         \advance\pgf@x by \pgf@xc%
         \pgfmathsetlength\pgf@xb{\pgfkeysvalueof{/pgf/minimum height}}%
         \divide\pgf@xb by 2
         \ifdim\pgf@x<\pgf@xb%
             \pgf@x=\pgf@xb%
         \fi%
     }
     \savedanchor\centerpoint{
         \pgf@x=0pt 
         \pgf@y=0pt
     }
     \anchor{center}{\centerpoint}
     \anchor{text}{%
         \pgf@x=-0.5\wd\pgfnodeparttextbox
         \ifhflip
            \pgf@y=-1.2\ht\pgfnodeparttextbox
            \advance\pgf@y by \dp\pgfnodeparttextbox
            \advance\pgf@y by -3pt
         \else
            \pgf@y=\dp\pgfnodeparttextbox
            \advance\pgf@y by -\dp\pgfnodeparttextbox
            \advance\pgf@y by 4pt
         \fi
     }
     \anchor{left}{%
        \pgf@x=-\halfbaselength
        \pgf@y=0pt
     }
     \anchor{right}{%
        \pgf@x=\halfbaselength
        \pgf@y=0pt
     }
     \anchor{tip}{%
        \pgf@x=0pt
        \ifhflip
           \pgf@y=-\triangleheight
        \else
           \pgf@y=\triangleheight
        \fi 
     }
     \anchor{a}{%
        \pgf@x=-\halfbaselength
        \divide\pgf@x by 2
        \pgf@y=0pt
     }
     \anchor{b}{%
        \pgf@x=\halfbaselength
        \divide\pgf@x by 2
        \pgf@y=0pt
     }
     \backgroundpath{%
 	    \pgf@xa=\halfbaselength
            \pgf@ya=\triangleheight 
 	    \ifhflip
            \pgfpathmoveto{\pgfqpoint{0pt}{-\pgf@ya}}
	        \pgfpathlineto{\pgfqpoint{-\pgf@xa}{0pt}}
	        \pgfpathlineto{\pgfqpoint{\pgf@xa}{0pt}}
	        \pgfpathclose
	    \else 
            \pgfpathmoveto{\pgfqpoint{0pt}{\pgf@ya}}
	        \pgfpathlineto{\pgfqpoint{-\pgf@xa}{0pt}}
	        \pgfpathlineto{\pgfqpoint{\pgf@xa}{0pt}}
	        \pgfpathclose
            \fi
     }
     \anchorborder{%
		\pgfkeysgetvalue{/pgf/shape border rotate}{\rotate}%
		\edef\externalx{\the\pgf@x}%
		\edef\externaly{\the\pgf@y}%
		\centerpoint%
		\pgf@xa\externalx\relax%
		\pgf@ya\externaly\relax%
		\advance\pgf@xa\pgf@x%
		\advance\pgf@ya\pgf@y%
		\edef\externalx{\the\pgf@xa}%
		\edef\externaly{\the\pgf@ya}%
		\pgfmathanglebetweenpoints{\centerpoint}{\pgfqpoint{\externalx}{\externaly}}%
                \pgfmathsubtract@{\pgfmathresult}{\rotate}%
		\ifdim\pgfmathresult pt<0pt\relax%
			\pgfmathadd@{\pgfmathresult}{360}%
		\fi%
		\let\externalangle\pgfmathresult%
        \pgf@xc=-\halfbaselength
        \pgf@yc=0pt
		\pgfmathdivide@{\externalangle}{50}
		\pgfmathparse{\halfbaselength*\pgfmathresult}
		\advance\pgf@xc by \pgfmathresult pt%
		\pgf@y=0pt
		\pgf@x=\pgf@xc
        }
}
\makeatother

\makeatletter
\@addtoreset{equation}{section}
\makeatother
\newtheorem{theorem}{Theorem}

\newtheorem{corollary}[theorem]{Corollary}
\newtheorem{definition}[theorem]{Definition}

\newtheorem{lemma}[theorem]{Lemma}

\newtheorem{proposition}[theorem]{Proposition}

\def\catC{\mathcal{C}}
\def\catD{\mathcal{D}}
\def\Span{\mathcal{SPN}}
\def\Loc{\mathcal{L}\mathrm{oc}}
\def\arrC{{\mathcal{C}^\downarrow}}
\def\dom{\mathop{\mathsf{dom}}}   
\def\cod{\mathop{\mathsf{cod}}}   
\newcommand{\cc}{\mathop{\mathsf{cc}}}
\newcommand{\dzeroK}{d_0}   
\def\Alg{\text{-}\mathrm{Alg}}
\def\monD{\mathfrak{D}}
\def\monL{\mathfrak{L}}
\def\monR{\mathfrak{R}}
\newcommand{\funK}{\mathfrak{K}}
\newcommand{\funI}{\mathfrak{I}}
\newcommand{\psin}[1]{\psi_{#1}}    
\def\catLoc{\mathop{\mathbf{Loc}}}
\def\val{\mathfrak{V}}
\def\valp{\mathfrak{V}^{(1)}}
\def\funP{\mathfrak{P}}
\def\funPP{\mathbb{P}}
\def\funT{\mathfrak{T}}

\def\Id{\operatorname{I\kern-0.2ex d}}

\newcommand{\upstairs}[1]{\overline{#1}}    
\newcommand{\downstairs}[1]{\underline{#1}} 
\newcommand{\monLof}[1]{{\monL_{\downstairs{#1}}\upstairs{#1}}}
\newcommand{\stairs}[1]{{#1}\mathord{\downarrow}}

\begin{document}
\title[]
{Geometric constructions preserve fibrations}
%
\author{Bertfried Fauser}
\author{Steven Vickers}
\address{%
School Of Computer Science\\
The University of Birmingham\\
Edgbaston, Birmingham, B25 2TT}
\email{[b.fauser|s.j.vickers]@cs.bham.ac.uk}
%

\date{\today}

\begin{abstract}
  Let $\catC$ be a representable 2-category,
  and $\funT_\bullet$ a 2-endofunctor of the arrow 2-category $\arrC$ such that
  (i) $\cod\funT_\bullet = \cod$ and
  (ii) $\funT_\bullet$ preserves proneness of morphisms in $\arrC$.
  Then $\funT_\bullet$ preserves fibrations and opfibrations in $\catC$.

  The proof takes Street's characterization of (e.g.) opfibrations as pseudoalgebras
  for 2-monads $\monL_B$ on slice categories $\catC/B$
  and develops it by defining a 2-monad $\monL_\bullet$ on $\arrC$
  that takes change of base into account, and uses known results on the lifting
  of 2-functors to pseudoalgebras.
\end{abstract}

\maketitle
\section{Introduction}\label{sec:introduction}

The results reported here arose out of an analysis \cite{FRV:Born} of existing topos approaches to quantum foundations as either fibrational or opfibrational. There the toposes and their spectral gadgets are interpreted as point-free bundles $p\colon \Sigma\rightarrow B$,
in other words maps in the category $\mathbf{Loc}$ of locales.
Points of the base space $B$ are \emph{contexts}, or classical perspectives on a quantum system,
and each fibre is a classical state space for its context.
Instances of the specialization order in $B$ (representing context refinement) induce
maps between the corresponding fibres,
in either a covariant (opfibrational) or contravariant (fibrational) way.
Moreover, the fibre maps have a canonical property that makes them either opfibrations
or fibrations in the 2-categorical sense of \cite{Street},
when one considers $\mathbf{Loc}$ as a 2-category,
with its order enrichment using the specialization order.

Our techniques in various places make use of features of the so-called ``geometric reasoning'',
and in particular the way it allows a bundle such as $p$ to be defined geometrically
as a transformation of points of $B$ into point-free spaces, the fibres.
Then a typical point $y$ of $\Sigma$ can be expressed as a pair $(x,y')$
where $x$ is a point of $B$ and $y'$ is a point of the fibre over it.
This is brought out in some detail in~\cite{SVW:GelfandSGTGM}.
From this it is immediately evident what is the specialization order within each fibre,
but it is harder to identify the general specialization across fibres in $\Sigma$.

At this point it is very helpful to know if $p$ is a fibration or an opfibration.
Suppose, for instance, it is known that $p$ is an opfibration.
We want to know when $(x_1,y'_1)\sqsubseteq (x_2,y'_2)$.
Certainly we must have $x_1 \sqsubseteq x_2$, since $p$ must, like any continuous map,
preserve specialization.
Hence there is a fibre map $r_{x_1 x_2}\colon \Sigma_{x_1}\rightarrow \Sigma_{x_2}$.
It then turns out
that $(x_1,y'_1)\sqsubseteq (x_2,y'_2)$ iff
$r_{x_1 x_2}(y'_1)\sqsubseteq y'_2$.

For the spectral bundles these facts arose from rather deep topological facts
\cite{Jo:FibPP}:
in the opfibrational case $p$ was a local homeomorphism,
and in the fibrational case $p$ corresponds to an internal
compact regular locale in the topos of sheaves over $B$.
There are some bundle constructions that do not preserve those topological properties but
for which it is still going to be interesting to know whether they result in
fibrations or opfibrations.
An example is the valuation locale
-- see~\cite{CoqSpit:IntVal} for an account of how it relates to the quantum discussion.

The fundamental result that we shall prove is that \emph{any} geometric construction
preserves both opfibrations and fibrations.

In locale theory, the definition of ``geometric construction'' (of locales) is as follows.
\begin{enumerate}
\item
  Suppose we have an endofunctor $T$ of $\mathbf{Loc}$.
  On the face of it, it is given as an endofunctor of the category $\mathbf{Fr}$
  of frames, although we shall move away from that view.
\item
  Suppose also that the definition of $T$ can be applied to internal frames in any
  topos.
  Since~\cite{JoyalTier} internal frames in the topos of sheaves over a locale $B$
  are dual to bundles over $B$,
  it follows that $T$ induces an endofunctor $T_B$ on the slice $\mathbf{Loc}/B$.
\item
  Geometricity is the property that the construction is stable
  (up to isomorphism) under change of base,
  i.e. pullback $f^{\ast} \colon \mathbf{Loc}/B \to \mathbf{Loc}/B'$
  along a map $f \colon B' \to B$.
  This is often proved by showing that $T$ can be described by a construction on
  presentations of frames.
  See, e.g.,~\cite{PPExp}, or~\cite{Vickers:Riesz} for an application to the valuation locale.
\end{enumerate}

These can be expressed in the language of indexed category theory
(see, e.g., \cite[B1]{Elephant1}):
the slice category $\mathbf{Loc}/B$, for a variable locale $B$,
is ``indexed over $\mathbf{Loc}$'',
and that notion includes the reindexing functors $f^{\ast}$.
Moreover, the geometricity of $T$ then amounts to it being an indexed endofunctor
of the indexed category --
this notion includes the $T_B$s' commuting with reindexing up to coherent isomorphisms.

Although this captures the intuitions, the overall structure is complicated
and we shall find it convenient to use the alternative fibrational structure
(see, again, \cite[B1]{Elephant1}).
For this, all the slices are bundled together into a single category,
the arrow category $\mathbf{Loc}^{\downarrow}$,%
\footnote{This is usually written $\mathbf{Loc}^{\rightarrow}$,
  but we use a down arrow to reinforce an idea that the objects are bundles.}
and then the endofunctors $T$ and $T_B$ as given by a single endofunctor
$T_{\bullet}$ on $\mathbf{Loc}^{\downarrow}$.
Indeed, it seems that we have to go to the fibrational setting,
since we also need to use slicewise endofunctors (some monads defined by Street)
that are not indexed.

It is geometricity that allows the $T_B$s to take in change of base.
We can factor an arbitrary bundle morphism via a pullback square --
this is the ``vertical-prone'' factorization for the codomain fibration
$\cod\colon \mathbf{Loc}^{\downarrow}\to \mathbf{Loc}$.
\begin{equation}\label{eq:bundleMorph}
  \xymatrix{
    E \ar[d]^p \ar[r]^g & E' \ar[d]^{p'} \\
    B               \ar[r]^f & B'
  }
  \quad
  \xymatrix{
    E \ar[dr]^p \ar[r]
    & f^{\ast}E' \ar[d] \ar[r]
    & E' \ar[d]^{p'} \\
    & B \ar[r]^f & B'
  }
  \text{.}
\end{equation}
Then geometricity gives us an isomorphism in
\[
  \xymatrix{
    T_{B}E \ar[dr] \ar[r]
    & T_{B} f^{\ast}E' \ar[d] \ar[r]^{\cong}
    & f^{\ast}T_{B'} E' \ar[dl] \ar[r]
    & T_{B'}E' \ar[d]^{p'} \\
    & B  \ar[rr]^f & & B'
  }
  \text{,}
\]
thus enabling us to define
$T_{f}g \colon T_{B}E \to T_{B'}E'$.

The only extra condition needed on $T_{\bullet}$ is that if a bundle morphism
as on the left-hand side of~\eqref{eq:bundleMorph}
happens to be a pullback square, then so is the result of applying $T_{\bullet}$ to it:
\[
  \xymatrix{
    T_{B}(E) \ar[d]^{T_{B}(p)} \ar[r]^{T_{f}(g)} & T_{B'}(E') \ar[d]^{T_{B'}(p')} \\
    B               \ar[r]^f & B'
  }
\]

In~\cite{TownsendVickers:strengthofpowerlocale} are given sufficient conditions for
an endofunctor on $\mathbf{Loc}$ (or, more generally, any cartesian category $\catC$)
to be part of an indexed endofunctor on slices.
This covers coherence of the geometricity isomorphisms, or, equivalently,
the strict functoriality of $T_\bullet$.
The conditions are verified there for the powerlocales,
and in \cite{Vickers:Riesz} for valuation locales.

With this formulation of geometricity,
we are able to generalize from $\mathbf{Loc}$ to a very general class of 2-categories,
and prove (Theorem~\ref{thm:main}) that if $T_{\bullet}$ is geometric,
then for each base object $B$, $T_B$ preserves the properties of being fibration or
opfibration over $B$.

Note that we shall be using fibrations and opfibrations in two different settings.
In a given 2-category $\catC$ (generalizing $\mathbf{Loc}$), with suitable limits,
we are interested in when a morphism in it -- an object of $\arrC$ --
is a fibration or an opfibration.
On the other hand, we also use extensively the fact that,
if we forget the 2-cells in $\catC$, then the codomain functor
$\cod \colon \arrC \to \catC$ is a bifibration -- both a fibration and an opfibration --
in the 2-category of categories.

We use some results from Street~\cite{Street}.
There, for each base object $B$, a 2-monad $\monL_B$ is defined on $\catC/B$ which characterizes
the opfibrations as those bundles $p\colon E \rightarrow B$ that support the structure of a
normalized $\monL_B$-pseudoalgebra.
There is also a dual result (reversing 2-cells) that the fibrations are the normalized pseudoalgebras
for a 2-monad $\monR_B$.
Hence to prove that $T_B$ preserves opfibrations it is natural to prove that it lifts
to the pseudoalgebra category of $\monL_B$.
From the 1-categorical theory one would expect this to be equivalent to defining
a natural transformation that, with $T_B$, makes a monad opfunctor,
and the 2-categorical issues have already been taken care of in~\cite{MarmoWood:CohPsLW}.
However, we also find that we must extend Street's technique slightly by defining
a 2-monad $\monL_{\bullet}$ on the whole of $\mathcal{C}^{\downarrow}$.
We then show that $T_{\bullet}$ lifts to the pseudoalgebra category.

\section{Background}\label{sec:background}
\subsection{Some results from Street and 2-tangle notation}\label{sec:street-tangle}

We now recall some definitions and results from~\cite{Street}.
However, we shall also replace Street's diagrams with a dual ``2-tangle''
notation,
which seems to be more compact for handling 2-categories and
lax or pseudoalgebra structures.

  We use the following different interpretations for tangles to work
with 2-categories. In a 2-category we have 0-cells (objects), 1-cells (arrows)
and 2-cells, which, in the canonical example of $\mathbf{Cat}$ are
categories, functors and natural transformations.
In our intended application of $\mathbf{Loc}$ they are locales, continuous maps,
and instances of the specialization order.
In an ordinary diagram an $n$-cell is
represented by an $n$-dimensional object, hence 0-cell = vertex,
1-cell = (oriented) edge, 2-cell = (oriented) area connecting paths of
edges.
Tangles invert this association so that 0-cells are areas,
1-cells are vertical edges, with right-to-left reading order \emph{across} the edges,
and 2-cells are ``coupons'', with downward reading order from the edges attached at
the top to those at the bottom.
We introduce a special coupon for trivial (equality) 2-cells,
representing commutative diagrams.
By way of example we provide here an introduction to this
notion. A formal way to use these diagrams is given
in~\cite{Mellies:FunctorialBoxes}.

 Commutative diagrams are made out of $n$-gons, and we discuss how 2-cells
for bi-, tri- and tetra-gons are described. First we look at a 2-cell
$\alpha$ of a bigon (parallel pair of arrows $f,g$ from $A$ to $B$).
The vertices $A,B$ become areas in the tangle picture, which we shade here for clarity.
Later we will drop these area labels when they are reconstructible from the
tangle.
\begin{align}
  \xymatrix{
    A
    \ar@/^1pc/@{->}[r]^{f}
    \ar@/_1pc/@{->}[r]_{g}
    &
    B
    {\ar@{=>}^{\scriptstyle \alpha} (7,3); (7,-3)}
  }
  \simeq
  \begin{pic}
     \begin{pgfonlayer}{background}
        \draw[fill,color=blue!30] (-0.2,-0.5) rectangle (0.5,0.5);
        \draw[fill,color=green!30] (0.5,-0.5) rectangle (1.2,0.5);
     \end{pgfonlayer}
     \node (i) at (0.5,1) {$f$};
     \node (o) at (0.5,-1) {$g$};
     \node[fill=white,circle,draw,line width= 1.25pt,inner sep=1pt] (m) at (0.5,0) {$\alpha$};
     \node at (1,0) {$A$};
     \node at (0,0) {$B$};
     \draw (i) -- (m);
     \draw (m) -- (o);
     \draw[arrows=->,double distance=1.75pt,>=stealth]
           (1.5,0.7) -- node[xshift=3mm,rotate=90] {\tiny 2-cells} (1.5,-0.7);
     \draw[arrows=->,>=stealth]
           (1,-1.25) -- node[below]{\tiny 1-cells}(0,-1.25);
  \end{pic}
  \qquad
  \vcenter
  {\vbox{\hsize=0.5\textwidth
        \noindent Correspondence between a bigon diagram and a tangle.}}
\end{align}
For trigons and tetragons we find:
\begin{align}
  \vcenter{\hsize=2cm
     \xymatrix{
   A
     \ar@{->}[d]_{f}
     \ar@{->}[dr]^{h}
    &
   \\
   B
     \ar@{->}[r]_{g}
    &
   C
   {\ar@{=>}^<<{\scriptstyle \alpha} (6,-7); (2,-11)}
  }}\simeq
  \begin{pic}
     \begin{pgfonlayer}{background}
        \draw[fill,color=blue!30] (-0.5,-0.3) rectangle (0.5,0.6);
        \draw[fill,color=blue!30] (-0.5,-0.7) rectangle (0  ,-0.3);
        \draw[fill,color=green!30] (0.5,-0.3) rectangle (1.5,0.6);
        \draw[fill,color=green!30] (1  ,-0.7) rectangle (1.5,-0.3);
        \draw[fill,color=red!30]   (0  ,-0.7) rectangle (1  ,-0.3);
     \end{pgfonlayer}
     \node[unit] (t) at (0.5,-0.3){$\alpha$};
	 \node (i) at (0.5,1) {$h$};
	 \node (o1) at (0,-1.2) {$g$};
	 \node (o2) at (1,-1.2) {$f$};
	 \draw (t.tip) --(i);
	 \draw (o1 |- t.center) -- (o1);
	 \draw (o2 |- t.center) -- (o2);
	 \node at (1.3,0.4) {$A$};
	 \node at (0.5,-0.5) {$B$};
	 \node at (-0.3,0.4) {$C$};
  \end{pic}
  \qquad\qquad
  \vcenter{\hsize=2cm
     \xymatrix{
   A
     \ar@{->}[r]^{f^\prime}
     \ar@{->}[d]_{f}
     &
   B
     \ar@{->}[d]^{g^\prime}
   \\
   C
     \ar@{->}[r]_{g}
     &
   D
   {\ar@{=>}^{\scriptstyle \alpha} (5,-11); (12,-5)}
  }}
  \simeq
  \begin{pic}
     \begin{pgfonlayer}{background}
        \draw[fill,color=blue!30] (-0.5,-0.8) rectangle (0,0.8);
        \draw[fill,color=red!30] (0,0.2) rectangle (1,0.8);
        \draw[fill,color=red!30] (0,-0.8) rectangle (1,-0.2);
        \draw[fill,color=green!30] (1,-0.8) rectangle (1.5,0.8);
     \end{pgfonlayer}
       \node (i1) at (0,1.2) {$g$};
       \node (i2) at (1,1.2) {$f$};
       \node[2cell] (t) at (0.5,0) {$\alpha$};
       \node (o1) at (0,-1.2) {$g^\prime$};
       \node (o2) at (1,-1.2) {$f^\prime$};
       \draw (i1) -- (i1 |- t.north);
       \draw (i2) -- (i2 |- t.north);
       \draw (o1) -- (o1 |- t.south);
       \draw (o2) -- (o2 |- t.south);
       \node at (-0.3,0) {$D$};
       \node at ( 1.3,0) {$A$};
       \node at ( 0.5,0.5) {$C$};
       \node at ( 0.5,-0.5) {$B$};
  \end{pic}
\end{align}
(In practice all our coupons are drawn as rectangles, however many inputs or outputs they have.)

The coupons for identity 2-cells, for ordinary commutative diagrams,
are depicted as double lines representing an equality symbol.
\begin{align}
  \vcenter{\hsize=2cm
     \xymatrix{
   A
     \ar@{->}[r]^{f^\prime}
     \ar@{->}[d]_{f}
     &
   B
     \ar@{->}[d]^{g^\prime}
   \\
   C
     \ar@{->}[r]_{g}
     &
   D
  }}
  \simeq
  \begin{pic}
     \begin{pgfonlayer}{background}
        \draw[fill,color=blue!30] (-0.5,-0.8) rectangle (0,0.8);
        \draw[fill,color=red!30] (0,0.1) rectangle (1,0.8);
        \draw[fill,color=red!30] (0,-0.8) rectangle (1,-0.1);
        \draw[fill,color=green!30] (1,-0.8) rectangle (1.5,0.8);
     \end{pgfonlayer}
       \node (i1) at (0,1.2) {$g$};
       \node (i2) at (1,1.2) {$f$};
       \node[iso2cell] (t) at (0.5,0) {};
       \node (o1) at (0,-1.2) {$g^\prime$};
       \node (o2) at (1,-1.2) {$f^\prime$};
       \draw (i1) -- (i1 |- t.north);
       \draw (i2) -- (i2 |- t.north);
       \draw (o1) -- (o1 |- t.south);
       \draw (o2) -- (o2 |- t.south);
       \node at (-0.3,0) {$D$};
       \node at ( 1.3,0) {$A$};
       \node at ( 0.5,0.5) {$C$};
       \node at ( 0.5,-0.5) {$B$};
  \end{pic}
  \qquad;\qquad
  \vcenter{\hsize=2cm
     \xymatrix{
   A
     \ar@{->}[d]_{f}
     \ar@{->}[dr]^{h}
    &
   \\
   B
     \ar@{->}[r]_{g}
    &
   C
  }}\simeq
  \begin{pic}
     \begin{pgfonlayer}{background}
        \draw[fill,color=blue!30] (-0.5,-0.0) rectangle (0.5,0.6);
        \draw[fill,color=blue!30] (-0.5,-0.7) rectangle (0  ,-0.0);
        \draw[fill,color=green!30] (0.5,-0.0) rectangle (1.5,0.6);
        \draw[fill,color=green!30] (1  ,-0.7) rectangle (1.5,-0.0);
        \draw[fill,color=red!30]   (0  ,-0.7) rectangle (1  ,-0.2);
     \end{pgfonlayer}
     \node[iso2cell] (t) at (0.5,-0.1) {};
	 \node (i) at (0.5,1) {$h$};
	 \node (o1) at (0,-1.2) {$g$};
	 \node (o2) at (1,-1.2) {$f$};
	 \draw (i |- t.north) --(i);
	 \draw (o1 |- t.south) -- (o1);
	 \draw (o2 |- t.south) -- (o2);
	 \node at (1.3,0.4) {$A$};
	 \node at (0.5,-0.5) {$B$};
	 \node at (-0.3,0.4) {$C$};
  \end{pic}
\end{align}
A general commutative $n$-gon may
have $i$ input and $j$ output lines (with $i+j=n$).
Such an $n$-gon may need further manipulation of its internal structure,
which we lose (on purpose, and to compactify) in our notation.
Furthermore, we can merge,
but have to be careful when disconnecting, such diagrams.
We drop from here onwards the shading and labelling of the 0-cells.
\begin{align}
   \begin{pic}
      \node (i1) at (0,1) {$f_1$};
      \node (i2) at (0.5,1) {$f_2$};
      \node (i3) at (2,1) {$f_i$};
      \node (o1) at (0,-1) {$g_1$};
      \node (o2) at (1.4,-1) {$g_{j-1}$};
      \node (o3) at (2,-1) {$g_j$};
      \node[iso2cell,minimum width=2.2cm] (t) at (1,0) {};
      \draw (i1) -- (i1 |- t.north);
      \draw (i2) -- (i2 |- t.north);
      \draw (i3) -- (i3 |- t.north);
      \draw (o1) -- (o1 |- t.south);
      \draw (o2) -- (o2 |- t.south);
      \draw (o3) -- (o3 |- t.south);
      \node at (1.1,0.5) {$\ldots$};
      \node at (0.7,-0.5) {$\ldots$};
   \end{pic}
   ;\qquad
   \begin{pic}[auto]
      \node (i1) at (0,1) {$f_1$};
      \node (i2) at (1,1) {$f_2$};
      \node (i3) at (2,1) {$f_3$};
      \node (o1) at (0,-1) {$g_1$};
      \node (o2) at (1,-1) {$g_2$};
      \node (o3) at (2,-1) {$g_3$};
      \node[iso2cell] (eq1) at (0.5,0.3) {};
      \node[iso2cell] (eq2) at (1.5,-0.3) {};
      \draw (i1) -- (i1 |- eq1.north);
      \draw (i2) -- (i2 |- eq1.north);
      \draw (i3) -- (i3 |- eq2.north);
      \draw (o1) -- (o1 |- eq1.south);
      \draw (o2) -- (o2 |- eq2.south);
      \draw (o3) -- (o3 |- eq2.south);
      \draw (i2 |- eq1.south) to node {\tiny $h$} (i2 |- eq2.north);
   \end{pic}
   \stackrel{\textrm{merge}}{\Rightarrow}
   \begin{pic}[auto]
      \node (i1) at (0,1) {$f_1$};
      \node (i2) at (1,1) {$f_2$};
      \node (i3) at (2,1) {$f_3$};
      \node (o1) at (0,-1) {$g_1$};
      \node (o2) at (1,-1) {$g_2$};
      \node (o3) at (2,-1) {$g_3$};
      \node[iso2cell,minimum width=2.1cm] (eq1) at (1,0) {};
      \draw (i1) -- (i1 |- eq1.north);
      \draw (i2) -- (i2 |- eq1.north);
      \draw (i3) -- (i3 |- eq1.north);
      \draw (o1) -- (o1 |- eq1.south);
      \draw (o2) -- (o2 |- eq1.south);
      \draw (o3) -- (o3 |- eq1.south);
   \end{pic}
\end{align}
A collection of identities with one input $f$ and one output $f$ is
equivalent to a trivial identity 2-cell.
\begin{align}
   \begin{pic}
      \node (i) at (2,1) {$f$};
      \node (o) at (0,-1) {$f$};
      \node (m) at (1,0) {};
      \node[iso2cell] (eq1) at (0.5,0.75) {};
      \node[iso2cell] (eq2) at (1.5,0.25) {};
      \node[iso2cell,minimum width=2.1cm] (eq3) at (1,-0.25) {};
      \draw (i) -- (i |- eq2.north);
      \draw (o) -- (o |- eq3.south);
      \draw (o |- eq1.south) -- (o |- eq3.north);
      \draw (m |- eq1.south) -- (m |- eq2.north);
      \draw (m |- eq2.south) -- (m |- eq3.north);
      \draw (i |- eq2.south) -- (i |- eq3.north);
   \end{pic}
   \stackrel{\textrm{merge}}{\Rightarrow}
   \begin{pic}
      \node (i) at (1,1) {$f$};
      \node (o) at (0,-1) {$f$};
      \node[iso2cell] (eq) at (0.5,0) {};
      \draw (i) -- (i |- eq.north);
      \draw (o) -- (o |- eq.south);
   \end{pic}
   =
  \begin{pic}
      \node (i) at (0,1) {$f$};
      \node[coupon] (m) at (0,0) {$1$};
      \node (o) at (0,-1) {$f$};
      \draw (i) -- (m);
      \draw (m) -- (o);
  \end{pic}
   =
  \begin{pic}
      \node (i) at (0,1) {$f$};
      \node (o) at (0,-1) {$f$};
      \draw (i) -- (o);
  \end{pic}
  \text{;}\,\,
  \begin{pic}
      \node (i1) at (0,1) {$g$};
      \node (i2) at (1,1) {$f$};
      \node (o1) at (0,-1) {$g$};
      \node (o2) at (1,-1) {$f$};
      \node[iso2cell] (eq1) at (0.5,0.3) {};
      \node[iso2cell] (eq2) at (0.5,-0.3) {};
      \draw (i1) -- (i1 |- eq1.north);
      \draw (i2) -- (i2 |- eq1.north);
      \draw (o1) -- (o1 |- eq2.south);
      \draw (o2) -- (o2 |- eq2.south);
  \end{pic}
  =
\begin{pic}
      \node (i1) at (0,1) {$g$};
      \node (i2) at (1,1) {$f$};
      \node (o1) at (0,-1) {$g$};
      \node (o2) at (1,-1) {$f$};
      \draw (i1) -- (o1);
      \draw (i2) -- (o2);
  \end{pic}
  \text{;}\,\,
\begin{pic}[auto]
      \node (i1) at (0,1) {};
      \node (i2) at (1,1) {};
      \node (o1) at (0,-1) {};
      \node (o2) at (1,-1) {};
      \node[iso2cell] (eq1) at (0.5,0.4) {};
      \node[iso2cell] (eq2) at (0.5,-0.4) {};
      \draw (i1 |- eq1.south) to node[swap] {$g$} (i1 |- eq2.north);
      \draw (i2 |- eq1.south) to node[swap] {$f$} (i2 |- eq2.north);
  \end{pic}
  =
  \quad
\end{align}
  The last two equalities are due to the invertibility of the identity
2-cell, when one edge of the commutative $n$-gon (here a trigon) is
identity. A similar equation holds for non-trivial 2-cells (e.g. triangles)
iff they are invertible.
The right-hand side of the final equation is an identity 2-cell on an identity 1-cell.
This is a void diagram and vanishes.

\medskip

  We continue to present some definitions and results from
Street~\cite{Street} in 2-tangle notation. Let $\catC$ be a category,
$A,B$ be objects in $\catC$. The category $\Span(A,B)$ has as objects
spans $(u_0,S,u_1)$ from $A$ to $B$ and arrows are arrows of spans
$f \colon (u_0,S,u_1) \rightarrow (v_0,S^\prime,v_1)$
\begin{align}\label{eq:span-arrow}
  \vcenter{\hsize=3cm\xymatrix{
    &
     S
     \ar@{->}[dl]_{u_0}
     \ar@{->}[dr]^{u_1}
    \\
     A
    & &
     B
  }}
  &
  ;\qquad
  \vcenter{\hsize=4cm\xymatrix{
    &
     S
     \ar@{->}[dl]_{u_0}
     \ar@{->}[dr]^{u_1}
     \ar@{->}[dd]^{f}
    \\
     A
    & &
     B
    \\
    &
     S^\prime
     \ar@{->}[ul]^{v_0}
     \ar@{->}[ur]_{v_1}
  }}
  \Leftrightarrow
   \begin{pic}
      \node (i1) at (0,1) {$v_0$};
      \node (i2) at (1,1) {$f$};
      \node (o) at (0.5,-1) {$u_0$};
      \node[iso2cell,minimum width=1.1cm] (eq) at (0.5,0) {};
      \draw (i1) -- (i1 |- eq.north);
      \draw (i2) -- (i2 |- eq.north);
      \draw (o |- eq.south) -- (o);
   \end{pic}
   \textrm{~and~}
   \begin{pic}
      \node (i1) at (0,1) {$v_1$};
      \node (i2) at (1,1) {$f$};
      \node (o) at (0.5,-1) {$u_1$};
      \node[iso2cell,minimum width=1.1cm] (eq) at (0.5,0) {};
      \draw (i1) -- (i1 |- eq.north);
      \draw (i2) -- (i2 |- eq.north);
      \draw (o |- eq.south) -- (o);
   \end{pic}.
\end{align}
For a span $S$, the \emph{reverse span} $S^*$ form $B$ to $A$ is given
by $(u_1,S,u_0)$. If $\catC$ has pullbacks, then a span $(u_0,S,u_1)$
from $A$ to $B$ and a span $(v_0,T,v_1)$ from $B$ to $C$ has a composite
span $(u_0\hat{v}_0,T\circ S,v_1\hat{u_1})$ from $A$ to $C$ defined as
\begin{align}
  \vcenter{\hsize=4cm\xymatrix{
    & &
     T\circ S
     {\ar@{-}(28,-6); (31,-9)
      \ar@{-}(31,-9); (34,-6)}
     \ar@{->}[dl]_{\hat{v}_0}
     \ar@{->}[dr]^{\hat{u}_1}
    \\
    &
     S
     \ar@{->}[dl]_{u_0}
     \ar@{->}[dr]^{u_1}
    & &
     T
     \ar@{->}[dl]_{v_0}
     \ar@{->}[dr]^{v_1}
    \\
     A
    & &
     B
    & &
     C
  }}
\end{align}
where, following Street, we decorate pulled back arrows with hats.
Similarly, given composable spans and arrows $f,g$ of such spans
then the arrow $g\circ f \colon T\circ S \rightarrow T^\prime\circ S^\prime$
is induced on pullbacks is an arrow of the composite spans.
An \emph{opspan} from $A$ to $B$ in $\catC$ is a span from $A$ to $B$
in $\catC^{op}$, but arrows of opspans are arrows from $\catC$.
\begin{definition}\label{def:commaObject}
A \emph{comma object} for the opspan $(r,D,s)$ from $A$ to $B$ is a
span $(d_0,r/s,d_1)$ from $A$ to $B$ together with a 2-cell
\begin{align}
  \vcenter{\hsize=3cm\xymatrix{
     r/s
     \ar@{->}[d]_{d_0}
     \ar@{->}[r]^{d_1}
    &
     B
     \ar@{->}[d]^{s}
    \\
     A
     \ar@{->}[r]^{r}
    &
     D
     {\ar@{=>}^{\scriptstyle \lambda} (4,-8); (12,-8)}
  }}
  &\simeq
  \begin{pic}
       \node (i1) at (0,1) {$r$};
       \node (i2) at (1,1) {$d_0$};
       \node[2cell,inner sep=2pt] (t) at (0.5,0) {$\lambda$};
       \node (o1) at (0,-1) {$s$};
       \node (o2) at (1,-1) {$d_1$};
       \draw (i1) -- (i1 |- t.north);
       \draw (i2) -- (i2 |- t.north);
       \draw (o1) -- (o1 |- t.south);
       \draw (o2) -- (o2 |- t.south);
  \end{pic}
\end{align}
satisfying the two following universality conditions.
\begin{itemize}
\item For any span $(u_0,S,u_1)$ from $A$ to $B$, composition with
$\lambda$ yields a bijection between arrows of spans
$f$~\eqref{eq:span-arrow} and 2-cells $\sigma$ given by (either one of)
\begin{align}\label{eq:comma1}
   \begin{pic}[auto]
       \node (i1) at (0,1) {$r$};
       \node (i2) at (1.5,1) {$u_0$};
       \node[iso2cell,minimum width=1.1cm] (eq1) at (1.5,0.5) {};
       \node[2cell,inner sep=2pt] (t) at (0.5,0) {$\lambda$};
       \node[iso2cell,minimum width=1.1cm] (eq2) at (1.5,-0.5) {};
       \node (m) at (1,0) {};
       \node (mm) at (2,0) {};
       \node (o1) at (0,-1) {$s$};
       \node (o2) at (1.5,-1) {$u_1$};
       \draw (i1) -- (i1 |- t.north);
       \draw (i2) -- (i2 |- eq1.north);
       \draw (m |- eq1.south) to node {\tiny $d_0$} (m |- t.north);
       \draw (m |- t.south) to node {\tiny $d_1$} (m |- eq2.north);
       \draw (mm |- eq1.south) to node {\tiny $f$} (mm |- eq2.north);
       \draw (o1) -- (o1 |- t.south);
       \draw (o2) -- (o2 |- eq2.south);
  \end{pic}
  =
  \begin{pic}
       \node (i1) at (0,1) {$r$};
       \node (i2) at (1,1) {$u_0$};
       \node[2cell,inner sep=2pt] (t) at (0.5,0) {$\sigma$};
       \node (o1) at (0,-1) {$s$};
       \node (o2) at (1,-1) {$u_1$};
       \draw (i1) -- (i1 |- t.north);
       \draw (i2) -- (i2 |- t.north);
       \draw (o1) -- (o1 |- t.south);
       \draw (o2) -- (o2 |- t.south);
  \end{pic};
  \quad\quad
  \begin{pic}[auto]
       \node (i1) at (0,1) {$r$};
       \node (i2) at (1,1) {$d_0$};
       \node (i3) at (2,1) {$f$};
       \node[2cell,inner sep=2pt] (t) at (0.5,0) {$\lambda$};
       \node (o1) at (0,-1) {$s$};
       \node (o2) at (1,-1) {$d_1$};
       \node (o3) at (2,-1) {$f$};
       \draw (i1) -- (i1 |- t.north);
       \draw (i2) -- (i2 |- t.north);
       \draw (i3) -- (o3);
       \draw (o1) -- (o1 |- t.south);
       \draw (o2) -- (o2 |- t.south);
  \end{pic}
  =
   \begin{pic}[auto]
       \node (i1) at (0,1) {$r$};
       \node (i2) at (1,1) {$d_0$};
       \node (i3) at (2,1) {$f$};
       \node[iso2cell,minimum width=1.1cm] (eq1) at (1.5,0.5) {};
       \node[2cell,inner sep=2pt] (t) at (0.5,0) {$\sigma$};
       \node[iso2cell,minimum width=1.1cm] (eq2) at (1.5,-0.5) {};
       \node (o1) at (0,-1) {$s$};
       \node (o2) at (1,-1) {$d_1$};
       \node (o3) at (2,-1) {$f$};
       \draw (i1) -- (i1 |- t.north);
       \draw (i2) -- (i2 |- eq1.north);
       \draw (i3) -- (i3 |- eq1.north);
       \draw (i2 |- eq1.south) to node {\tiny $u_0$} (i2 |- t.north);
       \draw (i2 |- t.south) to node {\tiny $u_1$} (i2 |- eq2.north);
       \draw (o1) -- (o1 |- t.south);
       \draw (o2) -- (o2 |- eq2.south);
       \draw (o3) -- (o3 |- eq2.south);
  \end{pic}\,.
\end{align}
\item Given 2-cells $\xi,\eta$ such that the two composites
\begin{align}\label{eq:comma2}
   \begin{pic}[auto]
       \node (i1) at (0,1.2) {$r$};
       \node (i2) at (1,1.2) {$d_0$};
       \node (i3) at (2,1.2) {$f$};
       \node (o1) at (0,-1.2) {$s$};
       \node (o2) at (1,-1.2) {$d_1$};
       \node (o3) at (2,-1.2) {$f^\prime$};
       \node[2cell,inner sep=2pt] (xi) at (1.5,0.4) {$\xi$};
       \node[2cell,inner sep=2pt] (la) at (0.5,-0.45) {$\lambda$};
       \draw (i1) -- (i1 |- la.north);
       \draw (i2) -- (i2 |- xi.north);
       \draw (i3) -- (i3 |- xi.north);
       \draw (i2 |- xi.south) to node {\tiny $d_0$} (i2 |- la.north);
       \draw (o1) -- (o1 |- la.south);
       \draw (o2) -- (o2 |- la.south);
       \draw (o3) -- (o3 |- xi.south);
   \end{pic}
  &=
   \begin{pic}[auto]
       \node (i1) at (0,1.2) {$r$};
       \node (i2) at (1,1.2) {$d_0$};
       \node (i3) at (2,1.2) {$f$};
       \node (o1) at (0,-1.2) {$s$};
       \node (o2) at (1,-1.2) {$d_1$};
       \node (o3) at (2,-1.2) {$f^\prime$};
       \node[2cell,inner sep=2pt] (xi) at (1.5,-0.4) {$\eta$};
       \node[2cell,inner sep=2pt] (la) at (0.5,0.4) {$\lambda$};
       \draw (i1) -- (i1 |- la.north);
       \draw (i2) -- (i2 |- la.north);
       \draw (i3) -- (i3 |- xi.north);
       \draw (i2 |- la.south) to node {\tiny $d_1$} (i2 |- xi.north);
       \draw (o1) -- (o1 |- la.south);
       \draw (o2) -- (o2 |- xi.south);
       \draw (o3) -- (o3 |- xi.south);
   \end{pic}
\end{align}
are equal, then there exists a unique 2-cell $\phi$ such that
$\xi=d_0\phi$, $\eta=d_1\phi$.
\begin{align}
  \exists !\,
   \begin{pic}
      \node (i1) at (0,1) {$f$};
      \node (o1) at (0,-1) {$f^\prime$};
      \node[fill=white,circle,draw,line width= 1.25pt,inner sep=1pt] (phi) at (0,0) {$\phi$};
      \draw (i1) -- (phi.north);
      \draw (phi.south) -- (o1);
   \end{pic}
   \textrm{~s.t.~}
  \begin{pic}
       \node (i1) at (0,1) {$d_0$};
       \node (i2) at (1,1) {$f$};
       \node[2cell,inner sep=2pt] (t) at (0.5,0) {$\xi$};
       \node (o1) at (0,-1) {$d_0$};
       \node (o2) at (1,-1) {$f^\prime$};
       \draw (i1) -- (i1 |- t.north);
       \draw (i2) -- (i2 |- t.north);
       \draw (o1) -- (o1 |- t.south);
       \draw (o2) -- (o2 |- t.south);
  \end{pic}
  &=
  \begin{pic}
      \node (i1) at (0,1) {$d_0$};
      \node (i2) at (1,1) {$f$};
      \node (o1) at (0,-1) {$d_0$};
      \node (o2) at (1,-1) {$f^\prime$};
      \node[fill=white,circle,draw,line width= 1.25pt,inner sep=1pt] (phi) at (1,0) {$\phi$};
      \draw (i2) -- (phi.north);
      \draw (phi.south) -- (o2);
      \draw (i1) -- (o1);
   \end{pic}
   \,\,\textrm{~and~}
  \begin{pic}
       \node (i1) at (0,1) {$d_1$};
       \node (i2) at (1,1) {$f$};
       \node[2cell,inner sep=2pt] (t) at (0.5,0) {$\eta$};
       \node (o1) at (0,-1) {$d_1$};
       \node (o2) at (1,-1) {$f^\prime$};
       \draw (i1) -- (i1 |- t.north);
       \draw (i2) -- (i2 |- t.north);
       \draw (o1) -- (o1 |- t.south);
       \draw (o2) -- (o2 |- t.south);
  \end{pic}
  =
     \begin{pic}
      \node (i1) at (0,1) {$d_1$};
      \node (i2) at (1,1) {$f$};
      \node (o1) at (0,-1) {$d_1$};
      \node (o2) at (1,-1) {$f^\prime$};
      \node[fill=white,circle,draw,line width= 1.25pt,inner sep=1pt] (phi) at (1,0) {$\phi$};
      \draw (i2) -- (phi.north);
      \draw (phi.south) -- (o2);
      \draw (i1) -- (o1);
   \end{pic}
\end{align}
\end{itemize}
\end{definition}
The comma object of the identity opspan $(1,A,1)$ from $A$ to $A$
is denoted by $\Phi A= A/A$.
When $\Phi A$ exists for each object $A$
and if $\catC$ has 2-pullbacks, then $\catC$ is a \emph{representable}
2-category. In a representable 2-category every opspan $(r,D,s)$ has a
comma object via span composition $r/s = s^*\circ \Phi D\circ r$.

An intuitive way to think about the 2-category $\catC$ is that each 0-cell $A$
itself has objects and morphisms,
but in the generalized sense of generalized elements in a topos.
Given a ``stage of definition'', another 0-cell $W$,
the generalized objects and morphisms of $A$ are the objects and morphisms
of the category $\catC(W,A)$.
1-cells $A\to A'$ and 2-cells between them then give functors
$\catC(W,A) \to \catC(W,A')$ and natural transformations between them.
Hiding $W$ we can then understand Definition~\ref{def:commaObject}
as saying that the (generalized) objects of $r/s$ are triples
$(u_0,u_1,\sigma)$ where $u_0,u_1$ are objects of $A$ and $B$
and $\sigma\colon ru_0 \to su_1$ is a morphism.
A morphism from $(u_0,u_1,\sigma)$ to $(u'_0,u'_1,\sigma')$
is a pair $(\xi,\eta)$ where $\xi\colon u_0 \to u'_0$
and $\eta\colon u_1 \to u'_1$, and there is a commutative square
\[
  \xymatrix{
    ru_0 \ar@{->}[d]_{r\xi} \ar@{->}[r]^{\sigma}
    & su_1 \ar@{->}[d]^{s\eta} \\
    ru'_0 \ar@{->}[r]_{\sigma'}
    & su'_1
  }
\]
It is often helpful to calculate 1-cells and 2-cells in terms of
their action on generalized objects and morphisms.

If $p\colon E \to B$ is a bundle (1-cell) in $\catC$,
and $u$ is a generalized object of $B$,
then the pullback $u^{\ast}E$ is the \emph{generalized fibre} of $p$ over $u$.
The reason for this point of view can be understood by considering the case
where the stage of definition $W$ is the terminal object 1 in $\mathbf{Loc}$
or in the 2-category of categories, when pullback along points gives actual fibres.
Thus geometricity of a construction as preservation under pullback can be understood
as the property that the construction works fibrewise,
and hence is compatible with viewing a bundle as a dependent type.

Let $\monD$ be a 2-monad on a 2-category $\catC$, with unit
$i \colon 1 \rightarrow \monD$ and multiplication (composition)
$c \colon \monD\monD \rightarrow \monD$.
\begin{definition}\label{def:laxAlgebra}
A \emph{lax $\monD$-algebra} is a quadruple
$(E,c,\zeta,\theta)$ where $E$ (the carrier) is an object of $\catC$,
$c$ (the structure map) is an arrow $c \colon \monD E \rightarrow E$ and
$\zeta,\theta$ are 2-cells
\begin{align}
  \vcenter{\hsize=2cm
     \xymatrix{
   E
     \ar@{->}[d]_{i_{E}}
     \ar@{->}[dr]^{1}
    &
   \\
   \monD E
     \ar@{->}[r]_{c}
    &
   E
   {\ar@{=>}^<<{\scriptstyle \zeta} (6,-7); (2,-11)}
  }}\simeq
  \begin{pic}
	   \node[unit] (t) at (0.5,1){$\zeta$};
	   \node (o1) at (0,0) {$c$};
	   \node (o2) at (1,0) {$i_E$};
	   \draw (o1 |- t.center) -- (o1);
	   \draw (o2 |- t.center) -- (o2);
  \end{pic}
;\quad\quad
  \vcenter{\hsize=2cm
     \xymatrix{
   \monD\monD E
     \ar@{->}[r]^{cE}
     \ar@{->}[d]_{\monD c}
     &
   \monD E
     \ar@{->}[d]^{c}
   \\
   \monD E
     \ar@{->}[r]_{c}
     &
   E
   {\ar@{=>}^{\scriptstyle \theta} (8,-11); (15,-5)}
  }}
  \simeq
  \begin{pic}
       \node (i1) at (0,1) {$c$};
       \node (i2) at (1,1) {$\monD c$};
       \node[2cell] (t) at (0.5,0) {$\theta$};
       \node (o1) at (0,-1) {$c$};
       \node (o2) at (1,-1) {$cE$};
       \draw (i1) -- (i1 |- t.north);
       \draw (i2) -- (i2 |- t.north);
       \draw (o1) -- (o1 |- t.south);
       \draw (o2) -- (o2 |- t.south);
  \end{pic};
\end{align}
such that the following 3 equations hold.
\begin{align}
  \begin{pic}[auto]
     \node (i1) at (1.5,1.5) {$c$};
     \node[unit] (z1) at (0.5,0.75) {$\zeta$};
     \node[iso2cell] (eq1) at (1.25,0.2) {};
     \node[2cell] (t) at (0.5,-0.5) {$\theta$};
     \node[iso2cell] (eq2) at (1.25,-1.2) {};
     \node (o1) at (0,-1.5) {$c$};
     \draw (z1.10) to node[swap] {\tiny$c$} (z1.10 |- t.north);
     \draw (o1 |- t.south) -- (o1);
     \draw (z1.90) to node[swap] {\tiny$i_E$} (z1.90 |- eq1.north);
     \draw (i1 |- eq1.north) -- (i1);
     \draw (z1.90 |- eq1.south) to node[swap] {\tiny $\monD c$} (z1.90 |- t.north);
     \draw (z1.90 |- t.south) to node[swap] {\tiny $c E$} (z1.90 |- eq2.north);
     \draw (i1 |- eq1.south) to node {\tiny $i_{\monD\! E}$} (i1 |- eq2.north);
  \end{pic}
  \stackrel{(1)}{=}
  \begin{pic}
     \node (i1) at (0,1.5) {$c$};
     \node (o1) at (0,-1.5) {$c$};
     \draw (i1) -- (o1);
  \end{pic}
  \stackrel{(2)}{=}
  \begin{pic}[auto]
     \node (i1) at (0,1.5) {$c$};
     \node (o1) at (0,-1.5) {$c$};
     \node[unit] (z1) at (1.25,0.75) {$\monD\zeta$};
     \node[2cell] (t) at (0.5,0) {$\theta$};
     \node[iso2cell] (eq1) at (1.25,-0.75) {};
     \draw (i1) -- (i1 |- t.north);
     \draw (z1.20) to node[swap] {\tiny $\monD\!c$} (z1.20 |- t.north);
     \draw (o1 |- t.south) -- (o1);
     \draw (z1.20 |- t.south) to node[swap] {\tiny$cE$} (z1.20 |- eq1.north);
     \draw (z1.80) to node {\tiny $\monD i_E$} (z1.80 |- eq1.north);
   \end{pic}
;\quad
  \begin{pic}[auto]
     \node (i1) at (0,1.5) {$c $};
     \node (i2) at (1,1.5) {$\monD c $};
     \node (i3) at (2,1.5) {$\monD^2 c $};
     \node[2cell] (dt) at (1.5,0.6) {$\monD\theta $};
     \node[iso2cell] (eq1) at (1.5,-1) {};
     \node[2cell] (t) at (0.5,-0.3) {$\theta $};
     \node (o1) at (0,-1.5) {$c $};
     \node (o2) at (1,-1.5) {$cE $};
     \node (o3) at (2,-1.5) {$c\monD E $};
     \draw (i1) -- (i1 |- t.north);
     \draw (o1) -- (o1 |- t.south);
     \draw (i2) -- (i2 |- dt.north);
     \draw (i3) -- (i3 |- dt.north);
     \draw (i2 |- dt.south) to node[swap] {\tiny $\monD c$} (i2 |- t.north);
     \draw (i2 |- t.south) to node[swap] {\tiny $cE$} (i2 |- eq1.north);
     \draw (i3 |- dt.south) to node {\tiny $\monD c$} (i3 |- eq1.north);
     \draw (o2 |- eq1.south) -- (o2);
     \draw (o3 |- eq1.south) -- (o3);
  \end{pic}
  \stackrel{(3)}{=}
  \begin{pic}[auto]
     \node (i1) at (0,1.5) {$c $};
     \node (i2) at (1,1.5) {$\monD c $};
     \node (i3) at (2,1.5) {$\monD^2 c $};
     \node[2cell] (t1) at (0.5,0.7) {$\theta $};
     \node[iso2cell] (eq1) at (1.5,0) {};
     \node[2cell] (t2) at (0.5,-0.7) {$\theta $};
     \node (o1) at (0,-1.5) {$c $};
     \node (o2) at (1,-1.5) {$cE $};
     \node (o3) at (2,-1.5) {$c\monD E $};
     \draw (i1) -- (i1 |- t1.north);
     \draw (i2) -- (i2 |- t1.north);
     \draw (i3) -- (i3 |- eq1.north);
     \draw (o1) -- (o1 |- t2.south);
     \draw (o2) -- (o2 |- t2.south);
     \draw (o3) -- (o3 |- eq1.south);
     \draw (i1 |- t1.south) to node[swap] {\tiny $c$} (i1 |- t2.north);
     \draw (i2 |- t1.south) to node[swap] {\tiny $cE$} (i2 |- eq1.north);
     \draw (i2 |- eq1.south) to node[swap] {\tiny $\monD c$} (i2 |- t2.north);
  \end{pic}
\end{align}
  A \emph{$\monD$-pseudoalgebra} has $\zeta,\theta$ invertible; a
\emph{normalized $\monD$-algebra} has $\zeta=1$, a
\emph{$\monD$-algebra} has $\zeta=1=\theta$. $(\monD E,c,\zeta=1,\theta=1)$
is the \emph{free $\monD$-algebra} on $E$.
\end{definition}

\begin{definition}
A \emph{lax homomorphism} of lax $\monD$-algebras from $E$ to $E^\prime$
is a pair $(f,\theta_f)$ of an arrow $f \colon E \rightarrow E^\prime$ and
a 2-cell
\begin{align}
   \vcenter{\hsize=2cm
     \xymatrix{
   \monD E
     \ar@{->}[r]^{c}
     \ar@{->}[d]_{\monD f}
     &
   E
     \ar@{->}[d]^{f}
   \\
   \monD E^\prime
     \ar@{->}[r]_{c}
     &
   E^\prime
   {\ar@{=>}^{\scriptstyle \theta_f} (8,-11); (15,-5)}
  }}
  \simeq
  \begin{pic}
     \node (i1) at (0,1) {$c$};
     \node (i2) at (1,1) {$\monD f$};
     \node[2cell] (t) at (0.5,0) {$\theta_f$};
     \node (o1) at (0,-1) {$f$};
     \node (o2) at (1,-1) {$c$};
     \draw (i1) -- (i1 |- t.north);
     \draw (i2) -- (i2 |- t.north);
     \draw (o1) -- (o1 |- t.south);
     \draw (o2) -- (o2 |- t.south);
  \end{pic}
\end{align}
such that the following equations hold
\begin{align}
  \begin{pic}[auto]
     \node[unit] (z1) at (0.5,0.75) {$\zeta$};
     \node[iso2cell,minimum width=1.1cm] (eq1) at (1.5,0.2) {};
     \node[2cell,minimum width=1.1cm] (t) at (0.5,-0.4) {$\theta_f$};
     \node (i) at (2,1.5) {$f$};
     \node (m1) at (0,0) {};
     \node (m2) at (1,0) {};
     \node (o1) at (0,-1.5) {$f$};
     \node (o2) at (1,-1.5) {$c$};
     \node (o3) at (2,-1.5) {$i_E$};
     \draw (m1 |- z1) to node[swap] {\tiny $c$}  (m1 |- t.north);
     \draw (m2 |- z1) to node[swap] {\tiny $i_{E^\prime}$}  (m2 |- eq1.north);
     \draw (m2 |- eq1.south) to node[swap] {\tiny $\monD f$}  (m2 |- t.north);
     \draw (o1 |- t.south) -- (o1);
     \draw (o2 |- t.south) -- (o2);
     \draw (i) to (i |- eq1.north);
     \draw (o3 |- eq1.south) -- (o3);
  \end{pic}
  \stackrel{(5)}{=}
  \begin{pic}[auto]
     \node (i) at (0,1.5) {$f$};
     \node (o1) at (0,-1.5) {$f$};
     \node (o2) at (1,-1.5) {$c$};
     \node (o3) at (2,-1.5) {$i_E$};
     \node[unit] (z1) at (1.5,0.5) {$\zeta$};
     \draw (i) to (o1);
     \draw (o2 |- z1) --(o2);
     \draw (o3 |- z1) --(o3);
  \end{pic}
;\qquad
  \begin{pic}[auto]
     \node (i1) at (0,1.5) {$c$};
     \node (i2) at (1,1.5) {$\monD c$};
     \node (i3) at (2,1.5) {$\monD^2 f$};
     \node (o1) at (0,-1.5) {$f$};
     \node (o2) at (1,-1.5) {$c$};
     \node (o3) at (2,-1.5) {$c_E$};
     \node[2cell,minimum width=1.1cm] (dt) at (1.5,0.8) {$\monD\theta_f$};
     \node[2cell,minimum width=1.1cm] (tf) at (0.5,0) {$\theta_f$};
     \node[2cell,minimum width=1.1cm] (t) at (1.5,-0.8) {$\theta$};
     \draw (i1) to (i1 |- tf.north);
     \draw (i2) to (i2 |- dt.north);
     \draw (i3) to (i3 |- dt.north);
     \draw (o1 |- tf.south) to (o1);
     \draw (o2 |- t.south) to (o2);
     \draw (o3 |- t.south) to (o3);
     \draw (i3 |- dt.south) to node {\tiny $\monD c$} (i3 |- t.north);
     \draw (i2 |- dt.south) to node[swap] {\tiny $\monD f$} (i2 |- tf.north);
     \draw (i2 |- tf.south) to node {\tiny $c$} (i2 |- t.north);
  \end{pic}
  \stackrel{(6)}{=}
  \begin{pic}[auto]
     \node (i1) at (0,1.5) {$c$};
     \node (i2) at (1,1.5) {$\monD c$};
     \node (i3) at (2,1.5) {$\monD^2 f$};
     \node (o1) at (0,-1.5) {$f$};
     \node (o2) at (1,-1.5) {$c$};
     \node (o3) at (2,-1.5) {$c_E$};
     \node[2cell,minimum width=1.1cm] (t) at (0.5,0.7) {$\theta$};
     \node[iso2cell,minimum width=1.1cm] (eq) at (1.5,0) {};
     \node[2cell,minimum width=1.1cm] (tf) at (0.5,-0.7) {$\theta_f$};
     \draw (i1) to (i1 |- t.north);
     \draw (i2) to (i2 |- t.north);
     \draw (i3) to (i3 |- eq.north);
     \draw (o1 |- tf.south) to (o1);
     \draw (o2 |- tf.south) to (o2);
     \draw (o3 |- eq.south) to (o3);
     \draw (i1 |- t.south) to node {\tiny $c$} (i1 |- tf.north);
     \draw (i2 |- t.south) to node {\tiny $cE^\prime$} (i2 |- eq.north);
     \draw (i2 |- eq.south) to node {\tiny $\monD f$} (i2 |- tf.north);
   \end{pic}
\end{align}
A lax homomorphism is called a \emph{pseudo-homomorphism} if $\theta_f$
is invertible, and is called a \emph{homomorphism} when $\theta_f$ is identity.
\end{definition}

\subsubsection{The monads $\monL_B$ and $\monR_B$}
\label{sec:LBandRB}

Street defines \cite[pp.118,~122]{Street} two 2-monads for which pseudoalgebra
structure is related to opfibrations and fibrations. This is done on the slice 2-category
$\catC/B$, with objects $(E, p \colon E \rightarrow B)$, arrows
$f \colon E \rightarrow E^\prime$ in $\catC$ such that
$p^\prime f = p$, and 2-cells
$\xymatrix{(E,p)
    \ar@/^1pc/@{->}[r]^{f}
    \ar@/_1pc/@{->}[r]_{g}
   {\ar@{=>}^<<{\scriptstyle \sigma} (11,2); (11,-2)}
  & (E^\prime,p^\prime) }$
in $\catC$ such that $p^\prime\sigma = 1_p$.

The first monad, for opfibrations, is defined on a 2-functor
$\monL_B \colon \catC/B \rightarrow \catC/B$ given by
\begin{align}
  \xymatrix{
    (E,p)
      \ar@/^1pc/@{->}[r]^{f}
      \ar@/_1pc/@{->}[r]_{g}
     {\ar@{=>}^<<{\scriptstyle \sigma} (11,2); (11,-2)}
    &
    (E^\prime,p^\prime)
      \ar@{|->}[r]
    &
    (p/B,d_1)
      \ar@/^1pc/@{->}[r]^{f/B}
      \ar@/_1pc/@{->}[r]_{g/B}
     {\ar@{=>}^<<{\scriptstyle \sigma/B} (59,2); (59,-2)}
    &
    (p^\prime/B,d_1)
    }
\end{align}
Thus a generalized object of $\monL_B E$ is a triple $(e,b,\alpha)$
with $e,b$ objects of $E$ and $B$, and $\alpha\colon pe \to b$.
$\monL_B$ can also be expressed in terms of $\Phi$ and span composition as
$\monL_B(E,p) =(\Phi B\circ p,d_1\hat{p})$,
$\monL(f)=1\circ f$, and $\monL_B(\sigma) = 1\circ \sigma$.
We shall see these in more detailed tangle form when generalized in
Section~\ref{sec:lbullet}.

We now have a diagram
\[
  \xymatrix{
    {\monL_B^2 E} \ar@{->}[d]_{d_1} \ar@{->}[r]^{d_0}
        &    {\monL_B E} \ar@{->}[d]_{d_1} \ar@{->}[r]^{d_0}
        &    E \ar@{->}[d]^{p} \\
    B \ar@{=}[r] & B \ar@{=}[r] & B
    {\ar@{<=}^{\lambda} (4,-8); (12,-8)}
    {\ar@{<=}^{\lambda} (24,-8); (32,-8)}
  }
\]
The unit $i$ and multiplication $c$ of Street's monad can be defined in terms of
tangles as follows.

\begin{definition}\label{def:ic}
The components on $(E,p)$ of $i \colon 1\rightarrow \monL_B$ and
$c \colon \monL_B^2\rightarrow \monL_B$ are defined by
\[
  d_0 i = 1_E \quad d_1 i = p \quad d_0 c = d_0 d_0 \quad d_1 c = d_1
\]
and
\[
  \begin{pic}[auto]  
     \node (i1) at (0,1.5) {$p$};
     \node (i2) at (1,1.5) {$d_0$};
     \node (i3) at (2,1.5) {$i$};
     \node (o1) at (1,0) {$d_1$};
     \node (o2) at (2,0) {$i$};
     \node[2cell,minimum width=1.1cm] (t) at (0.5,0.7) {$\lambda$};
     \draw (i1) to (i1 |- t.north);
     \draw (i2) to (i2 |- t.north);
     \draw (i3) to (o2);
     \draw (o1 |- t.south) to (o1);
   \end{pic}
   =
   \begin{pic}[auto]  
     \node (i1) at (0,1.5) {$p$};
     \node (i2) at (1,1.5) {$d_0$};
     \node (i3) at (2,1.5) {$i$};
     \node (o1) at (0,0) {$d_1$};
     \node (o2) at (2,0) {$i$};
     \node[iso2cell,minimum width=2.1cm] (eq) at (1,0.7) {};
     \draw (i1) to (i1 |- eq.north);
     \draw (i2) to (i2 |- eq.north);
     \draw (i3) to (i3 |- eq.north);
     \draw (o1 |- eq.south) to (o1);
     \draw (o2 |- eq.south) to (o2);
   \end{pic}
   \text{, }\quad
   \begin{pic}[auto]  
     \node (i1) at (0,1.5) {$p$};
     \node (i2) at (1,1.5) {$d_0$};
     \node (i3) at (2,1.5) {$c$};
     \node (o1) at (1,0) {$d_1$};
     \node (o2) at (2,0) {$c$};
     \node[2cell,minimum width=1.1cm] (t) at (0.5,0.7) {$\lambda$};
     \draw (i1) to (i1 |- t.north);
     \draw (i2) to (i2 |- t.north);
     \draw (i3) to (o2);
     \draw (o1 |- t.south) to (o1);
   \end{pic}
   =
   \begin{pic}[auto] 
     \node (i0) at (0,0) {$p$};
     \node (i1) at (1,0) {$d_0$};
     \node (i2) at (2,0) {$c$};
     \node[iso2cell,minimum width=1.1cm] (ieq) at (1.5,-0.7) {};
     \node[2cell,minimum width=1.1cm] (t1) at (0.5,-1.4) {$\lambda$};
     \node[2cell,minimum width=1.1cm] (t2) at (1.5,-2.1) {$\lambda$};
     \node[iso2cell,minimum width=1.1cm] (oeq) at (1.5,-2.8) {};
     \node (o1) at (1,-3.5) {$d_1$};
     \node (o2) at (2,-3.5) {$c$};
     \draw (i0) to (i0 |- t1.north);
     \draw (i1) to (i1 |- ieq.north);
     \draw (i2) to (i2 |- ieq.north);
     \draw (i1 |- ieq.south) to node {\tiny $d_0$} (i1 |- t1.north);
     \draw (i2 |- ieq.south) to node {\tiny $d_0$} (i2 |- t2.north);
     \draw (i1 |- t1.south) to node {\tiny $d_1$} (i1 |- t2.north);
     \draw (i2 |- t2.south) to node {\tiny $d_1$} (i2 |- oeq.north);
     \draw (o1 |- oeq.south) to (o1);
     \draw (o2 |- oeq.south) to (o2);
   \end{pic}
\]
\end{definition}
A generalized object of $\monL^2_B E$ is $(e,b_1,b_2,\alpha,\alpha_1)$
with $\alpha\colon pe \to b_1$ and $\alpha_1 \colon b_1 \to b_2$,
and then $c(e,b_1,b_2,\alpha,\alpha_1) = (e,b_2,\alpha_1 \alpha)$;
also, $i(e) = (e,pe,1)$.

Street shows that $i$ and $c$ are 2-natural transformations fulfilling the monad equations,
and hence making $(\monL_B,i,c)$ a 2-monad on $\catC/B$.
Moreover, although we shall not use this, $\monL_B$ has the `Kock property', that is $c \dashv i\monL_B$
in the 2-functor 2-category $[\catC,\catC]$ with identity counit.

Similarly one defines the monad $\monR_B$ which is the monad $\monL_B$
in the category $\catC^{co}/B$ where the direction of 2-cells are
reversed. This amounts to using comma objects $(B/p)$, arrows $B/f$ and
2-cells $B/\sigma$ in $\catC/B$. The further development is dual.
\begin{definition}\label{def:01fibration}
An arrow $p \colon E \rightarrow B$ is called a \emph{pseudo-opfibration}
(Street: \emph{0-fibration})
when $(E,p)$ supports the structure of an $\monL_B$-pseudoalgebra.
It is an \emph{opfibration} (Street: \emph{normal 0-fibration})
when it supports the structure of a normalized $\monL_B$-pseudoalgebra.

Analogously, $p\colon E \rightarrow B$ is called a \emph{pseudofibration}
(Street: \emph{1-fibration}) or a \emph{fibration} (Street: \emph{normal 1-fibration})
if it supports the corresponding structures for $\monR_B$.

We shall not use this, but an opfibration or fibration is said to \emph{split}
if the normal pseudoalgebra structure can be chosen to be an algebra.
\end{definition}
In terms of generalized objects, the pseudoalgebra structure map
$c\colon\monL_B E \to E$ maps $(e,b,\alpha)$ to some object in the fibre over $b$.
In doing so, it lifts $\alpha \colon pe \to b$
to a fibre map from $(pe)^{\ast}E$ to $b^{\ast}E$.

In practice the pseudo- opfibrations and fibrations are too weak to be useful for us,
as the lifting to fibre maps (reindexing) need not be functorial in any sense.
For good properties capturing the standard notions of fibration we need the
normal versions; for these reindexing is at least pseudofunctorial.
Note that all the notions are inherently cloven,
since the adjoint in the Chevalley criterion chooses supine or prone liftings,
and hence amounts to a cleavage or cocleavage.

\begin{proposition}[Chevalley criterion]\label{prop:Chevalley}
\quad\\*
\begin{enumerate}
\item
  The arrow $p \colon E \rightarrow B$ is a pseudo-opfibration over $B$
  if and only if the arrow $\tilde{p} \colon \Phi E\rightarrow p/B$
  corresponding to the 2-cell
  \begin{align}
  \vcenter{\hsize=3cm\xymatrix{
    \Phi E
    \ar@{->}[r]^{pd_1}
    \ar@{->}[d]_{d_0}
   &
    B
    \ar@{->}[d]^{1}
   \\
    E
    \ar@{->}[r]_{p}
   &
    B
    {\ar@{=>}^<<<<{\scriptstyle p\lambda} (8,-11); (13,-6)}
   }}
  \simeq
  \begin{pic}[auto]
     \node (i1) at (0,1) {$p$};
     \node (i2) at (1,1) {$d_0$};
     \node (o1) at (0,-1) {$1$};
     \node (o2) at (1,-1) {$pd_1$};
     \node[coupon,minimum width=1.1cm,inner sep=2pt] (t) at (0.5,0) {$p\lambda$};
     \draw (i1) to (i1 |- t.north);
     \draw (i2) to (i2 |- t.north);
     \draw (o2 |- t.south) to (o2);
    \end{pic}
  \end{align}
  has a left adjoint with unit an isomorphism.
\item
  The arrow $p$ is an opfibration
  iff $\tilde{p}$ has a left adjoint with unit an identity.
\item
  The corresponding statements hold for pseudofibrations and fibrations,
  replacing $p/B$ by $B/p$ and requiring a right adjoint with counit either an isomorphism
  or an identity.
\end{enumerate}
\end{proposition}
\begin{proof}
  (1) is \cite[proposition 9]{Street}, and (2) is the remarks following it.
  (3) follows by duality on 2-cells.
\end{proof}
\subsection{Liftings of 2-transitions}\label{sec:lifting}
The classical lifting result seems to go back to Applegate~\cite{Applegate},
see \cite{Johnston:AlgLiftings} and \cite{Manes:AlgTheories}. Let
$(\monD_1,i_1,c_1)$ be a monad on the (ordinary) category $\catD_1$ and
$(\monD_2,i_2,c_2)$ be a monad on $\catD_2$. Denote as usual the
category of Eilenberg-Moore algebras as $\catD_1^{\monD_1}$ or
$\monD_1\Alg$ etc. One has forgetful functors
$U_i \colon \catD_i^{\monD_i} \rightarrow \catD_i$ forgetting the algebra structure
map. Given a functor $T \colon \catD_1\rightarrow \catD_2$, then a functor
$\overline{T} \colon \catD_1^{\monD_1} \rightarrow \catD_2^{\monD_2}$ is a \emph{lifting}
of $T$ such that
\begin{align}\label{eq:lifting}
   \vcenter{\hsize=3cm\xymatrix{
    \catD_1^{\monD_1}
    \ar@{->}[r]^{\overline{T}}
    \ar@{->}[d]_{U_1}
   &
    \catD_2^{\monD_2}
    \ar@{->}[d]^{U_2}
   \\
    \catD_1
    \ar@{->}[r]^{T}
   &
    \catD_2
   }}
\end{align}
commutes.
\begin{lemma}[Applegate]
Let $\monD_i$ be monads on categories $\catD_i$ as above, and
$T \colon \catD_1\rightarrow \catD_2$ a functor. Then the functors
$\overline{T} \colon \catD_1^{\monD_1} \rightarrow \catD_2^{\monD_2}$ which are
liftings of $T$ are in 1-to-1 correspondence with natural
transformations $\psi \colon \monD_2T \rightarrow T\monD_1$ such that
the following diagram commutes.
\begin{equation}\label{eq:transitionDia}
  \xymatrix{
     T
     \ar@{->}[r]^{i_2 T}
     \ar@{->}[rd]_{Ti_1}
    &
     \monD_2 T
     \ar@{->}[d]^{\psi}
    &
     \monD_2^2 T
     \ar@{->}[l]_{c_2T}
     \ar@{->}[d]^{\monD_2\psi}
    \\
    &
     T\monD_1
    &
     \monD_2 T\monD_1
     \ar@{->}[d]^{\psi\monD_1}
    \\
    &
    &
     T\monD_1^2
     \ar@{->}[lu]^{Tc_1}
  }
\end{equation}
\end{lemma}
Such pairs $(T,\psi)$ are commonly known as \emph{monad functors} from
$\mathfrak{D}_1$ to $\mathfrak{D}_2$,
but we shall follow~\cite{MarmoWood:CohPsLW} in referring to $\psi$
as a \emph{transition from $\mathfrak{D}_1$ to $\mathfrak{D}_2$ along $T$}.

The commutative diagram~\eqref{eq:transitionDia} can also be written in tangle form as
\begin{equation}\label{eq:transitionTangle}
    \begin{pic}[auto]
      \node (i1) at (1,0) {$T$};
      \node[2cell,minimum width=0cm] (i) at (0,0) {$i_2$};
      \node[2cell,minimum width=1.1cm] (psi) at (0.5,-0.8) {$\psi$};
      \node (o0) at (0,-1.6) {$T$};
      \node (o1) at (1,-1.6) {$\monD_1$};
      \draw (i1) -- (i1 |- psi.north);
      \draw (o0 |- i.south) to node {\tiny $\monD_2$} (o0 |- psi.north);
      \draw (o0 |- psi.south) -- (o0);
      \draw (o1 |- psi.south) -- (o1);
    \end{pic}
    =
    \begin{pic}[auto]
      \node (i0) at (0,0) {$T$};
      \node[2cell,minimum width=0cm] (i) at (1,0) {$i_1$};
      \node (o0) at (0,-1) {$T$};
      \node (o1) at (1,-1) {$\monD_1$};
      \draw (i0) -- (o0);
      \draw (o1 |- i.south) -- (o1);
    \end{pic}
  \text{, }
    \begin{pic}[auto]
      \node (i0) at (0,0) {$\monD_2$};
      \node (i1) at (1,0) {$\monD_2$};
      \node (i2) at (2,0) {$T$};
      \node[2cell,minimum width=1.1cm] (c) at (0.5,-0.7) {$c_2$};
      \node[2cell,minimum width=1.1cm] (psi) at (1.5,-1.5) {$\psi$};
      \node (o1) at (1,-2.2) {$T$};
      \node (o2) at (2,-2.2) {$\monD_1$};
      \draw (i0) -- (i0 |- c.north);
      \draw (i1) -- (i1 |- c.north);
      \draw (i2) -- (i2 |- psi.north);
      \draw (i1 |- c.south) to node {\tiny $\monD_2$} (i1 |- psi.north);
      \draw (o1 |- psi.south) -- (o1);
      \draw (o2 |- psi.south) -- (o2);
    \end{pic}
    =
    \begin{pic}[auto]
      \node (i0) at (0,0) {$\monD_2$};
      \node (i1) at (1,0) {$\monD_2$};
      \node (i2) at (2,0) {$T$};
      \node[2cell,minimum width=1.1cm] (psi1) at (1.5,-0.7) {$\psi$};
      \node[2cell,minimum width=1.1cm] (psi2) at (0.5,-1.5) {$\psi$};
      \node[2cell,minimum width=1.1cm] (c) at (1.5,-2.3) {$c_1$};
      \node (o0) at (0,-3.0) {$T$};
      \node (o1) at (1,-3.0) {$\monD_1$};
      \draw (i0) -- (i0 |- psi2.north);
      \draw (i1) -- (i1 |- psi1.north);
      \draw (i2) -- (i2 |- psi1.north);
      \draw (i1 |- psi1.south) to node {\tiny $T$} (i1 |- psi2.north);
      \draw (i2 |- psi1.south) to node {\tiny $\monD_1$} (i2 |- c.north);
      \draw (i1 |- psi2.south) to node {\tiny $\monD_1$} (i1 |- c.north);
      \draw (o0 |- psi2.south) -- (o0);
      \draw (o1 |- c.south) -- (o1);
    \end{pic}
\end{equation}
Our previous tangle diagrams have been within a 2-category $\catC$.
These here are different in that they are in the 3-category of 2-categories,
and the 3-cells -- which relate to 2-cells in individual 2-categories --
are thus inaccessible without developing a 3-dimensional calculus.
The tangle diagrams have shifted a dimension,
and much of what follows addresses the interplay between the two levels,
and using the 2-dimensional calculus at different levels to capture the whole system.

Given $\psi$, one defines $\overline{T}$ as
$\overline{T}(\monD_1 E\stackrel{c_1}{\rightarrow} E) =
(\monD_2 TE\stackrel{\psi E}{\rightarrow} T\monD_1 E \stackrel{Tc_1}{\rightarrow} TE)$
defining the algebra structure $c_2= c_1\circ\psi$ on $TE$.
We need a similar result for 2-categories and lax- or pseudo-algebras,
which can for the pseudo case be derived from
Marmolejo and Wood~\cite{MarmoWood:CohPsLW}.%
\footnote{The analogous result for lax
algebras was obtained by N. Gambino (private communication).}
That paper deals with pseudomonads, where the three monad equations
hold only up to isomorphism subject to coherence conditions.
Hence it already covers our case of 2-monads.
For two pseudomonads $\mathfrak{D}_i$ on 2-categories $\catD_i$ ($i=1,2$),
and a 2-functor $T \colon\catD_1 \to \catD_2$,
the paper defines the notion of
\emph{transition from $\mathfrak{D}_1$ to $\mathfrak{D}_2$ along $T$}
as a strong transformation (the ``pseudo'' generalization of a 2-natural transformation)
$\psi\colon \mathfrak{D}_2 T \to T \mathfrak{D}_1$
together with two invertible modifications that relax the commutativities in
diagram~\eqref{eq:transitionDia}, subject to some coherence conditions.
\cite{MarmoWood:CohPsLW} goes on to show that the existence of a transition along $T$
implies that $T$ lifts to the pseudoalgebra categories.
(The definition of pseudoalgebra for a pseudomonad appears
in~\cite{Marmolejo:DoctrinesAdjointString}.)

In our situation we have 2-monads, but still need to use pseudoalgebras in order to
connect with Street's result.
We can summarize the restricted results of~\cite{MarmoWood:CohPsLW} as follows.

\begin{definition}\label{def:transition}
  Let $\mathfrak{D}_i = (D_i,i_i,c_i)$ ($i=1,2$) be 2-monads on 2-categories $\catD_i$,
  and let $T \colon\catD_1 \to \catD_2$ be a 2-functor.
  Then a \emph{2-transition from $\mathfrak{D}_1$ to $\mathfrak{D}_2$ along $T$}
  is a 2-natural transformation $\psi\colon \mathfrak{D}_2 T \to T \mathfrak{D}_1$
  such that the equations~\eqref{eq:transitionTangle} hold.
\end{definition}

\begin{proposition}\label{prop:lifting}
  Let $\mathfrak{D}_i = (D_i,i_i,c_i)$ ($i=1,2$) be 2-monads on 2-categories $\catD_i$,
  let $T \colon\catD_1 \to \catD_2$ be a 2-functor,
  and let $\psi$ be a 2-transition from $\mathfrak{D}_1$ to $\mathfrak{D}_2$ along $T$.
  \begin{enumerate}
  \item
    $T$ lifts to a 2-functor between the pseudoalgebra categories,
    $\widehat{T}\colon \catD_1^{\mathfrak{D}_1} \to \catD_2^{\mathfrak{D}_2}$,
    such that diagram~\eqref{eq:lifting} commutes,
    with $\overline{T}$ replaced by $\widehat{T}$.
  \item
    $\widehat{T}$ preserve normality.
  \end{enumerate}
\end{proposition}
\begin{proof}
  (1) is in~\cite{MarmoWood:CohPsLW}.
  For (2), let $(E,c,\zeta=1,\theta)$ be a normal $\mathfrak{D}_1$-pseudoalgebra as
  in Definition~\ref{def:laxAlgebra}.
  Its image under $\widehat{T}$ is carried by $TE$,
  its structure map is the bottom line in the following diagram,
  and its $\zeta$ is got by pasting the square and the triangle.
  The square commutes by~\eqref{eq:transitionDia},
  and $T\zeta$ is the identity because $E$ is normalized.
  \[
     \xymatrix{
       TE
         \ar@{->}[r]^{1}
         \ar@{->}[d]_{i_2 TE}
     & TE
         \ar@{->}[dr]^{1}
         \ar@{->}[d]_{Ti_1 E}
     \\
       \monD_2 TE
         \ar@{->}[r]_{\psi E}
     & T \monD_1 E
         \ar@{->}[r]_{Tc}
     & TE
   {\ar@{=>}^<<{\scriptstyle T\zeta} (29,-7); (25,-11)}
  }
  \]
\end{proof}
\section{The arrow category $\arrC$}\label{sec:arrowcat}
For the rest of the paper $\catC$ will be a representable 2-category.

In this section we summarize some general features of the arrow category $\arrC$.
An object of $\arrC$ is a bundle (an arbitrary morphism in $\catC$),
and a morphism $f$ is a commutative square in which the vertical arrows
are bundles, the domain and codomain of $f$.
We shall write $\upstairs{f}$ and $\downstairs{f}$ for the ``upstairs''
and ``downstairs'' parts of $f$, the horizontal arrows in the square.
Our general presumption is that $\catC$ is in fact a 2-category,
and then $\arrC$ is too.
We use similar notation for the components of a 2-cell $\alpha\colon f\rightarrow g$,
got as the results of applying the 2-functors
$\dom,\cod\colon \arrC \rightarrow \catC$.
The following diagram commutes at the 2-cell level,
in the sense that $p'\upstairs{\alpha} = \downstairs{\alpha}p$.
\[
  \xymatrix@R1.2cm@C1.2cm{
      E
      \ar@{->}[d]^{p}
      \ar@/^1pc/@{->}[r]^{\upstairs{f}}
      \ar@/_1pc/@{->}[r]_{\upstairs{g}}
     {\ar@{=>}^<<{\upstairs{\alpha}} (9,2); (9,-2)}
     &
      E^\prime
      \ar@{->}[d]^{p^\prime}
     \\
      B
      \ar@/^1pc/@{->}[r]^{\downstairs{f}}
      \ar@/_1pc/@{->}[r]_{\downstairs{g}}
     {\ar@{=>}^<<{\downstairs{\alpha}} (9,-15); (9,-19)}
     &
      B^\prime
  }
\]

For any 2-category $\catD$, a 2-functor $F \colon \catD \to \arrC$
can be understood as a pair of functors $\upstairs{F} = \dom F$
and $\downstairs{F}= \cod F$ from $\catD$ to $\catC$,
together with a 2-natural transformation $\stairs{F}\colon \upstairs{F} \to \downstairs{F}$.
Given two such 2-functors, $F$ and $G$,
a 2-natural transformation $\alpha\colon F \to G$ can be understood as a pair of
2-natural transformations
$\upstairs{\alpha} \colon \upstairs{F}\to \upstairs{G}$
and $\downstairs{\alpha} \colon \downstairs{F}\to \downstairs{G}$
such that $\stairs{G}\circ \upstairs{\alpha} = \downstairs{\alpha}\circ\stairs{F}$.
Thus a 2-natural transformation is a commutative square of 2-functors to $\catC$
and 2-natural transformations.

We note some basic 2-functors and 2-natural transformations associated with $\arrC$.
\begin{definition}\label{def:I}
  $\funI\colon \catC\to \arrC$ is given by
  $\upstairs{\funI} = \downstairs{\funI} = \Id_\catC$,
  with $\stairs{\funI}$ the identity 2-natural transformation.
  Thus $\funI B =
  \vcenter{\xymatrix{
    B \ar@{=}[d] \\
    B
  }}
  $.

  The 2-natural transformation $\cc\colon \Id_\arrC \to \funI\cod$
  is defined by
  \[
    \cc\left(
      \vcenter{\xymatrix{
        E \ar@{->}[d]^p \\
        B
      }}
    \right)
    =
      \vcenter{\xymatrix{
        E \ar@{->}[d]^p \ar@{->}[r]^p  &  B \ar@{=}[d]\\
        B \ar@{=}[r]                   &  B
      }}
  \]
\end{definition}

We note immediately a ``yanking'' formula, easily proved:
\begin{equation}\label{eq:ccyank}
  \begin{pic}[auto]
     \node (i0) at (0,0) {$\cod$};
     \node (i1) at (1,0) {};
     \node (o2) at (2,-1) {$\cod$};
     \node[2cell,minimum width=1.1cm] (cc) at (1.5,0) {$\cc$};
     \node[iso2cell,minimum width=1.1cm] (iso) at (0.5,-1) {};
     \draw (i0) -- (i0 |- iso.north);
     \draw (i1 |- cc.south) to node {\tiny $\funI$} (i1 |- iso.north);
     \draw (o2) -- (o2 |- cc.south);
  \end{pic}
  =
  \begin{pic}
     \node (i0) at (0,0) {$\cod$};
     \node (o0) at (0,-1) {$\cod$};
     \draw (i0) -- (o0);
  \end{pic}
\end{equation}

We next remark on the fact that $\cod \colon \arrC \to \catC$ is a bifibration
(fibration and opfibration) in the 2-category $\mathrm{Cat}$.
Here there are more concrete characterizations.

\begin{definition}\label{def:fibprone}
  Let $F\colon\catC\to\catD$ be a functor between categories.
  A morphism $f\colon X \to Y$ in $\catC$ is \emph{prone}%
  \footnote{We have adopted Paul Taylor's terms;
  prone and supine morphisms are also commonly called \emph{cartesian} and \emph{cocartesian}.}
  (with respect to $F$)
  if for every $g\colon Z \to Y$ such that $Fg$ factors via $Ff$,
  as $Fg = (Ff)h'$,
  there is a unique $h\colon Z \to X$ such that
  $g = fh$ and $Fh = h'$.
  \[
    \xymatrix@C=2cm{
      Z
        \ar@{.}[d]
        \ar@{.>}[r]_{\exists ! h}
        \ar@/^1pc/@{->}[rr]^{g}
      & X
        \ar@{->}[r]_{f}
        \ar@{.}[d]
      & Y
        \ar@{.}[d]
    \\
      FZ
        \ar@/_1pc/@{->}[rr]_{Fg}
        \ar@{->}[r]^{h'}
      & FX
        \ar@{->}[r]^{Ff}
      & FY
    }
  \]
  Dually, $f$ is \emph{supine}
  if for every $g\colon X \to Z$ such that $Fg$ factors via $Ff$,
  as $Fg = h'(Ff)$,
  there is a unique $h\colon Y \to Z$ such that
  $g = hf$ and $Fh = h'$.
\end{definition}

Then $F$ is a fibration if for every object $Y$ in $\catC$,
and every morphism $f'\colon X' \to F(Y)$ in $\catD$,
there is a prone morphism $f\colon X \to Y$ such that $Ff = f'$.
$F$ is an opfibration if for every object $Y$,
and every morphism $f'\colon X' \to FY$ in $\catD$,
there is a supine morphism $f \colon X \to Y$ such that $Ff = f'$.

For $\cod\colon\arrC\to\catC$ it is well known that a morphism of $\arrC$
is prone iff, as commutative square in $\catC$, it is a pullback;
and $\cod$ is a fibration if $\catC$ has pullbacks.

If in addition $\catC$ is (as in our situations) a representable 2-category,
then the pullbacks are 2-pullbacks, and the prone morphisms in $\arrC$
also allow lifting of 2-cells:
if we have $(g_i, h'_i)$ ($i = 1,2$) lifting to $h_i$,
and in addition we have $\alpha\colon g_1 \to g_2$ and $\beta'\colon h'_1\to h'_2$
such that $F\alpha = (Ff)\beta'$,
then there is a unique $\beta\colon h_1 \to h_2$ such that
$\alpha = f\beta$ and $F\beta = \beta'$.

It is also easy to show that $f$ is supine with respect to $\cod$
iff $\upstairs{f}$ is an isomorphism, and $\cod$ is always an opfibration.
Trivially, the supine morphisms allow lifting of 2-cells.

We can extend this to $\cod\colon 2\text{-fun}[\catD,\arrC]\to 2\text{-fun}[\catD,\catC]$.
The discussion at the start of this section shows that
$2\text{-fun}[\catD,\arrC]$ is isomorphic to $2\text{-fun}[\catD,\catC]^{\downarrow}$,
so we know that the prone morphisms in $2\text{-fun}[\catD,\arrC]$ correspond to the
pullback squares in $2\text{-fun}[\catD,\catC]$.

\begin{lemma}\label{lem:pullback}
  Let $v \colon \mathfrak{F} \rightarrow \mathfrak{G}$
  be a 2-natural transformations between 2-functors
  $\mathfrak{F}, \mathfrak{G} \colon \catD \to \arrC$,
  and suppose also that for each object $X$ of $\catD$ the square
  \[
    \xymatrix{
      {\upstairs{\mathfrak{F}}X}
            \ar@{->}[d]_{\stairs{\mathfrak{F}} (X)} \ar@{->}[r]^{\upstairs{v}X}
        & {\upstairs{\mathfrak{G}}X} \ar@{->}[d]^{\stairs{\mathfrak{G}} (X)} \\
      {\downstairs{\mathfrak{F}}X} \ar@{->}[r]^{\downstairs{v} X}
        &  \downstairs{\mathfrak{G}}X
              {\ar@{-}(6,-2); (6,-4)}
      {\ar@{-}(4,-4); (6,-4)}
    }
  \]
  is a pullback in $\catC$.

  Then the corresponding square of functors from $\catD$ to $\catC$ is a pullback
  in $2\text{-fun}[\catD,\catC]$, and hence $v$ is prone in $2\text{-fun}[\catD,\arrC]$.

\end{lemma}
\begin{proof}
  If $\mathfrak{H}$ is a 2-functor from $\catD$ to $\catC$, then
  constructing a pullback fill-in from $\mathfrak{H}$ to $\upstairs{\mathfrak{F}}$
  amounts to constructing the components for every $X$,
  which are determined uniquely by the pullback squares in $\catC$.
  Hence we have the uniqueness part of the pullback property.
  Proving that it gives a 2-natural transformation follows routinely from the pullback property.
\end{proof}

Proneness of $v$ in the above lemma amounts to the following:
2-natural transformations $u\colon\mathfrak{H}\to\mathfrak{F}$
are uniquely determined by $\downstairs{u}$ and $vu$.

\section{The monads $\monL_\bullet$ and $\monR_\bullet$ on $\arrC$}\label{sec:lbullet}
In this section we shall define two monads $\monL_\bullet$ and
$\monR_\bullet$ on the arrow 2-category $\arrC$ extending Street's
monads $\monL_B$ and $\monR_B$ on the 2-categories $\catC/B$. This will
allow us to have general base change morphisms for bundles, and not only
bundles over a fixed base $B$ as in the slice 2-category $\catC/B$.
We develop the theory for $\monL_\bullet$ as the $\monR_\bullet$ case
is obtained by working in $\catC^{co}$ with reversed 2-cells.

\begin{definition}\label{def:lbullet}
Let $\arrC$ be the arrow 2-category over the representable 2-category
$\catC$, and let $\monL_B \colon \catC/B\rightarrow \catC/B$ be the 2-monad
defined by Street. We define the 2-functor $\monL_\bullet \colon \arrC \rightarrow \arrC$
on the arrow 2-category as follows:
\begin{align}
  \monL_\bullet\left(
  \vcenter{\xymatrix@R1.2cm@C1.2cm{
      E
      \ar@{->}[d]^{p}
      \ar@/^1pc/@{->}[r]^{\upstairs{f}}
      \ar@/_1pc/@{->}[r]_{\upstairs{g}}
     {\ar@{=>}^<<{\upstairs{\alpha}} (9,2); (9,-2)}
     &
      E^\prime
      \ar@{->}[d]^{p^\prime}
     \\
      B
      \ar@/^1pc/@{->}[r]^{\downstairs{f}}
      \ar@/_1pc/@{->}[r]_{\downstairs{g}}
     {\ar@{=>}^<<{\downstairs{\alpha}} (9,-15); (9,-19)}
     &
      B^\prime
  }}
  \right)
  &=
  \vcenter{\xymatrix@R1.2cm@C1.2cm{
      \monL_B E
      \ar@{->}[d]^{d_1}
      \ar@/^1pc/@{->}[r]^{\monLof{f}}
      \ar@/_1pc/@{->}[r]_{\monLof{g}}
     {\ar@{=>}^<<{\monLof{\alpha}} (10,2); (10,-2)}
     &
      \monL_{B^\prime}E^\prime
      \ar@{->}[d]^{d_1  ^\prime}
     \\
      B
      \ar@/^1pc/@{->}[r]^{\downstairs{f}}
      \ar@/_1pc/@{->}[r]_{\downstairs{g}}
     {\ar@{=>}^<<{\downstairs{\alpha}} (10,-15); (10,-19)}
     &
      B^\prime
  }}
\end{align}
For 0-cells, $\monL_B E$ is as defined in Section~\ref{sec:LBandRB}.
On 1-cells, $\monLof{f}$ is uniquely defined
such that a), b) and c) hold:
\begin{align}\label{eq:nat1cells}
  \text{a):}
  \begin{pic}
     \node (i1) at (0,1) {$d_0^\prime$};
     \node (i2) at (1,1) {$\monLof{f}$};
     \node (o1) at (0,-1) {$\upstairs{f}$};
     \node (o2) at (1,-1) {$d_0$};
     \node[iso2cell] (iso) at (0.5,0) {};
     \draw (i1) -- (i1 |- iso.north);
     \draw (i2) -- (i2 |- iso.north);
     \draw (o1) -- (o1 |- iso.south);
     \draw (o2) -- (o2 |- iso.south);
  \end{pic}
;\quad
  \text{b):}
  \begin{pic}
     \node (i1) at (0,1) {$d_1^\prime$};
     \node (i2) at (1,1) {$\monLof{f}$};
     \node (o1) at (0,-1) {$\downstairs{f}$};
     \node (o2) at (1,-1) {$d_1$};
     \node[iso2cell] (iso) at (0.5,0) {};
     \draw (i1) -- (i1 |- iso.north);
     \draw (i2) -- (i2 |- iso.north);
     \draw (o1) -- (o1 |- iso.south);
     \draw (o2) -- (o2 |- iso.south);
  \end{pic}
;\quad
  \begin{pic}
     \node (i1) at (0,1) {$p^\prime$};
     \node (i2) at (1,1) {$d_0^\prime$};
     \node (i3) at (2,1) {$\monLof{f}$};
     \node[2cell,minimum width=1.1cm] (l) at (0.5,0) {$\lambda^\prime$};
     \node (o1) at (0.5,-1) {$d_1^\prime$};
     \node (o2) at (2,-1) {$\monLof{f}$};
     \draw (i1) -- (i1 |- l.north);
     \draw (i2) -- (i2 |- l.north);
     \draw (o1) -- (o1 |- l.south);
     \draw (i3) -- (o2);
  \end{pic}
  \stackrel{c)}{=}
  \begin{pic}[auto]
     \node (i1) at (0,1.25) {$p^\prime$};
     \node (i2) at (1,1.25) {$d_0^\prime$};
     \node (i3) at (2,1.25) {$\monLof{f}$};
     \node[iso2cell,minimum width=2.1cm] (iso1) at (1,0.6) {};
     \node[2cell,minimum width=1.1cm] (l) at (1.5,0) {$\lambda$};
     \node[iso2cell,minimum width=2.1cm] (iso2) at (1,-0.6) {};
     \node (o1) at (0.5,-1.25) {$d_1^\prime$};
     \node (o2) at (2,-1.25) {$\monLof{f}$};
     \draw (i1) to (i1 |- iso1.north);
     \draw (i1 |- iso1.south) to node {\tiny $\downstairs{f}$} (i1 |- iso2.north);
     \draw (o1) to (o1 |- iso2.south);
     \draw (i2) to (i2 |- iso1.north);
     \draw (i3) to (i3 |- iso1.north);
     \draw (i2 |- iso1.south) to node {\tiny $p$} (i2 |- l.north);
     \draw (i3 |- iso1.south) to node {\tiny $d_0$} (i3 |- l.north);
     \draw (o2 |- l.south) to node {\tiny $d_1$} (o2 |- iso2.north);
     \draw (o2 |- iso2.south) -- (o2);
  \end{pic}
  \text{.}
\end{align}

On 2-cells $\monLof{\alpha}$ is uniquely defined such
that d) and e) hold:
\begin{align}\label{eq:nat2cells}
  \begin{pic}
     \node (i1) at (0,1.25) {$d_0^\prime$};
     \node (i2) at (1,1.25) {$\monLof{f}$};
     \node (o1) at (0,-1.25) {$d_0^\prime$};
     \node (o2) at (1,-1.25) {$\monLof{g}$};
     \node[2cell] (a) at (1,0) {$\monLof{\alpha}$};
     \draw (i1) -- (o1);
     \draw (i2) -- (i2 |- a.north);
     \draw (o2 |- a.south) -- (o2);
  \end{pic}
  \,\,\stackrel{d)}{=}
  \begin{pic}[auto]
     \node (i0) at (0,0) {$d_0^\prime$};
     \node (i1) at (1,0) {$\monLof{f}$};
     \node[iso2cell,minimum width=1.1cm] (eq1) at (0.5,-0.7) {};
     \node[2cell,minimum width=0.5cm] (a) at (0,-1.5) {$\upstairs{\alpha}$};
     \node[iso2cell,minimum width=1.1cm] (eq2) at (0.5,-2.3) {};
     \node (o0) at (0,-3) {$d_0^\prime$};
     \node (o1) at (1,-3) {$\monLof{g}$};
     \draw (i0) -- (i0 |- eq1.north);
     \draw (i1) -- (i1 |- eq1.north);
     \draw (i0 |- eq1.south) to node {\tiny $\upstairs{f}$} (i0 |- a.north);
     \draw (i0 |- a.south) to node {\tiny $\upstairs{g}$} (i0 |- eq2.north);
     \draw (i1 |- eq1.south) to node {\tiny $d_0$} (i1 |- eq2.north);
     \draw (o0 |- eq2.south) to (o0);
     \draw (o1 |- eq2.south) to (o1);
  \end{pic}
  ;\qquad
  \begin{pic}
     \node (i1) at (0,1.25) {$d_1^\prime$};
     \node (i2) at (1,1.25) {$\monLof{f}$};
     \node (o1) at (0,-1.25) {$d_1^\prime$};
     \node (o2) at (1,-1.25) {$\monLof{g}$};
     \node[2cell] (a) at (1,0) {$\monLof{\alpha}$};
     \draw (i1) -- (o1);
     \draw (i2) -- (i2 |- a.north);
     \draw (o2 |- a.south) -- (o2);
  \end{pic}
  \,\,\stackrel{e)}{=}
  \begin{pic}[auto]
     \node (i0) at (0,0) {$d_1^\prime$};
     \node (i1) at (1,0) {$\monLof{f}$};
     \node[iso2cell,minimum width=1.1cm] (eq1) at (0.5,-0.7) {};
     \node[2cell,minimum width=0.5cm] (a) at (0,-1.5) {$\downstairs{\alpha}$};
     \node[iso2cell,minimum width=1.1cm] (eq2) at (0.5,-2.3) {};
     \node (o0) at (0,-3) {$d_1^\prime$};
     \node (o1) at (1,-3) {$\monLof{g}$};
     \draw (i0) -- (i0 |- eq1.north);
     \draw (i1) -- (i1 |- eq1.north);
     \draw (i0 |- eq1.south) to node {\tiny $\downstairs{f}$} (i0 |- a.north);
     \draw (i0 |- a.south) to node {\tiny $\downstairs{g}$} (i0 |- eq2.north);
     \draw (i1 |- eq1.south) to node {\tiny $d_1$} (i1 |- eq2.north);
     \draw (o0 |- eq2.south) to (o0);
     \draw (o1 |- eq2.south) to (o1);
  \end{pic}
  \text{.}
\end{align}
\end{definition}
Note that the object part of $\monL_\bullet$ on an object $(E,p \colon E \rightarrow B)$
is Street's monad $\monL_B$ on the slice $\catC/B$. The extension to
the arrow category includes base change.
\begin{proposition}
The map $\monL_\bullet \colon \arrC\rightarrow \arrC$ of Definition~\ref{def:lbullet}
is a 2-functor.
\end{proposition}
%
\begin{proof}
For functoriality on 1-cells, suppose we have bundle morphisms
$f\colon (E,p) \rightarrow (E',p')$ and $f'\colon (E',p') \rightarrow (E'',p'')$.
We must show that
\[
  \begin{pic}[auto]
     \node (i0) at (0,1.5) {$p''$};
     \node (i1) at (1,1.5) {$d''_0$};
     \node (i2) at (2,1.5) {$\monLof{f'}$};
     \node (i3) at (3,1.5) {$\monLof{f}$};
     \node (o1) at (1,0) {$d''_1$};
     \node (o2) at (2,0) {$\monLof{f'}$};
     \node (o3) at (3,0) {$\monLof{f}$};
     \node[2cell,minimum width=1.1cm] (t) at (0.5,0.7) {$\lambda''$};
     \draw (i0) to (i0 |- t.north);
     \draw (i1) to (i1 |- t.north);
     \draw (i2) to (o2);
     \draw (i3) to (o3);
     \draw (o1 |- t.south) to (o1);
   \end{pic}
   =
  \begin{pic}[auto]
     \node (i0) at (0,0) {$p''$};
     \node (i1) at (1,0) {$d''_0$};
     \node (i2) at (2,0) {$\monLof{f'}$};
     \node (i3) at (3,0) {$\monLof{f}$};
     \node[iso2cell,minimum width=3.1cm] (ieq) at (1.5,-0.7) {};
     \node[2cell,minimum width=1.1cm] (t) at (2.5,-1.5) {$\lambda$};
     \node[iso2cell,minimum width=3.1cm] (oeq) at (1.5,-2.3) {};
     \node (o1) at (1,-3) {$d''_1$};
     \node (o2) at (2,-3) {$\monLof{f'}$};
     \node (o3) at (3,-3) {$\monLof{f}$};
     \draw (i0) to (i0 |- ieq.north);
     \draw (i1) to (i1 |- ieq.north);
     \draw (i2) to (i2 |- ieq.north);
     \draw (i3) to (i3 |- ieq.north);
     \draw (o1 |- oeq.south) to (o1);
     \draw (o2 |- oeq.south) to (o2);
     \draw (o3 |- oeq.south) to (o3);
     \draw (i0 |- ieq.south) to node {\tiny $\downstairs{f'}$} (i0 |- oeq.north);
     \draw (i1 |- ieq.south) to node {\tiny $\downstairs{f}$} (i1 |- oeq.north);
     \draw (i2 |- ieq.south) to node {\tiny $p$} (i2 |- t.north);
     \draw (i3 |- ieq.south) to node {\tiny $d_0$} (i3 |- t.north);
     \draw (i3 |- t.south) to node {\tiny $d_1$} (i3 |- oeq.north);
   \end{pic}
\]
This is straightforwardly shown by applying the definition of $\monLof{f}$ and
$\monLof{f'}$. Functoriality on 2-cells is easier.
\end{proof}

\begin{proposition}\label{prop:d0L}
  There is a 2-natural transformation $d_0 \colon \dom\monL_\bullet \to \dom$
  whose component at $p\colon E \to B$ is $d_0 \colon \monL_B E \to E$.
\end{proposition}
\begin{proof}
  The 2-naturality follows from equations~\eqref{eq:nat1cells}~(a)
  and~\eqref{eq:nat2cells}~(d).
\end{proof}

We now define the components of $i \colon 1\rightarrow \monL_\bullet$ and
$c \colon \monL_\bullet^2\rightarrow \monL_\bullet$ exactly as in
Definition~\ref{def:ic}.

\begin{proposition}
We obtain a 2-monad $(\monL_\bullet,i,c)$ on $\arrC$.
\end{proposition}
\begin{proof}
The monad equations for $i$ and $c$ are applied slicewise,
in other words through their restrictions to Street's monad $\monL_B$ on $\catC/B$,
and so are already known to hold.
It remains to show that $i,c$ are 2-natural
transformations on $\arrC$, hence work as expected under base change.
We give the proof for $c$, which is the more complicated:
we show that $c'\upstairs{\monL_\bullet^2 f} = \upstairs{\monL_\bullet f}c$.

\[
  \begin{pic}[auto]
     \node (i0) at (0,1.5) {$p^\prime$};
     \node (i1) at (1,1.5) {$d^\prime_0$};
     \node (i2) at (2,1.5) {$c^\prime$};
     \node (i3) at (3,1.5) {$\upstairs{\monL_\bullet ^2 f}$};
     \node (o1) at (1,0) {$d'_1$};
     \node (o2) at (2,0) {$c^\prime$};
     \node (o3) at (3,0) {$\upstairs{\monL_\bullet ^2 f}$};
     \node[2cell,minimum width=1.1cm] (t) at (0.5,0.7) {$\lambda'$};
     \draw (i0) -- (i0 |- t.north);
     \draw (i1) -- (i1 |- t.north);
     \draw (i2) -- (o2);
     \draw (i3) -- (o3);
     \draw (i1 |- t.south) to (o1);
  \end{pic}
  =
  \begin{pic}[auto]
     \node (i0) at (0,0) {$p^\prime$};
     \node (i1) at (1,0) {$d^\prime_0$};
     \node (i2) at (2,0) {$c^\prime$};
     \node (i3) at (3,0) {$\upstairs{\monL_\bullet ^2 f}$};
     \node[iso2cell,minimum width=2.1cm] (ieq) at (1,-0.7) {};
     \node[2cell,minimum width=1.1cm] (t1) at (0.5,-1.5) {$\lambda'$};
     \node[2cell,minimum width=1.1cm] (t2) at (1.5,-2.3) {$\lambda'$};
     \node[iso2cell,minimum width=1.1cm] (oeq) at (1.5,-3) {};
     \node (o1) at (1,-3.7) {$d'_1$};
     \node (o2) at (2,-3.7) {$c^\prime$};
     \node (o3) at (3,-3.7) {$\upstairs{\monL_\bullet ^2 f}$};
     \draw (i0) -- (i0 |- ieq.north);
     \draw (i1) -- (i1 |- ieq.north);
     \draw (i2) -- (i2 |- ieq.north);
     \draw (i3) -- (o3);
     \draw (i0 |- ieq.south) to node {\tiny $p'$} (i0 |- t1.north);
     \draw (i1 |- ieq.south) to node {\tiny $d'_0$} (i1 |- t1.north);
     \draw (i2 |- ieq.south) to node {\tiny $d'_0$} (i2 |- t2.north);
     \draw (i1 |- t1.south) to node {\tiny $d'_1$} (i1 |- t2.north);
     \draw (i2 |- t2.south) to node {\tiny $d'_1$} (i2 |- oeq.north);
     \draw (o1 |- oeq.south) to (o1);
     \draw (o2 |- oeq.south) to (o2);
  \end{pic}
  =
  \begin{pic}[auto]
     \node (i0) at (0,0) {$p^\prime$};
     \node (i1) at (1,0) {$d^\prime_0$};
     \node (i2) at (2,0) {$c^\prime$};
     \node (i3) at (3,0) {$\upstairs{\monL_\bullet ^2 f}$};
     \node[iso2cell,minimum width=3.1cm] (ieq) at (1.5,-0.7) {};
     \node[2cell,minimum width=1.1cm] (t1) at (0.5,-1.5) {$\lambda'$};
     \node[iso2cell,minimum width=2.1cm] (eq1) at (2,-2.3) {};
     \node[2cell,minimum width=1.1cm] (t2) at (1.5,-3.1) {$\lambda'$};
     \node[iso2cell,minimum width=2.1cm] (oeq) at (2,-3.9) {};
     \node (o1) at (1,-4.6) {$d'_1$};
     \node (o2) at (2,-4.6) {$c^\prime$};
     \node (o3) at (3,-4.6) {$\upstairs{\monL_\bullet ^2 f}$};
     \draw (i0) -- (i0 |- ieq.north);
     \draw (i1) -- (i1 |- ieq.north);
     \draw (i2) -- (i2 |- ieq.north);
     \draw (i3) -- (i3 |- ieq.north);
     \draw (i0 |- ieq.south) to node {\tiny $p'$} (i0 |- t1.north);
     \draw (i1 |- ieq.south) to node {\tiny $d'_0$} (i1 |- t1.north);
     \draw (i2 |- ieq.south) to node {\tiny $\upstairs{\monL_\bullet f}$} (i2 |- eq1.north);
     \draw (i3 |- ieq.south) to node {\tiny $d_0$} (i3 |- eq1.north);
     \draw (i1 |- t1.south) to node {\tiny $d'_1$} (i1 |- eq1.north);
     \draw (i1 |- eq1.south) to node {\tiny $d'_1$} (i1 |- t2.north);
     \draw (i2 |- eq1.south) to node {\tiny $d'_0$} (i2 |- t2.north);
     \draw (i3 |- eq1.south) to node {\tiny $\upstairs{\monL^2_\bullet f}$} (i3 |- oeq.north);
     \draw (i2 |- t2.south) to node {\tiny $d'_1$} (i2 |- oeq.north);
     \draw (o1 |- oeq.south) to (o1);
     \draw (o2 |- oeq.south) to (o2);
     \draw (o3 |- oeq.south) to (o3);
  \end{pic}
\]
\[
  =
  \begin{pic}[auto]
     \node (i0) at (0,0) {$p^\prime$};
     \node (i1) at (1,0) {$d^\prime_0$};
     \node (i2) at (2,0) {$c^\prime$};
     \node (i3) at (3,0) {$\upstairs{\monL_\bullet ^2 f}$};
     \node[iso2cell,minimum width=3.1cm] (ieq) at (1.5,-0.7) {};
     \node[2cell,minimum width=1.1cm] (t1) at (1.5,-1.5) {$\lambda$};
     \node[iso2cell,minimum width=0.5cm] (eq1) at (0,-1.9) {};
     \node[2cell,minimum width=1.1cm] (t2) at (2.5,-2.3) {$\lambda$};
     \node[iso2cell,minimum width=3.1cm] (oeq) at (1.5,-3) {};
     \node (o1) at (1,-3.7) {$d'_1$};
     \node (o2) at (2,-3.7) {$c^\prime$};
     \node (o3) at (3,-3.7) {$\upstairs{\monL_\bullet ^2 f}$};
     \draw (i0) -- (i0 |- ieq.north);
     \draw (i1) -- (i1 |- ieq.north);
     \draw (i2) -- (i2 |- ieq.north);
     \draw (i3) -- (i3 |- ieq.north);
     \draw (i0 |- ieq.south) to node {\tiny $\downstairs{f}$} (i0 |- eq1.north);
     \draw (i0 |- eq1.south) to node {\tiny $\downstairs{\monL_\bullet f}$}
                (i0 |- oeq.north);
     \draw (i1 |- ieq.south) to node {\tiny $p$} (i1 |- t1.north);
     \draw (i2 |- ieq.south) to node {\tiny $d_0$} (i2 |- t1.north);
     \draw (i3 |- ieq.south) to node {\tiny $d_0$} (i3 |- t2.north);
     \draw (i2 |- t1.south) to node {\tiny $d_1$} (i2 |- t2.north);
     \draw (i3 |- t2.south) to node {\tiny $d_1$} (i3 |- oeq.north);
     \draw (o1 |- oeq.south) to (o1);
     \draw (o2 |- oeq.south) to (o2);
     \draw (o3 |- oeq.south) to (o3);
  \end{pic}
  =
  \begin{pic}[auto]
     \node (i0) at (0,0) {$p^\prime$};
     \node (i1) at (1,0) {$d^\prime_0$};
     \node (i2) at (2,0) {$c^\prime$};
     \node (i3) at (3,0) {$\upstairs{\monL_\bullet ^2 f}$};
     \node[iso2cell,minimum width=3.1cm] (ieq) at (1.5,-0.7) {};
     \node[2cell,minimum width=1.1cm] (t) at (1.5,-1.5) {$\lambda$};
     \node[iso2cell,minimum width=3.1cm] (oeq) at (1.5,-2.3) {};
     \node (o1) at (1,-3) {$d'_1$};
     \node (o2) at (2,-3) {$c^\prime$};
     \node (o3) at (3,-3) {$\upstairs{\monL_\bullet ^2 f}$};
     \draw (i0) -- (i0 |- ieq.north);
     \draw (i1) -- (i1 |- ieq.north);
     \draw (i2) -- (i2 |- ieq.north);
     \draw (i3) -- (i3 |- ieq.north);
     \draw (i0 |- ieq.south) to node {\tiny $\downstairs{f}$} (i0 |- oeq.north);
     \draw (i1 |- ieq.south) to node {\tiny $p$} (i1 |- t.north);
     \draw (i2 |- ieq.south) to node {\tiny $d_0$} (i2 |- t.north);
     \draw (i3 |- ieq.south) to node {\tiny $c$} (i3 |- oeq.north);
     \draw (i2 |- t.south) to node {\tiny $d_1$} (i2 |- oeq.north);
     \draw (o1 |- oeq.south) to (o1);
     \draw (o2 |- oeq.south) to (o2);
     \draw (o3 |- oeq.south) to (o3);
  \end{pic}
  =
  \begin{pic}[auto]
     \node (i0) at (0,0) {$p^\prime$};
     \node (i1) at (1,0) {$d^\prime_0$};
     \node (i2) at (2,0) {$c^\prime$};
     \node (i3) at (3,0) {$\upstairs{\monL_\bullet ^2 f}$};
     \node[iso2cell,minimum width=3.1cm] (ieq) at (1.5,-0.7) {};
     \node[2cell,minimum width=1.1cm] (t) at (0.5,-1.5) {$\lambda'$};
     \node[iso2cell,minimum width=2.1cm] (oeq) at (2,-2.3) {};
     \node (o1) at (1,-3) {$d'_1$};
     \node (o2) at (2,-3) {$c^\prime$};
     \node (o3) at (3,-3) {$\upstairs{\monL_\bullet ^2 f}$};
     \draw (i0) -- (i0 |- ieq.north);
     \draw (i1) -- (i1 |- ieq.north);
     \draw (i2) -- (i2 |- ieq.north);
     \draw (i3) -- (i3 |- ieq.north);
     \draw (i0 |- ieq.south) to node {\tiny $p'$} (i0 |- t.north);
     \draw (i1 |- ieq.south) to node {\tiny $d'_0$} (i1 |- t.north);
     \draw (i2 |- ieq.south) to node {\tiny $\upstairs{\monL_\bullet f}$} (i2 |- oeq.north);
     \draw (i3 |- ieq.south) to node {\tiny $c$} (i3 |- oeq.north);
     \draw (i1 |- t.south) to node {\tiny $d'_1$} (i1 |- oeq.north);
     \draw (o1 |- oeq.south) to (o1);
     \draw (o2 |- oeq.south) to (o2);
     \draw (o3 |- oeq.south) to (o3);
  \end{pic}
\]
\end{proof}

\begin{proposition}\label{prop:Lbulletpsalg}
  Let $\catC$ be representable category.
  Then a 1-cell $p\colon E \rightarrow B$ is a pseudo-opfibration iff,
  as an object in $\arrC$, it can be equipped with pseudoalgebra structure
  $(c,\zeta,\theta)$ for $\monL_\bullet$ such that $\downstairs{c} = 1_B$.
  It is an opfibration iff, in addition,
  the pseudoalgebra structure can be chosen to be normalized ($\zeta$ is an identity 2-cell).
\end{proposition}
\begin{proof}
  The condition $\downstairs{c} = 1_B$ is what is needed to have the
  $\monL_\bullet$-pseudoalgebra structure restrict to
  $\monL_B$-pseudoalgebra structure.
  All the pseudoalgebra conditions then live entirely inside the fibre of
  $\arrC$ over $B$,
  and we have reduced to Proposition~\ref{prop:Chevalley}.
\end{proof}

Dually, $p$ is a pseudofibration iff it has $\monR_\bullet$-pseudoalgebra structure
with $\downstairs{c} = 1_B$,
a fibration iff the pseudoalgebra can be chosen to be normalized.

\section{The 2-functors $\funK_n$}\label{sec:Kfunctor}

In this section we introduce a family of 2-functors $\funK_{n}\colon \arrC \rightarrow \arrC$.
Upstairs they are equal to $\monL_\bullet^{n}$, but their downstairs parts are not the identity.
The reason for doing this arises in Section~\ref{sec:fibpres},
where the definition of the natural transformation $\Psi_\bullet$ relies on an isomorphism
upstairs that does not correspond to one downstairs.
By modifying the downstairs part we can use natural isomorphisms between endofunctors on $\arrC$,
and this makes it much easier to calculate with the 2-dimensional calculus.

Note in the following that it is immediate from the definition that $\Phi B = \monL_B B$.
In fact, by composing the 2-functor $\upstairs{\monL_\bullet}\colon \arrC\rightarrow\catC$
with the 2-functor $\funI\colon\catC\rightarrow\arrC$ (Definition~\ref{def:I}),
we obtain the 2-functor $\Phi\colon \catC\rightarrow\catC$.
Now from the morphism
$\left(\begin{matrix} p \\ B \end{matrix}\right)\colon (E,p) \rightarrow (B,1)$
we get $\monL_B p\colon \monL_B E \rightarrow \Phi B$.

\begin{definition}\label{def:Kn}
For each non-negative integer $n$,
we define a 2-endofunctor $\funK_n \colon \arrC \rightarrow \arrC$.
As a 2-natural transformation between 2-functors to $\catC$, it is defined as
$\dom\monL_\bullet^n \cc$,
\[
  \funK_n = \begin{pic}[auto]
     \node (i0) at (0,0) {$\dom$};
     \node (i1) at (1,0) {$\monL_\bullet^n$};
     \node (o0) at (0,-1) {$\dom$};
     \node (o1) at (1,-1) {$\monL_\bullet^n$};
     \node (o2) at (2,-1) {$\funI$};
     \node (o3) at (3, -1) {$\cod$};
     \node[2cell,minimum width=1.1cm] (cc) at (2.5,0) {$\cc$};
     \draw (i0) -- (o0);
     \draw (i1) -- (o1);
     \draw (o2) -- (o2 |- cc.south);
     \draw (o3) -- (o3 |- cc.south);
  \end{pic}
\]
\end{definition}

With notation as in Definition~\ref{def:lbullet}, $\funK_n$ can be calculated as follows.
\begin{align}
  \funK_n\left(
  \vcenter{\xymatrix@R1.2cm@C1.2cm{
      E
      \ar@{->}[d]^{p}
      \ar@/^1pc/@{->}[r]^{\upstairs{f}}
      \ar@/_1pc/@{->}[r]_{\upstairs{g}}
     {\ar@{=>}^<<{\upstairs{\alpha}} (9,2); (9,-2)}
     &
      E^\prime
      \ar@{->}[d]^{p^\prime}
     \\
      B
      \ar@/^1pc/@{->}[r]^{\downstairs{f}}
      \ar@/_1pc/@{->}[r]_{\downstairs{g}}
     {\ar@{=>}^<<{\downstairs{\alpha}} (9,-15); (9,-19)}
     &
      B^\prime
  }}
  \right)
  &=
  \vcenter{\xymatrix@R1.2cm@C1.2cm{
      \monL_B^n E
      \ar@{->}[d]^{\monL_B^n p}
      \ar@/^1pc/@{->}[r]^{\monL_{\downstairs{f}}^n \upstairs{f}}
      \ar@/_1pc/@{->}[r]_{\monL_{\downstairs{g}}^n \upstairs{g}}
     {\ar@{=>}^<<{\monL_{\downstairs{\alpha}}^n \upstairs{\alpha}} (10,2); (10,-2)}
     &
      \monL_{B^\prime}^n E^\prime
      \ar@{->}[d]^{\monL_{B'}^n p'}
     \\
      \monL_B^n B
      \ar@/^1pc/@{->}[r]^{\monL_{\downstairs{f}}^n \downstairs{f}}
      \ar@/_1pc/@{->}[r]_{\monL_{\downstairs{g}}^n \downstairs{g}}
     {\ar@{=>}^<<{\monL_{\downstairs{\alpha}}^n \downstairs{\alpha}} (10,-16); (10,-20)}
     &
      \monL_{B^\prime}^n B^\prime
  }}
\end{align}

We shall in fact only need $\funK_1$ and $\funK_2$
(and note that $\funK_0$ is the identity endofunctor).
However, some results have uniform proofs for all $n$.

\begin{definition}\label{def:d0K}
  We define a 2-natural transformation $\dzeroK \colon \funK_{n+1} \to \funK_n$ from the following
  commutative square of 2-functors from $\arrC$ to $\catC$.
  \[
    \xymatrix@C=2cm{
      {\dom\monL_\bullet\monL_\bullet^n}
          \ar@{->}_{\dom\monL_\bullet\monL_\bullet^n \cc}[d]
          \ar@{->}^{d_0 \monL_\bullet^n}[r]
      & {\dom\monL_\bullet^n}
          \ar@{->}[d]_{\dom\monL_\bullet^n \cc} \\
      {\dom\monL_\bullet\monL_\bullet^n \funI\cod}
          \ar@{->}_{d_0\monL_\bullet^n \funI\cod}[r]
      &  {\dom\monL_\bullet^n \funI\cod}
    }
  \]
  (The $d_0$ in the diagram is that of Proposition~\ref{prop:d0L}.)
\end{definition}

\begin{proposition}\label{prop:LBp}
  In $\catC/B$ let $p \colon E \to B$ and $p' \colon E' \to B$ be objects,
  and let $k$ be a morphism from the first to the second over $B$.
  Then the following square is a pullback in $\catC$.
  \[
    \xymatrix{
      {\monL_B E} \ar@{->}[d]_{\monL_B k} \ar@{->}[r]^{d_0}
        & E \ar@{->}[d]^{k} \\
      {\monL_{B} E'} \ar@{->}[r]^{d'_0}  &  E'
      {\ar@{-}(6,-2); (6,-4)}
      {\ar@{-}(4,-4); (6,-4)}
    }
  \]
\end{proposition}
\begin{proof}
  First, note that $\monL_B k$ is uniquely determined by the properties that
  $d'_0 (\monL_B k) = k d_0$, $d'_1 (\monL_B k) = d_1$, and
  \begin{equation}\label{eq:LBp1}
    \begin{pic}
      \node (i0) at (0,0) {$p'$};
      \node (i1) at (1,0) {$d'_0$};
      \node (i2) at (2,0) {$\monL_B k$};
      \node[2cell,minimum width=1.1cm] (lambda) at (0.5,-0.7) {$\lambda'$};
      \node (o1) at (1,-1.4) {$d'_1$};
      \node (o2) at (2,-1.4) {$\monL_B k$};
      \draw (i0) -- (i0 |- lambda.north);
      \draw (i1) -- (i1 |- lambda.north);
      \draw (i1 |- lambda.south) -- (o1);
      \draw (i2) -- (o2);
    \end{pic}
    =
    \begin{pic}[auto]
      \node (i0) at (0,0) {$p'$};
      \node (i1) at (1,0) {$d'_0$};
      \node (i2) at (2,0) {$\monL_B k$};
      \node[iso2cell,minimum width=2.1cm] (iso1) at (1,-0.7) {};
      \node[2cell,minimum width=1.1cm] (lambda) at (0.5,-1.4) {$\lambda$};
      \node[iso2cell,minimum width=1.1cm] (iso2) at (1.5,-2.1) {};
      \node (o1) at (1,-2.8) {$d'_1$};
      \node (o2) at (2,-2.8) {$\monL_B k$};
      \draw (i0) -- (i0 |- iso1.north);
      \draw (i1) -- (i1 |- iso1.north);
      \draw (i2) -- (i2 |- iso1.north);
      \draw (i0 |- iso1.south) to node {\tiny $p$} (i0 |- lambda.north);
      \draw (i1 |- iso1.south) to node {\tiny $d_0$} (i1 |- lambda.north);
      \draw (i1 |- lambda.south) to node {\tiny $d_1$} (i1 |- iso2.north);
      \draw (i1 |- iso2.south) -- (o1);
      \draw (i2 |- iso2.south) -- (o2);
    \end{pic}
  \end{equation}

  Now suppose we have an object $X$ with morphisms $f\colon  X \rightarrow E$,
  $g\colon X \rightarrow \monL_{B} E'$ such that $kf = d'_0 g$.
  We require a unique morphism $h\colon X \to \monL_B E$ satisfying the pullback conditions
  $d_0 h = f$ and $(\monL_B k)h = g$.
  We shall show that these conditions are equivalent to the conditions
  \begin{subequations} \label{eq:LBp2}
    \begin{align}
      & d_0 h = f        \label{eq:LBp2a}\\
      & d_1 h = d'_1 g   \label{eq:LBp2b}\\
      & \begin{pic}
      \node (i0) at (0,0) {$p'$};
      \node (i1) at (1,0) {$d'_0$};
      \node (i2) at (2,0) {$g$};
      \node[2cell,minimum width=1.1cm] (lambda) at (0.5,-0.7) {$\lambda'$};
      \node (o1) at (1,-1.4) {$d'_1$};
      \node (o2) at (2,-1.4) {$g$};
      \draw (i0) -- (i0 |- lambda.north);
      \draw (i1) -- (i1 |- lambda.north);
      \draw (i1 |- lambda.south) -- (o1);
      \draw (i2) -- (o2);
    \end{pic}
    =
    \begin{pic}[auto]
      \node (i0) at (0,0) {$p'$};
      \node (i1) at (1,0) {$d'_0$};
      \node (i2) at (2,0) {$g$};
      \node[iso2cell,minimum width=2.1cm] (iso1) at (1,-0.7) {};
      \node[2cell,minimum width=1.1cm] (lambda) at (0.5,-1.4) {$\lambda$};
      \node[iso2cell,minimum width=1.1cm] (iso2) at (1.5,-2.1) {};
      \node (o1) at (1,-2.8) {$d'_1$};
      \node (o2) at (2,-2.8) {$g$};
      \draw (i0) -- (i0 |- iso1.north);
      \draw (i1) -- (i1 |- iso1.north);
      \draw (i2) -- (i2 |- iso1.north);
      \draw (i0 |- iso1.south) to node {\tiny $p$} (i0 |- lambda.north);
      \draw (i1 |- iso1.south) to node {\tiny $d_0$} (i1 |- lambda.north);
      \draw (i2 |- iso1.south) to node {\tiny $h$} (i2 |- iso2.north);
      \draw (i1 |- lambda.south) to node {\tiny $d_1$} (i1 |- iso2.north);
      \draw (i1 |- iso2.south) -- (o1);
      \draw (i2 |- iso2.south) -- (o2);
    \end{pic}
                    \label{eq:LBp2c}
    \end{align}
  \end{subequations}
  Note that~\eqref{eq:LBp1} is just the case of~\eqref{eq:LBp2c}
  where $g$ is replaced by $\monL_B k$ and $h$ by the identity 1-cell.
  By inverting the identity 2-cells on the right hand side of~\eqref{eq:LBp2c}
  we see that the equations~\eqref{eq:LBp2}
  suffice to define $h$ uniquely using the fact that $\monL_B E$ is a comma object.

  Assuming the pullback equations, we can derive~\eqref{eq:LBp2} by substituting
  $g = (\monL_B k)h$ and using~\eqref{eq:LBp1}.
  For the converse,
  we use the fact that $\monL_{B} E'$ is a comma object to prove $(\monL_B k)h = g$.
  We have $d'_0 g = kf = k d_0 h = d'_0 (\monL_B k) h$
  and $d'_1 g = d_1 h = d'_1 (\monL_B k) h$.
  It remains to show $\lambda' g = \lambda' (\monL_B k) h$, for which we calculate
  \[
    \begin{pic}[auto]
      \node (i0) at (0,0) {$p'$};
      \node (i1) at (1,0) {$d'_0$};
      \node (i2) at (2,0) {$\monL_B k$};
      \node (i3) at (3,0) {$h$};
      \node[iso2cell,minimum width=2.1cm] (iso1) at (2,-0.7) {};
      \node[2cell,minimum width=1.1cm] (lambda) at (0.5,-1.4) {$\lambda'$};
      \node[iso2cell,minimum width=2.1cm] (iso2) at (2,-2.1) {};
      \node (o1) at (1,-2.8) {$d'_1$};
      \node (o2) at (2,-2.8) {$\monL_B k$};
      \node (o3) at (3,-2.8) {$h$};
      \draw (i0) -- (i0 |- lambda.north);
      \draw (i1) -- (i1 |- iso1.north);
      \draw (i2) -- (i2 |- iso1.north);
      \draw (i3) -- (i3 |- iso1.north);
      \draw (i1 |- iso1.south) to node {\tiny $d'_0$} (i1 |- lambda.north);
      \draw (i3 |- iso1.south) to node {\tiny $g$} (i3 |- iso2.north);
      \draw (i1 |- lambda.south) to node {\tiny $d'_1$} (i1 |- iso2.north);
      \draw (i1 |- iso2.south) -- (o1);
      \draw (i2 |- iso2.south) -- (o2);
      \draw (i3 |- iso2.south) -- (o3);
    \end{pic}
    =
    \begin{pic}[auto]
      \node (i0) at (0,0) {$p'$};
      \node (i1) at (1,0) {$d'_0$};
      \node (i2) at (2,0) {$\monL_B k$};
      \node (i3) at (3,0) {$h$};
      \node[iso2cell,minimum width=2.1cm] (iso1) at (1,-0.7) {};
      \node[2cell,minimum width=1.1cm] (lambda) at (0.5,-1.4) {$\lambda$};
      \node[iso2cell,minimum width=1.1cm] (iso2) at (1.5,-2.1) {};
      \node (o1) at (1,-2.8) {$d'_1$};
      \node (o2) at (2,-2.8) {$\monL_B k$};
      \node (o3) at (3,-2.8) {$h$};
      \draw (i0) -- (i0 |- iso1.north);
      \draw (i1) -- (i1 |- iso1.north);
      \draw (i2) -- (i2 |- iso1.north);
      \draw (i3) -- (o3);
      \draw (i0 |- iso1.south) to node {\tiny $p$} (i0 |- lambda.north);
      \draw (i1 |- iso1.south) to node {\tiny $d_0$} (i1 |- lambda.north);
      \draw (i1 |- lambda.south) to node {\tiny $d_1$} (i1 |- iso2.north);
      \draw (i1 |- iso2.south) -- (o1);
      \draw (i2 |- iso2.south) -- (o2);
    \end{pic}
    =
    \begin{pic}[auto]
      \node (i0) at (0,0) {$p'$};
      \node (i1) at (1,0) {$d'_0$};
      \node (i2) at (2,0) {$\monL_B k$};
      \node (i3) at (3,0) {$h$};
      \node[2cell,minimum width=1.1cm] (lambda') at (0.5,-1.4) {$\lambda'$};
      \node (o1) at (1,-2.8) {$d'_1$};
      \node (o2) at (2,-2.8) {$\monL_B k$};
      \node (o3) at (3,-2.8) {$h$};
      \draw (i0) -- (i0 |- lambda.north);
      \draw (i1) -- (i1 |- lambda.north);
      \draw (i2) -- (o2);
      \draw (i3) -- (o3);
      \draw (i1 |- lambda.south) -- (o1);
    \end{pic}
  \]
\end{proof}

\begin{corollary}\label{cor:d0Fprone}
  For any 2-functor $\mathfrak{F} \colon \catD \to \arrC$,
  we have that $\dzeroK\mathfrak{F}\colon\funK_{n+1}\mathfrak{F}\to\funK_n\mathfrak{F}$
  satisfies the conditions of Lemma~\ref{lem:pullback} and hence is prone.
\end{corollary}
\begin{proof}
  Take Proposition~\ref{prop:LBp}, substituting $\monL_\bullet^n \mathfrak{F}X$ for $p$,
  $\monL_\bullet^n\funI\cod \mathfrak{F}X$ for $p'$, and $\monL_\bullet^n \cc\mathfrak{F}X$ for $k$.
  Then the pullback diagram there is that of Definition~\ref{def:d0K},
  applied to $\mathfrak{F}X$.
  We can now apply Lemma~\ref{lem:pullback}.
\end{proof}

\begin{definition}\label{def:d1K}
We define a 2-natural transformation $d_1 \colon \funK_{n+1} \to \funK_n \monL_\bullet$ as follows.
Consider a square of 2-endomorphisms on $\arrC$,
\begin{equation}\label{eq:d1K}
  \xymatrix{
      {\monL_\bullet} \ar@{->}_{\monL_\bullet \cc}[d] \ar@{=}[r]
      & {\monL_\bullet} \ar@{->}^{\cc \monL_\bullet}[d] \\
      {\monL_\bullet \funI\cod} \ar@{->}[r]  &  {\funI\cod\monL_\bullet}
  }
\end{equation}
The 2-natural transformation on the bottom is $\cc \monL_\bullet \funI \cod$:
\[
  \begin{pic}[auto]
     \node (i1) at (1,0.7) {};
     \node (i2) at (2,0.7) {$\monL_\bullet$};
     \node (i3) at (3,0.7) {$\funI$};
     \node (i4) at (4,0.7) {$\cod$};
     \node (o0) at (0,-0.7) {$\funI$};
     \node (o4) at (4,-0.7) {$\cod$};
     \node (o5) at (5,-0.7) {$\monL_\bullet$};
     \node[2cell,minimum width=1.1cm] (cc) at (0.5,0.7) {$\cc$};
     \node[iso2cell,minimum width = 1.1cm] (codL) at (1.5,0) {};
     \node[iso2cell,minimum width = 2.1cm] (codI) at (2,-0.7) {};
     \node[iso2cell,minimum width = 1.1cm] (codL2) at (4.5,0) {};
     \draw (o0) -- (o0 |- cc.south);
     \draw (i2) -- (i2 |- codL.north);
     \draw (i3) -- (i3 |- codI.north);
     \draw (i4) -- (i4 |- codL2.north);
     \draw (o4) -- (o4 |- codL2.south);
     \draw (o5) -- (o5 |- codL2.south);
     \draw (i1 |- cc.south) to node {\tiny $\cod$} (i1 |- codL.north);
     \draw (i1 |- codL.south) to node {\tiny $\cod$} (i1 |- codI.north);
  \end{pic}
\]
The square commutes, because, using equation~\eqref{eq:ccyank}, we have
\[
  \begin{pic}[auto]
     \node (i1) at (1,0.7) {};
     \node (i2) at (2,0.7) {$\monL_\bullet$};
     \node (i3) at (3,0.7) {};
     \node (o0) at (0,-0.7) {$\funI$};
     \node (o4) at (4,-0.7) {$\cod$};
     \node (o5) at (5,-0.7) {$\monL_\bullet$};
     \node[2cell,minimum width=1.1cm] (cc1) at (0.5,0.7) {$\cc$};
     \node[2cell,minimum width=1.1cm] (cc2) at (3.5,0.7) {$\cc$};
     \node[iso2cell,minimum width = 1.1cm] (codL1) at (1.5,0) {};
     \node[iso2cell,minimum width = 2.1cm] (codI) at (2,-0.7) {};
     \node[iso2cell,minimum width = 1.1cm] (codL2) at (4.5,0) {};
     \draw (o0) -- (o0 |- cc1.south);
     \draw (i2) -- (i2 |- codL.north);
     \draw (i3 |- cc2.south) to node {\tiny $\funI$} (i3 |- codI.north);
     \draw (o4 |- cc2.south) to node {\tiny $\cod$} (o4 |- codL2.north);
     \draw (o4) -- (o4 |- codL2.south);
     \draw (o5) -- (o5 |- codL2.south);
     \draw (i1 |- cc.south) to node {\tiny $\cod$} (i1 |- codL.north);
     \draw (i1 |- codL.south) to node {\tiny $\cod$} (i1 |- codI.north);
  \end{pic}
  =
  \begin{pic}[auto]
     \node (i2) at (2,0) {$\monL_\bullet$};
     \node (o0) at (0,-2.1) {$\funI$};
     \node (o1) at (1,-2.1) {$\cod$};
     \node (o2) at (2,-2.1) {$\monL_\bullet$};
     \node[2cell,minimum width=1.1cm] (cc) at (0.5,0) {$\cc$};
     \node[iso2cell,minimum width = 1.1cm] (codL1) at (1.5,-0.7) {};
     \node[iso2cell,minimum width = 1.1cm] (codL2) at (1.5,-1.4) {};
     \draw (o0) -- (o0 |- cc.south);
     \draw (i2) -- (i2 |- codL1.north);
     \draw (o1) -- (o1 |- codL2.south);
     \draw (o2) -- (o2 |- codL2.south);
     \draw (o1 |- cc.south) to node {\tiny $\cod$} (o1 |- codL1.north);
     \draw (o1 |- codL1.south) to node {\tiny $\cod$} (o1 |- codL2.north);
  \end{pic}
  =
  \begin{pic}[auto]
     \node (i2) at (2,0) {$\monL_\bullet$};
     \node (o0) at (0,-1) {$\funI$};
     \node (o1) at (1,-1) {$\cod$};
     \node (o2) at (2,-1) {$\monL_\bullet$};
     \node[2cell,minimum width=1.1cm] (cc) at (0.5,0) {$\cc$};
     \draw (o0) -- (o0 |- cc.south);
     \draw (o1) -- (o1 |- cc.south);
     \draw (o2) -- (i2);
  \end{pic}
\]
Applying $\dom\monL_\bullet^n$ to the whole square, we get a commutative square of 2-functors
from $\arrC$ to $\catC$ in which the left and right sides correspond to
$\funK_{n+1}$ and $\funK_n \monL_\bullet$ as 2-functors from $\arrC$ to itself.
Thus we have a 2-natural transformation $d_1 \colon \funK_{n+1} \to \funK_n \monL_\bullet$.
\end{definition}
The notation $d_1$ is used because its downstairs part,
applied to a bundle $p\colon E \to B$, is
$\monL_B^n d_1 \colon \monL_B^{n+1} B \to \monL_B^{n} B$.

\begin{lemma}\label{lem:d1d0}
  \[
    \begin{pic}[auto]
     \node (i0) at (0,0) {$\funK_{n+2}$};
     \node (i1) at (1,0) {};
     \node[2cell,minimum width=1.1cm] (d1) at (0.5,-0.7) {$d_1$};
     \node[2cell,minimum width=0cm] (d0) at (0, -1.5) {$\dzeroK$};
     \node (o0) at (0,-2.2) {$\funK_{n}$};
     \node (o1) at (1,-2.2) {$\monL_\bullet$};
     \draw (i0) -- (i0 |- d1.north);
     \draw (i0 |- d1.south) to node {\tiny $\funK_{n+1}$} (i0 |- d0.north);
     \draw (o0 |- d0.south) -- (o0);
     \draw (o1 |- d1.south) -- (o1);
    \end{pic}
    =
    \begin{pic}[auto]
     \node (i0) at (0,0) {$\funK_{n+2}$};
     \node (i1) at (1,0) {};
     \node[2cell,minimum width=0cm] (d0) at (0, -0.7) {$\dzeroK$};
     \node[2cell,minimum width=1.1cm] (d1) at (0.5,-1.5) {$d_1$};
     \node (o0) at (0,-2.2) {$\funK_{n}$};
     \node (o1) at (1,-2.2) {$\monL_\bullet$};
     \draw (i0) -- (i0 |- d0.north);
     \draw (i0 |- d0.south) to node {\tiny $\funK_{n+1}$} (i0 |- d1.north);
     \draw (o0 |- d1.south) -- (o0);
     \draw (o1 |- d1.south) -- (o1);
    \end{pic}
  \]
\end{lemma}
\begin{proof}
  We must show that the equation holds when composed on the left with $\dom$ and with $\cod$.
  $\dom$ is easy, given that each $\dom d_1$ is an identity 2-cell,
  and $\dom \dzeroK \monL_\bullet = \dom \dzeroK$
  (with different values of $n$ for the two instances of $\dzeroK)$.
  $\cod$ is straightforward from the definitions.
  Modulo appropriate identity 2-cells at top and bottom,
  both sides reduce to
  \[
    \begin{pic}[auto]
     \node (i0) at (0,0) {$\dom$};
     \node (i1) at (1,0) {$\monL_\bullet$};
     \node (i2) at (2,0) {$\monL_\bullet^n$};
     \node (i3) at (3,0) {};
     \node (i4) at (4,0) {};
     \node (i5) at (5,0) {$\monL_\bullet$};
     \node (i6) at (6,0) {$\funI$};
     \node (i7) at (7,0) {$\cod$};
     \node (i8) at (8,0) {};
     \node[2cell,minimum width=1.1cm] (cc) at (3.5, 0) {$\cc$};
     \node[2cell,minimum width=1.1cm] (d0) at (0.5,-0.7) {$d_0$};
     \node[iso2cell,minimum width=2.1cm] (iso1) at (5,-0.7) {};
     \node[iso2cell,minimum width=1.1cm] (iso2) at (7.5,-0.7) {};
     \node (o0) at (0,-1.4) {$\dom$};
     \node (o2) at (2,-1.4) {$\monL_\bullet^n$};
     \node (o3) at (3,-1.4) {$\funI$};
     \node (o7) at (7,-1.4) {$\cod$};
     \node (o8) at (8,-1.4) {$\monL_\bullet$};
     \draw (i0) -- (i0 |- d0.north);
     \draw (i1) -- (i1 |- d0.north);
     \draw (i2) -- (o2);
     \draw (i5) -- (i5 |- iso1.north);
     \draw (i6) -- (i6 |- iso1.north);
     \draw (i7) -- (i7 |- iso2.north);
     \draw (i3 |- cc.south) -- (o3);
     \draw (i4 |- cc.south) to node {\tiny $\cod$} (i4 |- iso1.north);
     \draw (i0 |- d0.south) -- (o0);
     \draw (i7 |- iso2.south) -- (o7);
     \draw (i8 |- iso2.south) -- (o8);
    \end{pic}
    \text{.}
  \]
\end{proof}

The following lemma is used in Definition~\ref{def:Psi}.
\begin{lemma}\label{lem:d1Fsupine}
  Let $\mathfrak{F}$ be a 2-functor from $\catD$ to $\arrC$.
  Then $d_1 \mathfrak{F} \colon \funK_{n+1} \mathfrak{F} \to \funK_{n} \mathfrak{F}$
  is supine over $[\catD,\cod] \colon [\catD,\arrC] \to [\catD, \catC]$.
\end{lemma}
\begin{proof}
  In diagram~\eqref{eq:d1K}, the top arrow is an isomorphism,
  so each component of $d_1$ is a supine morphism in $\arrC$.
\end{proof}

\begin{definition}
  Let $\alpha \colon \monL_\bullet^m \to \monL_\bullet^n$ be a 2-natural transformation.
  Consider the square of 2-functors
  \[
    \xymatrix@C=2cm{
      {\dom\monL_\bullet^m} \ar@{->}[d]_{\dom\monL_\bullet^m\cc} \ar@{->}[r]^{\dom\alpha}
           & {\dom\monL_\bullet^n} \ar@{->}[d]^{\dom\monL_\bullet^n\cc} \\
      {\dom\monL_\bullet^m \funI\cod} \ar@{->}[r]_{\dom\alpha\funI\cod}
           &  {\dom\monL_\bullet^n \funI\cod}
    }
  \]
  It commutes, and so defines a 2-natural transformation $\alpha \colon \funK_m \to \funK_n$.
\end{definition}

We use this to define 2-natural transformations $i\colon \Id_\arrC \to \funK_1$
and $c \colon \funK_2 \to \funK_1$.

\begin{lemma}\label{lem:cid0d1}
  \[
    \begin{pic}[auto]
      \node (i0) at (0,0) {$\funK_2$};
      \node[2cell,minimum width=0cm] (c) at (0,-0.7) {$c$};
      \node[2cell,minimum width=0cm] (d0) at (0, -1.5) {$\dzeroK$};
      \draw (i0) -- (i0 |- c.north);
      \draw (i0 |- c.south) to node {\tiny $\funK_1$} (i0 |- d0.north);
    \end{pic}
    \stackrel{\text{(a)}}{=}
    \begin{pic}[auto]
      \node (i0) at (0,0) {$\funK_2$};
      \node[2cell,minimum width=0cm] (d01) at (0,-0.7) {$\dzeroK$};
      \node[2cell,minimum width=0cm] (d02) at (0, -1.5) {$\dzeroK$};
      \draw (i0) -- (i0 |- d01.north);
      \draw (i0 |- d01.south) to node {\tiny $\funK_1$} (i0 |- d02.north);
    \end{pic}
    \text{, }
    \begin{pic}[auto]
      \node (i0) at (0,0) {$\funK_2$};
      \node[2cell,minimum width=0cm] (c) at (0,-0.7) {$c$};
      \node[2cell,minimum width=0cm] (d1) at (0, -1.5) {$d_1$};
      \node (o0) at (0,-2.2) {$\monL_\bullet$};
      \draw (i0) -- (i0 |- c.north);
      \draw (o0 |- c.south) to node {\tiny $\funK_1$} (o0 |- d1.north);
      \draw (o0 |- d1.south) -- (o0);
    \end{pic}
    \stackrel{\text{(b)}}{=}
    \begin{pic}[auto]
      \node (i0) at (0,0) {$\funK_2$};
      \node (i1) at (1,0) {};
      \node[2cell,minimum width=1.1cm] (d11) at (0.5,-0.7) {$d_1$};
      \node[2cell, minimum width = 0cm] (d12) at (0, -1.5) {$d_1$};
      \node[2cell,minimum width=1.1cm] (c) at (0.5,-2.3) {$c$};
      \node (o0) at (0,-3) {$\monL_\bullet$};
      \draw (i0) -- (i0 |- d11.north);
      \draw (i0 |- d11.south) to node {\tiny $\funK_1$} (i0 |- d12.north);
      \draw (i0 |- d12.south) to node {\tiny $\monL_\bullet$} (i0 |- c.north);
      \draw (i1 |- d11.south) to node {\tiny $\monL_\bullet$} (i1 |- c.north);
      \draw (i0 |- c.south) -- (o0);
    \end{pic}
    \text{, }
    \begin{pic}[auto]
      \node (i0) at (0,0) {};
      \node[2cell,minimum width=0cm] (i) at (0,0) {$i$};
      \node[2cell,minimum width=0cm] (d0) at (0, -0.8) {$\dzeroK$};
      \draw (i0 |- i.south) to node {\tiny $\funK_1$} (i0 |- d0.north);
    \end{pic}
    \stackrel{\text{(c)}}{=}
    \quad
    \text{, }
    \begin{pic}[auto]
      \node (i0) at (0,0) {};
      \node[2cell,minimum width=0cm] (i) at (0,0) {$i$};
      \node[2cell,minimum width=0cm] (d1) at (0, -0.8) {$d_1$};
      \node (o0) at (0,-1.5) {$\monL_\bullet$};
      \draw (i0 |- i.south) to node {\tiny $\funK_1$} (i0 |- d1.north);
      \draw (i0 |- d1.south) -- (o0);
    \end{pic}
    \stackrel{\text{(d)}}{=}
    \begin{pic}[auto]
      \node (i0) at (0,0) {};
      \node[2cell,minimum width=0cm] (i) at (0,0) {$i$};
      \node (o0) at (0,-0.7) {$\monL_\bullet$};
      \draw (i0 |- i.south) -- (o0);
    \end{pic}
    \text{.}
  \]
  (The right-hand side of (c) is the empty diagram,
  i.e. an identity 2-cell on an identity 1-cell.)
\end{lemma}
\begin{proof}
  In each part we must check the equation ``upstairs and downstairs'',
  i.e. when left composed with $\dom$ and with $\cod$.

  (a): For the upstairs part, after composing with an appropriate identity 2-cell at the top,
  we find we require
  \[
    \begin{pic}[auto]
     \node (i0) at (0,0) {$\dom$};
     \node (i1) at (1,0) {$\monL_\bullet$};
     \node (i2) at (2,0) {$\monL_\bullet$};
     \node[2cell,minimum width=1.1cm] (c) at (1.5,-0.7) {$c$};
     \node[2cell,minimum width=1.1cm] (d0) at (0.5, -1.5) {$d_0$};
     \node (o0) at (0,-2.2) {$\dom$};
     \draw (i0) -- (i0 |- d0.north);
     \draw (i1) -- (i1 |- c.north);
     \draw (i2) -- (i2 |- c.north);
     \draw (i1 |- c.south) to node {\tiny $\monL_\bullet$} (i1 |- d0.north);
     \draw (o0 |- d0.south) -- (o0);
    \end{pic}
    =
    \begin{pic}[auto]
     \node (i0) at (0,0) {$\dom$};
     \node (i1) at (1,0) {$\monL_\bullet$};
     \node (i2) at (2,0) {$\monL_\bullet$};
     \node[2cell,minimum width=1.1cm] (d01) at (0.5, -0.7) {$d_0$};
     \node[2cell,minimum width=2.1cm] (d02) at (1,-1.5) {$d_0$};
     \node (o0) at (0,-2.2) {$\dom$};
     \draw (i0) -- (i0 |- d01.north);
     \draw (i1) -- (i1 |- d01.north);
     \draw (i2) -- (i2 |- d02.north);
     \draw (i0 |- d01.south) to node {\tiny $\dom$} (i0 |- d02.north);
     \draw (o0 |- d02.south) -- (o0);
    \end{pic}
  \]
  This is just a rearrangement of the condition $d_0 c = d_0 d_0$ in
  Definition~\ref{def:ic}.

  The downstairs part follows from the upstairs,
  because composing with $\cod$ on the left is equal to composing with
  $\dom$ on the left and $\funI\cod$ on the right.

  (b):
  The upstairs part is clear, bearing in mind that each $\upstairs{d_1}$ is an identity morphism.
  For the downstairs part, composing with an appropriate identity 2-cell at the top,
  we calculate as below.
  Note that for the $c$ belonging to $\monL_\bullet$ we have
  $\cod c = \cod$.
  This is used in the second and fourth equations.
  The third uses equation~\eqref{eq:ccyank}.
  \begin{align*}
    \text{LHS }
    & =
    \begin{pic}[auto]
      \node (i0) at (0,0) {$\dom$};
      \node (i1) at (1,0) {};
      \node (i2) at (2,0) {};
      \node (i3) at (3,0) {$\monL_\bullet$};
      \node (i4) at (4,0) {$\monL_\bullet$};
      \node (i5) at (5,0) {$\funI$};
      \node (i6) at (6,0) {$\cod$};
      \node (i7) at (7,0) {};
      \node[2cell,minimum width=1.1cm] (cc) at (1.5,-0.7) {$\cc$};
      \node[2cell,minimum width=1.1cm] (c) at (3.5,-0.7) {$c$};
      \node[iso2cell,minimum width=1.1cm] (iso2) at (6.5,-0.7) {};
      \node[iso2cell,minimum width=1.1cm] (iso3) at (0.5,-1.4) {};
      \node[iso2cell,minimum width=3.1cm] (iso4) at (3.5,-1.4) {};
      \node (o6) at (6,-1.4) {$\cod$};
      \node (o7) at (7,-1.4) {$\monL_\bullet$};
      \draw (i0) -- (i0 |- iso3.north);
      \draw (i3) -- (i3 |- c.north);
      \draw (i4) -- (i4 |- c.north);
      \draw (i5) -- (i5 |- iso4.north);
      \draw (i6) -- (i6 |- iso2.north);
      \draw (i1 |- cc.south) to node {\tiny $\funI$} (i1 |- iso3.north);
      \draw (i2 |- cc.south) to node {\tiny $\cod$} (i2 |- iso4.north);
      \draw (i3 |- c.south) to node {\tiny $\monL_\bullet$} (i3 |- iso4.north);
      \draw (o6 |- iso2.south) -- (o6);
      \draw (o7 |- iso2.south) -- (o7);
    \end{pic}    \\
    & =
    \begin{pic}[auto]
      \node (i0) at (0,0) {$\dom$};
      \node (i1) at (1,0) {};
      \node (i2) at (2,0) {};
      \node (i3) at (3,0) {$\monL_\bullet$};
      \node (i4) at (4,0) {$\monL_\bullet$};
      \node (i5) at (5,0) {$\funI$};
      \node (i6) at (6,0) {$\cod$};
      \node (i7) at (7,0) {};
      \node[2cell,minimum width=1.1cm] (cc) at (1.5,0) {$\cc$};
      \node[iso2cell,minimum width=1.1cm] (iso1) at (0.5,-0.7) {};
      \node[iso2cell,minimum width=2.1cm] (iso2) at (3,-0.7) {};
      \node[iso2cell,minimum width=1.1cm] (iso3) at (6.5,-0.7) {};
      \node[iso2cell,minimum width=3.1cm] (iso4) at (3.5,-1.4) {};
      \node (o6) at (6,-1.4) {$\cod$};
      \node (o7) at (7,-1.4) {$\monL_\bullet$};
      \draw (i0) -- (i0 |- iso1.north);
      \draw (i3) -- (i3 |- iso2.north);
      \draw (i4) -- (i4 |- iso2.north);
      \draw (i5) -- (i5 |- iso4.north);
      \draw (i6) -- (i6 |- iso3.north);
      \draw (i1 |- cc.south) to node {\tiny $\funI$} (i1 |- iso1.north);
      \draw (i2 |- cc.south) to node {\tiny $\cod$} (i2 |- iso2.north);
      \draw (i2 |- iso2.south) to node {\tiny $\cod$} (i2 |- iso4.north);
      \draw (i3 |- iso2.south) to node {\tiny $\monL_\bullet$} (i3 |- iso4.north);
      \draw (o6 |- iso3.south) -- (o6);
      \draw (o7 |- iso3.south) -- (o7);
    \end{pic}\\
    & =
    \begin{pic}[auto]
      \node (i0) at (0,0) {$\dom$};
      \node (i1) at (1,0) {};
      \node (i2) at (2,0) {};
      \node (i3) at (3,0) {$\monL_\bullet$};
      \node (i4) at (4,0) {};
      \node (i5) at (5,0) {};
      \node (i6) at (6,0) {$\monL_\bullet$};
      \node (i7) at (7,0) {$\funI$};
      \node (i8) at (8,0) {$\cod$};
      \node (i9) at (9,0) {};
      \node[2cell,minimum width=1.1cm] (cc1) at (1.5,0) {$\cc$};
      \node[2cell,minimum width=1.1cm] (cc2) at (4.5,0) {$\cc$};
      \node[iso2cell,minimum width=1.1cm] (iso1) at (0.5,-0.7) {};
      \node[iso2cell,minimum width=1.1cm] (iso2) at (2.5,-0.7) {};
      \node[iso2cell,minimum width=1.1cm] (iso3) at (5.5,-0.7) {};
      \node[iso2cell,minimum width=1.1cm] (iso4) at (8.5,-0.7) {};
      \node[iso2cell,minimum width=2.1cm] (iso5) at (3,-1.4) {};
      \node[iso2cell,minimum width=2.1cm] (iso6) at (6,-1.4) {};
      \node (o8) at (8,-1.4) {$\cod$};
      \node (o9) at (9,-1.4) {$\monL_\bullet$};
      \draw (i0) -- (i0 |- iso1.north);
      \draw (i3) -- (i3 |- iso2.north);
      \draw (i6) -- (i6 |- iso3.north);
      \draw (i7) -- (i7 |- iso6.north);
      \draw (i8) -- (i8 |- iso4.north);
      \draw (i1 |- cc1.south) to node {\tiny $\funI$} (i1 |- iso1.north);
      \draw (i2 |- cc1.south) to node {\tiny $\cod$} (i2 |- iso2.north);
      \draw (i4 |- cc2.south) to node {\tiny $\funI$} (i4 |- iso5.north);
      \draw (i5 |- cc2.south) to node {\tiny $\cod$} (i5 |- iso3.north);
      \draw (i2 |- iso2.south) to node {\tiny $\cod$} (i2 |- iso5.north);
      \draw (i5 |- iso3.south) to node {\tiny $\cod$} (i5 |- iso6.north);
      \draw (o8 |- iso4.south) -- (o8);
      \draw (o9 |- iso4.south) -- (o9);
    \end{pic}
    = \text{ RHS.}
  \end{align*}

  (c) and (d) are analogous to (a) and (b), but simpler.
\end{proof}

\section{Functors $\funT_\bullet$ preserving (op)fibrations}\label{sec:fibpres}
In this section we present our main technical result, Theorem~\ref{thm:main}:
any indexed bundle 2-endomorphism $\funT_\bullet$ preserves both opfibrations and fibrations
(as well as the pseudo- versions).
The main part of the argument is to show that $\funT_\bullet$ lifts to the 2-categories
of pseudoalgebras and normalized pseudoalgebras of $\monL_\bullet$,
by defining a 2-transition $\Psi_\bullet$ from $\monL_\bullet$ to itself along $\funT_\bullet$.
A dual argument shows that it also lifts to the corresponding 2-categories for $\monR_\bullet$.
\begin{definition}\label{def:tbullet}
  Let $\arrC$ be the arrow 2-category over a representable 2-category $\catC$.

  A \emph{bundle 2-endomorphism} for $\catC$ is a 2-endofunctor $\funT_\bullet$ of $\arrC$
  such that $\cod\funT_\bullet = \cod$.
  The $\bullet$ used in the notation for bundle endofunctors indicates the ability
  to restrict to $\funT_B$ on each slice $\catC/B$,
  and use notation similar to that already introduced for $\monL_\bullet$.
  Thus $\funT_{\downstairs{f}}\upstairs{f}$ denotes $\upstairs{\funT_\bullet f}$.
  (The ``bundle endomorphism'' condition already tells us that
  $\downstairs{\funT_\bullet f} = \downstairs{f}$.)

  We say that $\funT_\bullet$ is \emph{indexed} if it preserves proneness --
  whenever a morphism $f$ in $\arrC$ is, as commutative square in $\catC$,
  a pullback square, then so too is $\funT_\bullet f$.
  This is equivalent to its being an indexed endofunctor in the sense of indexed categories,
  also to $(\funT_\bullet, \Id_\catC)$ being a morphism of fibrations
  -- though not of opfibrations.
\end{definition}

Note that the slice endofunctors $\monL_B$ and $\monR_B$ are not preserved by pullback:
although $\monL_\bullet$ and $\monR_\bullet$ are bundle 2-endomorphisms,
they are not indexed.
Hence we cannot use the language of indexed categories to discuss the interaction
between them and $\funT_\bullet$.
Instead we use the codomain bifibration explicitly.

\begin{definition}\label{def:psin}
  Let $\funT_\bullet$ be a bundle 2-endofunctor for $\catC$.
  For each natural number $n$, we define the 2-natural transformation
  $\psin{n} \colon \funT_\bullet \funK_n \to \funK_n \funT_\bullet$
  recursively as follows.
  $\psin{0} \colon \funT_\bullet \to \funT_\bullet$ is the identity,
  and $\psin{n+1}$ is
  (using Corollary~\ref{cor:d0Fprone} to deduce that $\dzeroK\funT_\bullet$ is prone)
  the unique 2-natural transformation such that
  \[
    \begin{pic}[auto]
      \node (i0) at (0,0) {$\funT_\bullet$};
      \node (i1) at (1,0) {$\funK_{n+1}$};
      \node[2cell,minimum width=1.1cm] (a) at (0.5,-0.7) {$\psin{n+1}$};
      \node[2cell,minimum width=0cm] (d0) at (0,-1.5) {$\dzeroK$};
      \node (o0) at (0,-2.2) {$\funK_n$};
      \node (o1) at (1,-2.2) {$\funT_\bullet$};
      \draw (i0) -- (i0 |- a.north);
      \draw (i1) -- (i1 |- a.north);
      \draw (i0 |- a.south) to node {\tiny $\funK_{n+1}$} (i0 |- d0.north);
      \draw (i1 |- a.south) -- (o1);
      \draw (o0 |- d0.south) -- (o0);
    \end{pic}
    =
    \begin{pic}[auto]
      \node (i0) at (0,0) {$\funT_\bullet$};
      \node (i1) at (1,0) {$\funK_{n+1}$};
      \node[2cell,minimum width=0cm] (d0) at (1,-0.7) {$\dzeroK$};
      \node[2cell,minimum width=1.1cm] (a) at (0.5,-1.5) {$\psin{n}$};
      \node (o0) at (0,-2.2) {$\funK_n$};
      \node (o1) at (1,-2.2) {$\funT_\bullet$};
      \draw (i0) -- (i0 |- a.north);
      \draw (i1) -- (i1 |- d0.north);
      \draw (i1 |- d0.south) to node {\tiny $\funK_{n}$} (i1 |- a.north);
      \draw (o0 |- a.south) -- (o0);
      \draw (o1 |- a.south) -- (o1);
    \end{pic}
    \text{ and }
    \begin{pic}[auto]
      \node (i0) at (0,0) {$\cod$};
      \node (i1) at (1,0) {$\funT_\bullet$};
      \node (i2) at (2,0) {$\funK_{n+1}$};
      \node[2cell,minimum width=1.1cm] (a) at (1.5,-0.7) {$\psin{n+1}$};
      \node (o0) at (0,-1.4) {$\cod$};
      \node (o1) at (1,-1.4) {$\funK_{n+1}$};
      \node (o2) at (2,-1.4) {$\funT_\bullet$};
      \draw (i0) -- (o0);
      \draw (i1) -- (i1 |- a.north);
      \draw (i2) -- (i2 |- a.north);
      \draw (o1 |- a.south) -- (o1);
      \draw (o2 |- a.south) -- (o2);
    \end{pic}
    =
    \begin{pic}[auto]
      \node (i0) at (0,0) {$\cod$};
      \node (i1) at (1,0) {$\funT_\bullet$};
      \node (i2) at (2,0) {$\funK_{n+1}$};
      \node (i3) at (3,0) {};
      \node (i4) at (4,0) {};
      \node[iso2cell,minimum width=4.1cm] (iso1) at (2,-0.7) {};
      \node[iso2cell,minimum width=3.1cm] (iso2) at (1.5,-1.7) {};
      \node (o0) at (0,-2.4) {$\cod$};
      \node (o1) at (1,-2.4) {$\funK_{n+1}$};
      \node (o4) at (4,-2.4) {$\funT_\bullet$};
      \draw (i0) -- (i0 |- iso1.north);
      \draw (i1) -- (i1 |- iso1.north);
      \draw (i2) -- (i2 |- iso1.north);
      \draw (i0 |- iso1.south) to node {\tiny $\dom$} (i0 |- iso2.north);
      \draw (i1 |- iso1.south) to node {\tiny $\monL_\bullet^{n+1}$} (i1 |- iso2.north);
      \draw (i2 |- iso1.south) to node {\tiny $\funI$} (i2 |- iso2.north);
      \draw (i3 |- iso1.south) to node {\tiny $\cod$} (i3 |- iso2.north);
      \draw (o4 |- iso1.south) -- (o4);
      \draw (o0 |- iso2.south) -- (o0);
      \draw (o1 |- iso2.south) -- (o1);
    \end{pic}\text{.}
   \]
\end{definition}

For the rest of this section, we shall take it that we are given a representable
category $\catC$ and an indexed bundle 2-endofunctor $\funT_\bullet$ for it.

\begin{proposition}
  Each $\psin{n}$ is a 2-natural isomorphism.
\end{proposition}
\begin{proof}
  Because $\funT_\bullet$ preserves proneness, we know that $\funT_\bullet \dzeroK$
  satisfies the conditions of Lemma~\ref{lem:pullback} and hence is prone.
  This allows us to define $\psin{n}^{-1}$ with diagrams similar to those of
  Definition~\ref{def:psin}.
  In the first equation each diagram is reflected left to right,
  while in the second they are upside down.
  One can then prove it is the inverse of $\psin{n}$.
\end{proof}

\begin{lemma}\label{lem:psiic}
  \[
    \begin{pic}[auto]
      \node (i0) at (0,0) {};
      \node[2cell,minimum width=0cm] (i) at (0,0) {$i$};
      \node (i1) at (1,0) {$\funT_\bullet$};
      \node (o0) at (0,-1) {$\funK_1$};
      \node (o1) at (1,-1) {$\funT_\bullet$};
      \draw (i0 |- i.south) -- (o0);
      \draw (i1) -- (o1);
    \end{pic}
    \stackrel{\text{(a)}}{=}
    \begin{pic}[auto]
      \node (i0) at (0,0) {$\funT_\bullet$};
      \node (i1) at (1,0) {};
      \node[2cell,minimum width=0cm] (i) at (1,0) {$i$};
      \node[2cell,minimum width=1.1cm] (psi) at (0.5,-0.7) {$\psin{1}$};
      \node (o0) at (0,-1.4) {$\funK_1$};
      \node (o1) at (1,-1.4) {$\funT_\bullet$};
      \draw (i0) -- (i0 |- psi.north);
      \draw (i1 |- i.south) to node {\tiny $\funK_1$} (i1 |- psi.north);
      \draw (i0 |- psi.south) -- (o0);
      \draw (i1 |- psi.south) -- (o1);
    \end{pic}
    \text{, }
    \begin{pic}[auto]
      \node (i0) at (0,0) {$\funT_\bullet$};
      \node (i1) at (1,0) {$\funK_2$};
      \node[2cell,minimum width=1.1cm] (alpha) at (0.5,-0.7) {$\psin{2}$};
      \node[2cell,minimum width=0cm] (c) at (0,-1.5) {$c$};
      \node (o0) at (0,-2.2) {$\funK_1$};
      \node (o1) at (1,-2.2) {$\funT_\bullet$};
      \draw (i0) -- (i0 |- alpha.north);
      \draw (i1) -- (i1 |- alpha.north);
      \draw (i0 |- alpha.south) to node {\tiny $\funK_2$} (i0 |- c.north);
      \draw (o0 |- c.south) -- (o0);
      \draw (o1 |- alpha.south) -- (o1);
    \end{pic}
    \stackrel{\text{(b)}}{=}
    \begin{pic}[auto]
      \node (i0) at (0,0) {$\funT_\bullet$};
      \node (i1) at (1,0) {$\funK_2$};
      \node[2cell,minimum width=0cm] (c) at (1,-0.7) {$c$};
      \node[2cell,minimum width=1.1cm] (alpha) at (0.5,-1.4) {$\psin{1}$};
      \node (o0) at (0,-2.1) {$\funK_1$};
      \node (o1) at (1,-2.1) {$\funT_\bullet$};
      \draw (i0) -- (i0 |- alpha.north);
      \draw (i1) -- (i1 |- c.north);
      \draw (i1 |- c.south) to node {\tiny $\funK_1$} (i1 |- alpha.north);
      \draw (o0 |- alpha.south) -- (o0);
      \draw (o1 |- alpha.south) -- (o1);
    \end{pic} \text{.}
   \]
\end{lemma}
\begin{proof}
  By Corollary~\ref{cor:d0Fprone},
  taking $\dzeroK \colon \funK_1 \to \funK_0 = \Id_{\arrC}$,
  it suffices to check for each equation that it holds when composed (i) with $\cod$ on the left,
  and (ii) with $\dzeroK \funT_\bullet$ at the bottom.
  (i) is clear from the facts that each $\cod\psin{n}$ is an identity,
  and $\cod\funT_\bullet = \cod$.
  It remains to check (ii) for each equation.

  For (a) we have
  \[
    \text{LHS } =
    \begin{pic}[auto]
      \node (i0) at (0,0) {};
      \node[2cell,minimum width=0cm] (i) at (0,0) {$i$};
      \node (i1) at (1,0) {$\funT_\bullet$};
      \node (o0) at (0,-1) {$\funK_1$};
      \node[2cell,minimum width=0cm] (d0) at (0,-1) {$\dzeroK$};
      \node (o1) at (1,-1) {$\funT_\bullet$};
      \draw (i0 |- i.south) to node {\tiny $\funK_1$} (i0 |- d0.north);
      \draw (i1) -- (o1);
    \end{pic}
    =
    \begin{pic}[auto]
      \node (i0) at (0,0) {$\funT_\bullet$};
      \node (o0) at (0,-1) {$\funT_\bullet$};
      \draw (i0) -- (o0);
    \end{pic}
    =
    \begin{pic}[auto]
      \node (i0) at (0,0) {$\funT_\bullet$};
      \node (i1) at (1,0) {};
      \node[2cell,minimum width=0cm] (i) at (1,0) {$i$};
      \node (o0) at (0,-1) {$\funT_\bullet$};
      \node (o1) at (1,-1) {};
      \node[2cell,minimum width=0cm] (d0) at (1,-1) {$\dzeroK$};
      \draw (i0) -- (o0);
      \draw (i1 |- i.south) to node {\tiny $\funK_1$} (i1 |- d0.north);
    \end{pic}
    =
    \begin{pic}[auto]
      \node (i0) at (0,0) {$\funT_\bullet$};
      \node (i1) at (1,0) {};
      \node[2cell,minimum width=0cm] (i) at (1,0) {$i$};
      \node[2cell,minimum width=1.1cm] (psi) at (0.5,-0.8) {$\psin{1}$};
      \node (o0) at (0,-1.6) {};
      \node[2cell,minimum width=0cm] (d0) at (0,-1.6) {$\dzeroK$};
      \node (o1) at (1,-1.6) {$\funT_\bullet$};
      \draw (i0) -- (i0 |- psi.north);
      \draw (i1 |- i.south) to node {\tiny $\funK_1$} (i1 |- psi.north);
      \draw (i0 |- psi.south) to node {\tiny $\funK_1$} (i0 |- d0.north);
      \draw (i1 |- psi.south) -- (o1);
    \end{pic}
  \]
  Here the second and third equations use Lemma~\ref{lem:cid0d1}~(c),
  and the fourth uses Definition~\ref{def:psin}.

  For (b) we have
  \[
    \begin{pic}[auto]
      \node (i0) at (0,0) {$\funT_\bullet$};
      \node (i1) at (1,0) {$\funK_2$};
      \node[2cell,minimum width=1.1cm] (alpha) at (0.5,-0.7) {$\psin{2}$};
      \node[2cell,minimum width=0cm] (c) at (0,-1.5) {$c$};
      \node[2cell,minimum width=0cm] (d0) at (0,-2.3) {$\dzeroK$};
      \node (o1) at (1,-2.3) {$\funT_\bullet$};
      \draw (i0) -- (i0 |- alpha.north);
      \draw (i1) -- (i1 |- alpha.north);
      \draw (i0 |- alpha.south) to node {\tiny $\funK_2$} (i0 |- c.north);
      \draw (o1 |- alpha.south) -- (o1);
      \draw (i0 |- c.south) to node {\tiny $\funK_1$} (i0 |- d0.north);
    \end{pic}
    =
    \begin{pic}[auto]
      \node (i0) at (0,0) {$\funT_\bullet$};
      \node (i1) at (1,0) {$\funK_2$};
      \node[2cell,minimum width=1.1cm] (alpha) at (0.5,-0.7) {$\psin{2}$};
      \node[2cell,minimum width=0cm] (d01) at (0,-1.5) {$\dzeroK$};
      \node[2cell,minimum width=0cm] (d02) at (0,-2.3) {$\dzeroK$};
      \node (o1) at (1,-2.3) {$\funT_\bullet$};
      \draw (i0) -- (i0 |- alpha.north);
      \draw (i1) -- (i1 |- alpha.north);
      \draw (i0 |- alpha.south) to node {\tiny $\funK_2$} (i0 |- d01.north);
      \draw (o1 |- alpha.south) -- (o1);
      \draw (i0 |- d01.south) to node {\tiny $\funK_1$} (i0 |- d02.north);
    \end{pic}
    =
    \begin{pic}[auto]
      \node (i0) at (0,0) {$\funT_\bullet$};
      \node (i1) at (1,0) {$\funK_2$};
      \node[2cell,minimum width=0cm] (d01) at (1,-0.7) {$\dzeroK$};
      \node[2cell,minimum width=0cm] (d02) at (1,-1.5) {$\dzeroK$};
      \node (o0) at (0,-1.5) {$\funT_\bullet$};
      \draw (i0) -- (o0);
      \draw (i1) -- (i1 |- d01.north);
      \draw (i1 |- d01.south) to node {\tiny $\funK_1$} (i1 |- d02.north);
    \end{pic}
    =
    \begin{pic}[auto]
      \node (i0) at (0,0) {$\funT_\bullet$};
      \node (i1) at (1,0) {$\funK_2$};
      \node[2cell,minimum width=0cm] (c) at (1,-0.7) {$c$};
      \node[2cell,minimum width=0cm] (d0) at (1,-1.5) {$\dzeroK$};
      \node (o0) at (0,-1.5) {$\funT_\bullet$};
      \draw (i0) -- (o0);
      \draw (i1) -- (i1 |- c.north);
      \draw (i1 |- c.south) to node {\tiny $\funK_1$} (i1 |- d0.north);
    \end{pic}
    =
    \begin{pic}[auto]
      \node (i0) at (0,0) {$\funT_\bullet$};
      \node (i1) at (1,0) {$\funK_2$};
      \node[2cell,minimum width=0cm] (c) at (1,-0.7) {$c$};
      \node[2cell,minimum width=1.1cm] (alpha) at (0.5,-1.5) {$\psin{1}$};
      \node[2cell,minimum width=0cm] (d0) at (0,-2.3) {$\dzeroK$};
      \node (o1) at (1,-2.3) {$\funT_\bullet$};
      \draw (i0) -- (i0 |- alpha.north);
      \draw (i1) -- (i1 |- c.north);
      \draw (i1 |- c.south) to node {\tiny $\funK_1$} (i1 |- alpha.north);
      \draw (i0 |- alpha.south) to node {\tiny $\funK_1$} (i0 |- d0.north);
      \draw (o1 |- alpha.south) -- (o1);
    \end{pic}
  \]
  For the first and third equations we have used Lemma~\ref{lem:cid0d1}~(a),
  while for the second and fourth we have used Definition~\ref{def:psin}.
\end{proof}

\begin{definition}\label{def:Psi}
  Using Lemma~\ref{lem:d1Fsupine} with $\funT_\bullet$ for $\mathfrak{F}$,
  and using invertibility of $\psi_1$,
  we define the 2-natural transformation
  $\Psi_\bullet \colon \monL_\bullet\funT_\bullet \rightarrow \funT_\bullet \monL_\bullet$
  as the unique such over $\cod$ satisfying
  \[
    \begin{pic}[auto]
      \node (i0) at (0,0) {$\funT_\bullet$};
      \node (i1) at (1,0) {$\funK_1$};
      \node[2cell,minimum width=1.1cm] (alpha) at (0.5,-0.7) {$\psin{1}$};
      \node[2cell,minimum width=0cm] (d1) at (0,-1.5) {$d_1$};
      \node[2cell,minimum width=1.1cm] (psi) at (0.5,-2.3) {$\Psi_\bullet$};
      \node (o0) at (0,-3) {$\funT_\bullet$};
      \node (o1) at (1,-3) {$\monL_\bullet$};
      \draw (i0) -- (i0 |- alpha.north);
      \draw (i1) -- (i1 |- alpha.north);
      \draw (i0 |- alpha.south) to node {\tiny $\funK_1$} (i0 |- d1.north);
      \draw (i1 |- alpha.south) to node {\tiny $\funT_\bullet$} (i1 |- psi.north);
      \draw (i0 |- d1.south) to node {\tiny $\monL_\bullet$} (i0 |- psi.north);
      \draw (o0 |- psi.south) -- (o0);
      \draw (o1 |- psi.south) -- (o1);
    \end{pic}
    =
    \begin{pic}[auto]
      \node (i0) at (0,0) {$\funT_\bullet$};
      \node (i1) at (1,0) {$\funK_1$};
      \node[2cell,minimum width=0cm] (d1) at (1,-0.7) {$d_1$};
      \node (o0) at (0,-1.4) {$\funT_\bullet$};
      \node (o1) at (1,-1.4) {$\monL_\bullet$};
      \draw (i0) -- (o0);
      \draw (i1) -- (i1 |- d1.north);
      \draw (o1 |- d1.south) -- (o1);
    \end{pic}
  \]
\end{definition}
In diagrammatic form, this works as follows.
From Proposition~\ref{prop:LBp},
we get two pullback squares
\[
  \xymatrix{
    {\monL_B E} \ar@{->}[d] \ar@{->}[r]^{d_0}
        &  {E} \ar@{->}[d]^{p} \\
    {\Phi B} \ar@{->}[r]^{d_0} & B
    {\ar@{-}(6,-2); (6,-4)}
    {\ar@{-}(4,-4); (6,-4)}
  }
  \quad
  \xymatrix{
    {\monL_B\funT_{B} E} \ar@{->}[d] \ar@{->}[r]^{d_0}
        &  {\funT_B E} \ar@{->}[d] \\
    {\Phi B} \ar@{->}[r]^{d_0} & B
    {\ar@{-}(6,-2); (6,-4)}
    {\ar@{-}(4,-4); (6,-4)}
  }
  \text{.}
\]
Applying $\funT_\bullet$ to the first, we obtain that
$\monL_B \funT_B E \cong \funT_{\Phi B}\monL_B E$.
Composing this with
\begin{equation}\label{eq:Td11}
  \downstairs{\funT_\bullet}\left(\vcenter{
    \xymatrix{
      {\monL_B E} \ar@{->}[d]_{\monL_B p} \ar@{=}[r] & {\monL_B E} \ar@{->}[d]^{d_1} \\
      {\Phi B} \ar@{->}[r]^{d_1}  &  B
    }
  }\right)
\end{equation}
gives us a 1-cell $\funT_{d_1}1\colon \monL_B \funT_B E \rightarrow \funT_B \monL_B E$,
which will be the component at $(E,p)$ of our 2-natural transformation $\Psi_\bullet$.

\begin{lemma}\label{lem:psiPsi}
  \[
    \begin{pic}[auto]
      \node (i0) at (0,0) {$\funT_\bullet$};
      \node (i1) at (1,0) {};
      \node (i2) at (2,0) {$\funK_{n+1}$};
      \node[2cell,minimum width=2.1cm] (alpha) at (1,-0.7) {$\psin{n+1}$};
      \node[2cell,minimum width=1.1cm] (d1) at (0.5,-1.5) {$d_1$};
      \node[2cell,minimum width=1.1cm] (psi) at (1.5,-2.3) {$\Psi_\bullet$};
      \node (o0) at (0,-3) {$\funK_{n}$};
      \node (o1) at (1,-3) {$\funT_\bullet$};
      \node (o2) at (2,-3) {$\monL_\bullet$};
      \draw (i0) -- (i0 |- alpha.north);
      \draw (i2) -- (i2 |- alpha.north);
      \draw (i0 |- alpha.south) to node {\tiny $\funK_{n+1}$} (i0 |- d1.north);
      \draw (i2 |- alpha.south) to node {\tiny $\funT_\bullet$} (i2 |- psi.north);
      \draw (o0 |- d1.south) -- (o0);
      \draw (i1 |- d1.south) to node {\tiny $\monL_\bullet$} (i1 |- psi.north);
      \draw (o1 |- psi.south) -- (o1);
      \draw (o2 |- psi.south) -- (o2);
    \end{pic}
    =
    \begin{pic}[auto]
      \node (i0) at (0,0) {$\funT_\bullet$};
      \node (i1) at (1,0) {$\funK_{n+1}$};
      \node (i2) at (2,0) {};
      \node[2cell,minimum width=1.1cm] (d1) at (1.5,-0.7) {$d_1$};
      \node[2cell,minimum width=1.1cm] (alpha) at (0.5,-1.5) {$\psin{n}$};
      \node (o0) at (0,-2.2) {$\funK_{n}$};
      \node (o1) at (1,-2.2) {$\funT_\bullet$};
      \node (o2) at (2,-2.2) {$\monL_\bullet$};
      \draw (i0) -- (i0 |- alpha.north);
      \draw (i1) -- (i1 |- d1.north);
      \draw (i1 |- d1.south) to node {\tiny $\funK_n$} (i1 |- alpha.north);
      \draw (i2 |- d1.south) -- (o2);
      \draw (o0 |- alpha.south) -- (o0);
      \draw (o1 |- alpha.south) -- (o1);
    \end{pic}
  \]
\end{lemma}
\begin{proof}
  Note that the case $n=0$ is simply Definition~\ref{def:Psi}.
  Using Corollary~\ref{cor:d0Fprone}, it suffices to prove the equation when composed
  (i) on the left with $\cod$, and (ii) at the bottom with $\dzeroK \funT_\bullet \monL_\bullet$.

  For (i), applying $\cod$, and bearing in mind that each
  $\cod \psin{n}$ and $\cod \Psi_\bullet$is an identity,
  we get -- modulo some identity 2-cells at top and bottom --
  \begin{align*}
    \text{LHS } & =
    \begin{pic}[auto]
      \node (i0) at (0,0) {$\dom$};
      \node (i1) at (1,0) {$\monL_\bullet^n$};
      \node (i2) at (2,0) {};
      \node (i3) at (3,0) {};
      \node (i4) at (4,0) {$\monL_\bullet$};
      \node (i5) at (5,0) {$\funI$};
      \node (i6) at (6,0) {$\cod$};
      \node (i7) at (7,0) {};
      \node (i8) at (8,0) {$\funT_\bullet$};
      \node[2cell,minimum width=1.1cm] (cc) at (2.5,0) {$\cc$};
      \node[iso2cell,minimum width=2.1cm] (iso1) at (4,-0.7) {};
      \node[iso2cell,minimum width=1.1cm] (iso2) at (6.5,-0.7) {};
      \node[2cell,minimum width=1.1cm] (psi) at (7.5,-1.4) {$\Psi_\bullet$};
      \node (o0) at (0,-2.1) {$\dom$};
      \node (o1) at (1,-2.1) {$\monL_\bullet^n$};
      \node (o2) at (2,-2.1) {$\funI$};
      \node (o6) at (6,-2.1) {$\cod$};
      \node (o7) at (7,-2.1) {$\funT_\bullet$};
      \node (o8) at (8,-2.1) {$\monL_\bullet$};
      \draw (i0) -- (o0);
      \draw (i1) -- (o1);
      \draw (i2 |- cc.south) -- (o2);
      \draw (i3 |- cc.south) to node {\tiny $\cod$} (i3 |- iso1.north);
      \draw (i4) -- (i4 |- iso1.north);
      \draw (i5) -- (i5 |- iso1.north);
      \draw (i6) -- (i6 |- iso2.north);
      \draw (i8) -- (i8 |- psi.north);
      \draw (i6 |- iso2.south) -- (o6);
      \draw (i7 |- iso2.south) to node {\tiny $\monL_\bullet$} (i7 |- psi.north);
      \draw (o7 |- psi.south) -- (o7);
      \draw (o8 |- psi.south) -- (o8);
    \end{pic} \\
    & =
    \begin{pic}[auto]
      \node (i0) at (0,0) {$\dom$};
      \node (i1) at (1,0) {$\monL_\bullet^n$};
      \node (i2) at (2,0) {};
      \node (i3) at (3,0) {};
      \node (i4) at (4,0) {$\monL_\bullet$};
      \node (i5) at (5,0) {$\funI$};
      \node (i6) at (6,0) {$\cod$};
      \node (i7) at (7,0) {};
      \node (i8) at (8,0) {$\funT_\bullet$};
      \node[2cell,minimum width=1.1cm] (cc) at (2.5,0) {$\cc$};
      \node[iso2cell,minimum width=2.1cm] (iso1) at (4,-0.7) {};
      \node[iso2cell,minimum width=2.1cm] (iso2) at (7,-0.7) {};
      \node (o0) at (0,-1.4) {$\dom$};
      \node (o1) at (1,-1.4) {$\monL_\bullet^n$};
      \node (o2) at (2,-1.4) {$\funI$};
      \node (o6) at (6,-1.4) {$\cod$};
      \node (o7) at (7,-1.4) {$\funT_\bullet$};
      \node (o8) at (8,-1.4) {$\monL_\bullet$};
      \draw (i0) -- (o0);
      \draw (i1) -- (o1);
      \draw (i2 |- cc.south) -- (o2);
      \draw (i3 |- cc.south) to node {\tiny $\cod$} (i3 |- iso1.north);
      \draw (i4) -- (i4 |- iso1.north);
      \draw (i5) -- (i5 |- iso1.north);
      \draw (i6) -- (i6 |- iso2.north);
      \draw (i8) -- (i8 |- iso2.north);
      \draw (i6 |- iso2.south) -- (o6);
      \draw (i7 |- iso2.south) -- (o7);
      \draw (i8 |- iso2.south) -- (o8);
    \end{pic}
    = \text{ RHS.}
  \end{align*}

  For (ii), we use induction on $n$.
  The base case, $n=0$, has already been covered, so we assume $n>0$.
  Now using Definitions~\ref{def:psin} and~\ref{def:Psi}, and Lemma~\ref{lem:d1d0},
  and induction in the fourth equation, we have --
  \begin{align*}
    \text{LHS } & =
    \begin{pic}[auto]
      \node (i0) at (0,0) {$\funT_\bullet$};
      \node (i1) at (1,0) {};
      \node (i2) at (2,0) {$\funK_{n+1}$};
      \node[2cell,minimum width=2.1cm] (alpha) at (1,-0.7) {$\psin{n+1}$};
      \node[2cell,minimum width=1.1cm] (d1) at (0.5,-1.5) {$d_1$};
      \node[2cell,minimum width=0cm] (d0) at (0,-2.3) {$\dzeroK$};
      \node[2cell,minimum width=1.1cm] (psi) at (1.5,-2.3) {$\Psi_\bullet$};
      \node (o0) at (0,-3) {$\funK_{n-1}$};
      \node (o1) at (1,-3) {$\funT_\bullet$};
      \node (o2) at (2,-3) {$\monL_\bullet$};
      \draw (i0) -- (i0 |- alpha.north);
      \draw (i2) -- (i2 |- alpha.north);
      \draw (i0 |- alpha.south) to node {\tiny $\funK_{n+1}$} (i0 |- d1.north);
      \draw (i2 |- alpha.south) to node {\tiny $\funT_\bullet$} (i2 |- psi.north);
      \draw (i0 |- d1.south) to node {\tiny $\funK_n$} (i0 |- d0.north);
      \draw (i1 |- d1.south) to node {\tiny $\monL_\bullet$} (i1 |- psi.north);
      \draw (o0 |- d0.south) -- (o0);
      \draw (o1 |- psi.south) -- (o1);
      \draw (o2 |- psi.south) -- (o2);
    \end{pic}
    =
    \begin{pic}[auto]
      \node (i0) at (0,0) {$\funT_\bullet$};
      \node (i1) at (1,0) {};
      \node (i2) at (2,0) {$\funK_{n+1}$};
      \node[2cell,minimum width=2.1cm] (alpha) at (1,-0.7) {$\psin{n+1}$};
      \node[2cell,minimum width=0cm] (d0) at (0,-1.5) {$\dzeroK$};
      \node[2cell,minimum width=1.1cm] (d1) at (0.5,-2.3) {$d_1$};
      \node[2cell,minimum width=1.1cm] (psi) at (1.5,-3.1) {$\Psi_\bullet$};
      \node (o0) at (0,-3.8) {$\funK_{n-1}$};
      \node (o1) at (1,-3.8) {$\funT_\bullet$};
      \node (o2) at (2,-3.8) {$\monL_\bullet$};
      \draw (i0) -- (i0 |- alpha.north);
      \draw (i2) -- (i2 |- alpha.north);
      \draw (i0 |- alpha.south) to node {\tiny $\funK_{n+1}$} (i0 |- d0.north);
      \draw (i2 |- alpha.south) to node {\tiny $\funT_\bullet$} (i2 |- psi.north);
      \draw (i0 |- d0.south) to node {\tiny $\funK_n$} (i0 |- d1.north);
      \draw (i0 |- d1.south) -- (o0);
      \draw (i1 |- d1.south) to node {\tiny $\monL_\bullet$} (i1 |- psi.north);
      \draw (o1 |- psi.south) -- (o1);
      \draw (o2 |- psi.south) -- (o2);
    \end{pic}
    =
    \begin{pic}[auto]
      \node (i0) at (0,0) {$\funT_\bullet$};
      \node (i1) at (1,0) {};
      \node (i2) at (2,0) {$\funK_{n+1}$};
      \node[2cell,minimum width=0cm] (d0) at (2,-0.7) {$\dzeroK$};
      \node[2cell,minimum width=2.1cm] (alpha) at (1,-1.5) {$\psin{n}$};
      \node[2cell,minimum width=1.1cm] (d1) at (0.5,-2.3) {$d_1$};
      \node[2cell,minimum width=1.1cm] (psi) at (1.5,-3.1) {$\Psi_\bullet$};
      \node (o0) at (0,-3.8) {$\funK_{n-1}$};
      \node (o1) at (1,-3.8) {$\funT_\bullet$};
      \node (o2) at (2,-3.8) {$\monL_\bullet$};
      \draw (i0) -- (i0 |- alpha.north);
      \draw (i2) -- (i2 |- d0.north);
      \draw (i2 |- d0.south) to node {\tiny $\funK_n$} (i2 |- alpha.north);
      \draw (i0 |- alpha.south) to node {\tiny $\funK_{n}$} (i0 |- d1.north);
      \draw (i2 |- alpha.south) to node {\tiny $\funT_\bullet$} (i2 |- psi.north);
      \draw (i0 |- d1.south) -- (o0);
      \draw (i1 |- d1.south) to node {\tiny $\monL_\bullet$} (i1 |- psi.north);
      \draw (o1 |- psi.south) -- (o1);
      \draw (o2 |- psi.south) -- (o2);
    \end{pic}
    =
    \begin{pic}[auto]
      \node (i0) at (0,0) {$\funT_\bullet$};
      \node (i1) at (1,0) {};
      \node (i2) at (2,0) {$\funK_{n+1}$};
      \node[2cell,minimum width=0cm] (d0) at (2,-0.7) {$\dzeroK$};
      \node[2cell,minimum width=1.1cm] (d1) at (1.5,-1.5) {$d_1$};
      \node[2cell,minimum width=1.1cm] (alpha) at (0.5,-2.3) {$\psin{n-1}$};
      \node (o0) at (0,-3) {$\funK_{n-1}$};
      \node (o1) at (1,-3) {$\funT_\bullet$};
      \node (o2) at (2,-3) {$\monL_\bullet$};
      \draw (i0) -- (i0 |- alpha.north);
      \draw (i2) -- (i2 |- d0.north);
      \draw (i2 |- d0.south) to node {\tiny $\funK_{n}$} (i2 |- d1.north);
      \draw (i1 |- d1.south) to node {\tiny $\funK_{n-1}$} (i1 |- alpha.north);
      \draw (i2 |- d1.south) -- (o2);
      \draw (i0 |- alpha.south) -- (o0);
      \draw (i1 |- alpha.south) -- (o1);
    \end{pic}   \\
    & =
    \begin{pic}[auto]
      \node (i0) at (0,0) {$\funT_\bullet$};
      \node (i1) at (1,0) {};
      \node (i2) at (2,0) {$\funK_{n+1}$};
      \node[2cell,minimum width=1.1cm] (d1) at (1.5,-0.7) {$d_1$};
      \node[2cell,minimum width=0cm] (d0) at (1,-1.5) {$\dzeroK$};
      \node[2cell,minimum width=1.1cm] (alpha) at (0.5,-2.3) {$\psin{n-1}$};
      \node (o0) at (0,-3) {$\funK_{n-1}$};
      \node (o1) at (1,-3) {$\funT_\bullet$};
      \node (o2) at (2,-3) {$\monL_\bullet$};
      \draw (i0) -- (i0 |- alpha.north);
      \draw (i2) -- (i2 |- d1.north);
      \draw (i1 |- d1.south) to node {\tiny $\funK_{n}$} (i1 |- d0.north);
      \draw (i2 |- d1.south) -- (o2);
      \draw (i1 |- d0.south) to node {\tiny $\funK_{n-1}$} (i1 |- alpha.north);
      \draw (i0 |- alpha.south) -- (o0);
      \draw (i1 |- alpha.south) -- (o1);
    \end{pic}
    =
    \begin{pic}[auto]
      \node (i0) at (0,0) {$\funT_\bullet$};
      \node (i1) at (1,0) {};
      \node (i2) at (2,0) {$\funK_{n+1}$};
      \node[2cell,minimum width=1.1cm] (d1) at (1.5,-0.7) {$d_1$};
      \node[2cell,minimum width=1.1cm] (alpha) at (0.5,-1.5) {$\psin{n}$};
      \node[2cell,minimum width=0cm] (d0) at (0,-2.3) {$\dzeroK$};
      \node (o0) at (0,-3) {$\funK_{n-1}$};
      \node (o1) at (1,-3) {$\funT_\bullet$};
      \node (o2) at (2,-3) {$\monL_\bullet$};
      \draw (i0) -- (i0 |- alpha.north);
      \draw (i2) -- (i2 |- d1.north);
      \draw (i1 |- d1.south) to node {\tiny $\funK_{n}$} (i1 |- alpha.north);
      \draw (i2 |- d1.south) -- (o2);
      \draw (i0 |- alpha.south) to node {\tiny $\funK_{n}$} (i0 |- d0.north);
      \draw (i1 |- alpha.south) -- (o1);
      \draw (i0 |- d0.south) -- (o0);
    \end{pic}
    = \text{ RHS.}
  \end{align*}
\end{proof}
\begin{lemma}\label{lem:Psii}
  \[
    \begin{pic}[auto]
      \node (i1) at (1,0) {$\funT_\bullet$};
      \node[2cell,minimum width=0cm] (i) at (0,0) {$i$};
      \node[2cell,minimum width=1.1cm] (psi) at (0.5,-0.8) {$\Psi_\bullet$};
      \node (o0) at (0,-1.5) {$\funT_\bullet$};
      \node (o1) at (1,-1.5) {$\monL_\bullet$};
      \draw (i1) -- (i1 |- psi.north);
      \draw (o0 |- i.south) to node {\tiny $\monL_\bullet$} (o0 |- psi.north);
      \draw (o0 |- psi.south) -- (o0);
      \draw (o1 |- psi.south) -- (o1);
    \end{pic}
    =
    \begin{pic}[auto]
      \node(i0) at (0,0) {};
      \node (i1) at (1,0) {$\funT_\bullet$};
      \node[2cell,minimum width=0cm] (i) at (0,0) {$i$};
      \node[2cell,minimum width=0cm] (d1) at (0,-0.8) {$d_1$};
      \node[2cell,minimum width=1.1cm] (psi) at (0.5,-1.6) {$\Psi_\bullet$};
      \node (o0) at (0,-2.3) {$\funT_\bullet$};
      \node (o1) at (1,-2.3) {$\monL_\bullet$};
      \draw (i0 |- i.south) to node {\tiny $\funK_1$} (i0 |- d1.north);
      \draw (i1) -- (i1 |- psi.north);
      \draw (i0 |- d1.south) to node {\tiny $\monL_\bullet$} (i0 |- psi.north);
      \draw (o0 |- psi.south) -- (o0);
      \draw (o1 |- psi.south) -- (o1);
    \end{pic}
    =
    \begin{pic}[auto]
      \node(i0) at (0,0) {};
      \node[2cell,minimum width=0cm] (i) at (0,0) {$i$};
      \node (i1) at (1,0) {$\funT_\bullet$};
      \node[2cell,minimum width=1.1cm] (psi) at (0.5,-0.8) {$\psin{1}^{-1}$};
      \node[2cell,minimum width=0cm] (d1) at (1,-1.6) {$d_1$};
      \node (o0) at (0,-2.3) {$\funT_\bullet$};
      \node (o1) at (1,-2.3) {$\monL_\bullet$};
      \draw (i0 |- i.south) to node {\tiny $\funK_1$} (i0 |- psi.north);
      \draw (i1) -- (i1 |- psi.north);
      \draw (o0 |- psi.south) -- (o0);
      \draw (i1 |- psi.south) to node {\tiny $\funK_1$} (i1 |- d1.north);
      \draw (o1 |- d1.south) -- (o1);
    \end{pic}
    =
    \begin{pic}[auto]
      \node (i0) at (0,0) {$\funT_\bullet$};
      \node (i1) at (1,0) {};
      \node[2cell,minimum width=0cm] (i) at (1,0) {$i$};
      \node[2cell,minimum width=0cm] (d1) at (1,-0.8) {$d_1$};
      \node (o0) at (0,-1.5) {$\funT_\bullet$};
      \node (o1) at (1,-1.5) {$\monL_\bullet$};
      \draw (i0) -- (o0);
      \draw (i1 |- i.south) to node {\tiny $\funK_1$} (i1 |- d1.north);
      \draw (o1 |- d1.south) -- (o1);
    \end{pic}
    =
    \begin{pic}[auto]
      \node (i0) at (0,0) {$\funT_\bullet$};
      \node[2cell,minimum width=0cm] (i) at (1,0) {$i$};
      \node (o0) at (0,-1) {$\funT_\bullet$};
      \node (o1) at (1,-1) {$\monL_\bullet$};
      \draw (i0) -- (o0);
      \draw (o1 |- i.south) -- (o1);
    \end{pic}
  \]
\end{lemma}
\begin{proof}
  The first and last equations use Lemma~\ref{lem:cid0d1},
  and the other two use the definition of $\Psi_\bullet$
  and Lemma~\ref{lem:psiic}.
\end{proof}

\begin{lemma}\label{lem:Psic}
  \[
    \begin{pic}[auto]
      \node (i0) at (0,0) {$\monL_\bullet$};
      \node (i1) at (1,0) {$\monL_\bullet$};
      \node (i2) at (2,0) {$\funT_\bullet$};
      \node[2cell,minimum width=1.1cm] (c) at (0.5,-0.7) {$c$};
      \node[2cell,minimum width=1.1cm] (psi) at (1.5,-1.5) {$\Psi_\bullet$};
      \node (o1) at (1,-2.2) {$\funT_\bullet$};
      \node (o2) at (2,-2.2) {$\monL_\bullet$};
      \draw (i0) -- (i0 |- c.north);
      \draw (i1) -- (i1 |- c.north);
      \draw (i2) -- (i2 |- psi.north);
      \draw (i1 |- c.south) to node {\tiny $\monL_\bullet$} (i1 |- psi.north);
      \draw (o1 |- psi.south) -- (o1);
      \draw (o2 |- psi.south) -- (o2);
    \end{pic}
    =
    \begin{pic}[auto]
      \node (i0) at (0,0) {$\monL_\bullet$};
      \node (i1) at (1,0) {$\monL_\bullet$};
      \node (i2) at (2,0) {$\funT_\bullet$};
      \node[2cell,minimum width=1.1cm] (psi1) at (1.5,-0.7) {$\Psi_\bullet$};
      \node[2cell,minimum width=1.1cm] (psi2) at (0.5,-1.5) {$\Psi_\bullet$};
      \node[2cell,minimum width=1.1cm] (c) at (1.5,-2.3) {$c$};
      \node (o0) at (0,-3) {$\funT_\bullet$};
      \node (o1) at (1,-3) {$\monL_\bullet$};
      \draw (i0) -- (i0 |- psi2.north);
      \draw (i1) -- (i1 |- psi1.north);
      \draw (i2) -- (i2 |- psi1.north);
      \draw (i1 |- psi1.south) to node {\tiny $\funT_\bullet$} (i1 |- psi2.north);
      \draw (i2 |- psi1.south) to node {\tiny $\monL_\bullet$} (i2 |- c.north);
      \draw (i1 |- psi2.south) to node {\tiny $\monL_\bullet$} (i1 |- c.north);
      \draw (o0 |- psi2.south) -- (o0);
      \draw (o1 |- c.south) -- (o1);
    \end{pic}
  \]
\end{lemma}
\begin{proof}
  We must prove the equation when composed with $\dom$ and with $\cod$ on the left.
  For $\cod$ this is immediate, since everything is over $\cod$.
  For $\dom$, after composing top and bottom with appropriate identity 2-cells,
  we calculate as follows.
  Here, equations (1) and (7) use the fact that $\dom d_1$ is an identity;
  equations (2) and (6) use Definition~\ref{def:Psi};
  and equations (3)-(5) use Lemmas~\ref{lem:psiic},~\ref{lem:cid0d1}~(b)
  and~\ref{lem:psiPsi} respectively.

  \begin{align*}
    \text{LHS } & \stackrel{(1)}{=}
    \begin{pic}[auto]
      \node (i0) at (0,0) {$\dom$};
      \node (i1) at (1,0) {$\funK_2$};
      \node (i2) at (2,0) {$\funT_\bullet$};
      \node[2cell,minimum width=0cm] (c) at (1,-0.7) {$c$};
      \node[2cell,minimum width=0cm] (d1) at (1,-1.5) {$d_1$};
      \node[2cell,minimum width=1.1cm] (psi) at (1.5,-2.3) {$\Psi_\bullet$};
      \node (o0) at (0,-3) {$\dom$};
      \node (o1) at (1,-3) {$\funT_\bullet$};
      \node (o2) at (2,-3) {$\monL_\bullet$};
      \draw (i0) -- (o0);
      \draw (i1) -- (i1 |- c.north);
      \draw (i2) -- (i2 |- psi.north);
      \draw (i1 |- c.south) to node {\tiny $\funK_1$} (i1 |- d1.north);
      \draw (i1 |- d1.south) to node {\tiny $\monL_\bullet$} (i1 |- psi.north);
      \draw (o1 |- psi.south) -- (o1);
      \draw (o2 |- psi.south) -- (o2);
    \end{pic}
    \stackrel{(2)}{=}
    \begin{pic}[auto]
      \node (i0) at (0,0) {$\dom$};
      \node (i1) at (1,0) {$\funK_2$};
      \node (i2) at (2,0) {$\funT_\bullet$};
      \node[2cell,minimum width=0cm] (c) at (1,-0.7) {$c$};
      \node[2cell,minimum width=1.1cm] (alpha) at (1.5,-1.5) {$\psin{1}^{-1}$};
      \node[2cell,minimum width=0cm] (d1) at (2,-2.3) {$d_1$};
      \node (o0) at (0,-3) {$\dom$};
      \node (o1) at (1,-3) {$\funT_\bullet$};
      \node (o2) at (2,-3) {$\monL_\bullet$};
      \draw (i0) -- (o0);
      \draw (i1) -- (i1 |- c.north);
      \draw (i2) -- (i2 |- alpha.north);
      \draw (i1 |- c.south) to node {\tiny $\funK_1$} (i1 |- alpha.north);
      \draw (o1 |- alpha.south) -- (o1);
      \draw (i2 |- alpha.south) to node {\tiny $\funK_1$} (i2 |- d1.north);
      \draw (o2 |- d1.south) -- (o2);
    \end{pic}
    \stackrel{(3)}{=}
    \begin{pic}[auto]
      \node (i0) at (0,0) {$\dom$};
      \node (i1) at (1,0) {$\funK_2$};
      \node (i2) at (2,0) {$\funT_\bullet$};
      \node[2cell,minimum width=1.1cm] (alpha) at (1.5,-0.7) {$\psin{2}^{-1}$};
      \node[2cell,minimum width=0cm] (c) at (2,-1.5) {$c$};
      \node[2cell,minimum width=0cm] (d1) at (2,-2.3) {$d_1$};
      \node (o0) at (0,-3) {$\dom$};
      \node (o1) at (1,-3) {$\funT_\bullet$};
      \node (o2) at (2,-3) {$\monL_\bullet$};
      \draw (i0) -- (o0);
      \draw (i1) -- (i1 |- alpha.north);
      \draw (i2) -- (i2 |- alpha.north);
      \draw (o1 |- alpha.south) -- (o1);
      \draw (i2 |- alpha.south) to node {\tiny $\funK_2$} (i2 |- c.north);
      \draw (i2 |- c.south) to node {\tiny $\funK_1$} (i2 |- d1.north);
      \draw (o2 |- d1.south) -- (o2);
    \end{pic}
    \stackrel{(4)}{=}
    \begin{pic}[auto]
      \node (i0) at (0,0) {$\dom$};
      \node (i1) at (1,0) {$\funK_2$};
      \node (i2) at (2,0) {$\funT_\bullet$};
      \node (i3) at (3,0) {};
      \node[2cell,minimum width=1.1cm] (alpha) at (1.5,-0.7) {$\psin{2}^{-1}$};
      \node[2cell,minimum width=1.1cm] (d11) at (2.5,-1.5) {$d_1$};
      \node[2cell,minimum width=0cm] (d12) at (2,-2.3) {$d_1$};
      \node[2cell,minimum width=1.1cm] (c) at (2.5,-3.1) {$c$};
      \node (o0) at (0,-3.8) {$\dom$};
      \node (o1) at (1,-3.8) {$\funT_\bullet$};
      \node (o2) at (2,-3.8) {$\monL_\bullet$};
      \draw (i0) -- (o0);
      \draw (i1) -- (i1 |- alpha.north);
      \draw (i2) -- (i2 |- alpha.north);
      \draw (o1 |- alpha.south) -- (o1);
      \draw (i2 |- alpha.south) to node {\tiny $\funK_2$} (i2 |- d11.north);
      \draw (i2 |- d11.south) to node {\tiny $\funK_1$} (i2 |- d12.north);
      \draw (i3 |- d11.south) to node {\tiny $\monL_\bullet$} (i3 |- c.north);
      \draw (i2 |- d12.south) to node {\tiny $\monL_\bullet$} (i2 |- c.north);
      \draw (o2 |- c.south) -- (o2);
    \end{pic}                              \\
    & \stackrel{(5)}{=}
    \begin{pic}[auto]
      \node (i0) at (0,0) {$\dom$};
      \node (i1) at (1,0) {$\funK_2$};
      \node (i2) at (2,0) {};
      \node (i3) at (3,0) {$\funT_\bullet$};
      \node[2cell,minimum width=1.1cm] (d11) at (1.5,-0.7) {$d_1$};
      \node[2cell,minimum width=1.1cm] (psi1) at (2.5,-1.5) {$\Psi_\bullet$};
      \node[2cell,minimum width=1.1cm] (alpha) at (1.5,-2.3) {$\psin{1}^{-1}$};
      \node[2cell,minimum width=0cm] (d12) at (2,-3.1) {$d_1$};
      \node[2cell,minimum width=1.1cm] (c) at (2.5,-3.9) {$c$};
      \node (o0) at (0,-4.6) {$\dom$};
      \node (o1) at (1,-4.6) {$\funT_\bullet$};
      \node (o2) at (2,-4.6) {$\monL_\bullet$};
      \draw (i0) -- (o0);
      \draw (i1) -- (i1 |- d11.north);
      \draw (i3) -- (i3 |- psi1.north);
      \draw (i1 |- d11.south) to node {\tiny $\funK_1$} (i1 |- alpha.north);
      \draw (i2 |- d11.south) to node {\tiny $\monL_\bullet$} (i2 |- psi1.north);
      \draw (i2 |- psi1.south) to node {\tiny $\funT_\bullet$} (i2 |- alpha.north);
      \draw (i3 |- psi1.south) to node {\tiny $\monL_\bullet$} (i3 |- c.north);
      \draw (o1 |- alpha.south) -- (o1);
      \draw (i2 |- alpha.south) to node {\tiny $\funK_1$} (i2 |- d12.north);
      \draw (i2 |- d12.south) to node {\tiny $\monL_\bullet$} (i2 |- c.north);
      \draw (o2 |- c.south) -- (o2);
    \end{pic}
    \stackrel{(6)}{=}
    \begin{pic}[auto]
      \node (i0) at (0,0) {$\dom$};
      \node (i1) at (1,0) {$\funK_2$};
      \node (i2) at (2,0) {};
      \node (i3) at (3,0) {$\funT_\bullet$};
      \node[2cell,minimum width=1.1cm] (d11) at (1.5,-0.7) {$d_1$};
      \node[2cell,minimum width=0cm] (d12) at (1,-1.5) {$d_1$};
      \node[2cell,minimum width=1.1cm] (psi1) at (2.5,-1.5) {$\Psi_\bullet$};
      \node[2cell,minimum width=1.1cm] (psi2) at (1.5,-2.3) {$\Psi_\bullet$};
      \node[2cell,minimum width=1.1cm] (c) at (2.5,-3.1) {$c$};
      \node (o0) at (0,-3.8) {$\dom$};
      \node (o1) at (1,-3.8) {$\funT_\bullet$};
      \node (o2) at (2,-3.8) {$\monL_\bullet$};
      \draw (i0) -- (o0);
      \draw (i1) -- (i1 |- d11.north);
      \draw (i3) -- (i3 |- psi1.north);
      \draw (i1 |- d11.south) to node {\tiny $\funK_1$} (i1 |- d12.north);
      \draw (i2 |- d11.south) to node {\tiny $\monL_\bullet$} (i2 |- psi1.north);
      \draw (i1 |- d12.south) to node {\tiny $\monL_\bullet$} (i1 |- psi2.north);
      \draw (i2 |- psi1.south) to node {\tiny $\funT_\bullet$} (i2 |- psi2.north);
      \draw (i3 |- psi1.south) to node {\tiny $\monL_\bullet$} (i3 |- c.north);
      \draw (o1 |- psi2.south) -- (o1);
      \draw (i2 |- psi2.south) to node {\tiny $\monL_\bullet$} (i2 |- c.north);
      \draw (o2 |- c.south) -- (o2);
    \end{pic}
    \stackrel{(7)}{=} \text{ RHS.}
  \end{align*}
\end{proof}

\begin{proposition}\label{prop:2mon2func}
  Let $\funT_\bullet$ be an indexed bundle 2-endofunctor for $\catC$,
  and let $\monL_\bullet$ be the monad from Definition~\ref{def:lbullet}.
  Then $\Psi_\bullet$ is a 2-transition from
  $(\arrC, \monL_\bullet)$ to itself along $\funT_\bullet$.
\end{proposition}
\begin{proof}
We have already shown that $\Psi_\bullet$ is 2-natural,
and the two equations needed to make a 2-transition are those of
Lemmas~\ref{lem:Psii} and~\ref{lem:Psic}.
\end{proof}
\begin{proposition}\label{prop:psalglifting}
  The 2-functor $\funT_\bullet$ lifts to an endofunctor
  $\widehat{\funT}_\bullet$ on the category $(\arrC)^{\monL_\bullet}$
  of $\monL_\bullet$-pseudoalgebras $(E,c,\zeta,\theta)$.
  $\widehat{\funT}_\bullet$ preserves normality ($\zeta=1$),
  and also the property that $\downstairs{c}=1_B$.
\end{proposition}
\begin{proof}
  Given Proposition~\ref{prop:2mon2func}, the pseudoalgebra lifting and
  preservation of normality now follow from Proposition~\ref{prop:lifting},
  while the final part follows from the fact that both $\funT_\bullet$
  and $\Psi_\bullet$ are over $\cod$.
\end{proof}
\begin{theorem}\label{thm:main}
  Let $\catC$ be a representable 2-category,
  and $\funT_\bullet$ an indexed bundle 2-endofunctor for $\catC$.
  Then $\funT_\bullet$ preserves pseudofibrations, fibrations, pseudo-opfibrations
  and opfibrations in $\catC$.
\end{theorem}
\begin{proof}
  For (pseudo-)opfibrations, we combine Proposition~\ref{prop:psalglifting} with
  Proposition~\ref{prop:Lbulletpsalg}.

  Working in $\catC^{co}$ shows that the same development holds for the
  monad $\monR_\bullet$. This gives the result for (pseudo)fibrations.
\end{proof}

\section{Conclusions}\label{sec:conc}
As mentioned in the introduction, the starting point for this work was rather specific.
Recent topos-theoretic approaches to quantum foundations can be understood \cite{FRV:Born}
as constructing ``spectral bundles'' in the category of locales.
There are two technically distinct approaches,
``presheaf'' and ``copresheaf'',
exemplified by \cite{DoeringIsham:WhatThingTTFP} and \cite{HeunenLandsmanSpitters:ToposAQT}
respectively.
In the presheaf approach, the spectral bundle is a local homeomorphism and hence an opfibration:
it has fibre maps covariant with respect to specialization in the base.
In the copresheaf approach, the spectral bundle is fibrewise compact regular
(i.e. it corresponds to a compact regular locale in the topos of sheaves over the base)
and hence a fibration:
it has contravariant fibre maps.

We should like to apply constructions to these bundles,
for example the valuation locale construction in order to gain access to probabilistic
features of quantum physics,
and it is natural to wish these constructions to work fibrewise on bundles.
This naturally calls for constructions on locales that are geometric in the sense of
being preserved under pullback, since fibres are pullbacks.

In general, geometric constructions will not preserve the bundle properties of
being a local homeomorphism or of fibrewise compact regularity.
However, it would still be useful to know that they preserve fibrations and opfibrations,
and that is what this paper proves.
Thus it seems that contextual physics
(in which everything is considered to be fibred over a base space of contexts)
would fall naturally into two kinds, opfibrational and fibrational,
both closed under geometric constructions.
The choice is determined by the choice for the spectral bundles,
opfibrational for the presheaf approach, fibrational for copresheaves.
(There is still debate over which is better.
Our own opinion is that Gelfand-Naimark duality suggests that the spectral bundle
should be fibrewise compact regular, thus leading to fibrations.)

Along the way the work took on other important ideas.
The first was that ``geometric'' just means indexed, in the sense of indexed categories,
and could alternatively be formulated using the codomain bifibration.
The paper \cite{TownsendVickers:strengthofpowerlocale} provides a way to guarantee
the coherence needed for this, applicable to standard examples of geometric
constructions on locales.

The new definition of geometric as indexed also becomes applicable to 2-categories much more general
than $\mathbf{Loc}$,
and this prompted a big generalization of the present work.
In the more general setting it seemed very natural to use Street's characterization
of fibrations and opfibrations in terms of pseudoalgebras,
and lifting functors to algebra categories.

Throughout we have used tangle diagrams as a 2-dimensional calculus for 2-categories
-- in fact, we have even been able to use them in the 3-category of 2-categories,
by separating out two levels of 2-dimensionality.
We have found them more conclusive than the usual diagrams.

We have proved our results for a rather particular choice of laxities.
We work in 2-categories, with 2-monads (hence all strict),
but pseudoalgebras -- to match Street's criterion.
Also, the geometric endofunctors $\funT_\bullet$ are 2-functors.
It may be that other combinations might be useful, or other indexed categories.

\section*{Acknowledgement}
The research reported in this paper was conducted as part of the project
``Applications of geometric logic to topos approaches to quantum theory'',
project EP/G046298/1 funded by the UK Engineering and Physical Sciences
Research Council.

We are grateful to Anders Kock and Peter Johnstone for discussions that
clarified the relation of geometricity with indexed categories;
and to Jean B\'enabou and Ross Street for clearing up our initial
misapprehension of the role of normality (pseudofibrations v. fibrations)
in \cite{Street}.

%
%
%
%

{\small
\bibliography{GCPFBiblio,MyBiblio}
\bibliographystyle{agsm}
\def\topsep{0pt}
\def\parsep{0pt plus 5pt minus 1pt}
\def\itemsep{-0.5ex}
}
\end{document}